\documentclass{amsart}

\usepackage[foot]{amsaddr}

\usepackage{bbm}	% for \mathbb and \mathbbm
\usepackage{mathrsfs}	% for \mathscr

\usepackage{amsmath, amsthm, amssymb}	% basic AMS-packages
\usepackage{enumitem}	% for adjustments on enumerate

\usepackage{tikz}	% for tikz-pictures
\usepackage[linktocpage,hidelinks]{hyperref}
\usepackage{mathtools}

\setlist[enumerate]{label=(\roman*)}

\newtheorem{theorem}{Theorem}[section]
\newtheorem{lemma}[theorem]{Lemma}
\newtheorem{corollary}[theorem]{Corollary}

\theoremstyle{definition}
\newtheorem{definition}[theorem]{Definition}
\newtheorem{example}[theorem]{Example}

\theoremstyle{remark}
\newtheorem{remark}[theorem]{Remark}

\numberwithin{equation}{section}

\newcommand{\R}{\mathbb{R}}	% reelle Zahlen
\newcommand{\N}{\mathbb{N}}	% natuerliche Zahlen
\newcommand{\Q}{\mathbb{Q}}	% rationale Zahlen
\newcommand{\sB}{\mathscr{B}}	% Borelmenge
\newcommand{\sE}{\mathscr{E}}	% sigma-Algebra des Zustandsraums E

\newcommand{\cC}{\mathcal{C}}	% Raum aller stetigen Funktionen

\renewcommand{\O}{\Omega}	% Grundraum aller Elementarereignisse
\renewcommand{\o}{\omega}	% Elemente des Grundraums
\newcommand{\sA}{\mathscr{A}}	% sigma-Algebra eines W-Raums
\newcommand{\PV}{\mathbb{P}}	% W-Mass
\newcommand{\EV}{\mathbb{E}}	% Symbol fuer den Erwartungswert

\newcommand{\sF}{\mathscr{F}}	% natuerliche Filtration
\newcommand{\sG}{\mathscr{G}}	% allgemeine Filtration
\newcommand{\T}{\Theta}		% Shiftoperator
\newcommand{\D}{\Delta}		% Friedhof, Laplace-Operator
\newcommand{\sq}{\square}       % auxiliary Friedhof
\newcommand{\z}{\zeta}		% Lebenszeit

\renewcommand{\a}{\alpha}	% Stoppoperator
\newcommand{\g}{\gamma}		% Translationsoperator
\newcommand{\G}{\Gamma}		% Zentrierungsoperator
\renewcommand{\t}{\tau}		% allg. Stoppzeit
\newcommand{\h}{H}		% Ersteintrittszeit

\newcommand{\sD}{\mathscr{D}}	% Definitionsbereich des Erzeugers

\newcommand{\cG}{\mathcal{G}}	% metrischer Graph, als Tupel von 
\newcommand{\cV}{\mathcal{V}}	% -- Knotenmenge
\newcommand{\cE}{\mathcal{E}}	% -- externe Kanten
\newcommand{\cR}{\rho}		% Kantenlänge

\newcommand{\bs}{\backslash}	% Mengendifferenz
\newcommand{\comp}{\complement}	% Komplement

\renewcommand{\b}{\beta}
\renewcommand{\l}{\lambda}
\newcommand{\e}{\varepsilon}	% epsilon > 0
\newcommand{\p}{\pi}		% allg. Projektionsabbildung

\newcommand{\tB}{\widetilde{B}}
\newcommand{\tM}{\widetilde{M}}

\newcommand{\tX}{\widetilde{X}{}}

\newcommand{\tE}{\widetilde{E}}

\newcommand{\tsF}{\widetilde{\sF}}
\newcommand{\bsF}{\overline{\sF}}
\newcommand{\bbsF}{\skew{0.0}\overline{\bsF}{}}

\newcommand{\hW}{\hat{W}}
\newcommand{\hQ}{\hat{Q}}
\newcommand{\hL}{\hat{L}}
\newcommand{\hT}{\hat{\T}}
\newcommand{\hG}{\hat{\G}}
\newcommand{\hg}{\hat{\g}}

\newcommand{\pP}{{}^{+}\!{P}}
\newcommand{\pQ}{{}^{+}\!{Q}}
\newcommand{\nP}{{}^{0}\!{P}}
\newcommand{\nQ}{{}^{0}\!{Q}}

\newcommand{\1}{\mathbbm{1}}		% Indikatorfunktion
\newcommand{\id}{\operatorname{id}} 	% Identitaetsfunktion

\newcommand{\abs}[1]{\left\lvert #1 \right\rvert}  % Absolutbetrag
\newcommand{\bigabs}[1]{\big\lvert #1 \big\rvert}  % Absolutbetrag
\newcommand\restr[2]{{		  	% restriction of a function
  \left.\kern-\nulldelimiterspace % automatically resize the bar with \right
  #1 				  % the function
  \vphantom{\big|}		  % pretend it's a little taller at normal size
  \right|_{#2}    		  % this is the delimiter
  }}

\newcommand{\sgn}{\operatorname{sgn}}
\newcommand{\ran}{\operatorname{ran}}

\newcommand{\processX}{$X = \big( \O, \sG, (\sG_t)_{t \geq 0}, (X_t)_{t \geq 0}, (\T_t)_{t \geq 0}, (\PV_x)_{x \in E} \big)$}

\newcommand{\KPS}{Kostrykin, Potthoff and Schrader}
\newcommand{\IM}{It\^{o}\textendash{}McKean}

\newcommand{\FW}{Feller\textendash{}Wentzell}

\newcommand{\sT}{\mathscr{T}}	
\newcommand{\cB}{\mathcal{B}}	
\newcommand{\textdef}{\textit}

\begin{document}

\title[Brownian Motions on Star Graphs]{Brownian Motions on Star Graphs with Non-Local Boundary Conditions}

\author[Florian Werner]{Florian Werner}
\address{Institut f\"ur Mathematik, Universit\"at Mannheim, 68131 Mannheim, Germany}
%\curraddr{}
\email{fwerner@math.uni-mannheim.de}
%\thanks{}

%\author{Author Two}
%\address{}
%\email{}
%\thanks{}

%    General info

\date{\today}%{December 4, 2017}% and, in revised form, June 22, 2001.}



%% for Springer template:
%\renewcommand{\qedhere}{\qed}
%\newcommand{\sqed}{\qed}
\newcommand{\sqed}{\relax}

\begin{abstract}
  Brownian motions on star graphs in the sense of \IM,
that is, Walsh processes admitting a generalized boundary behavior including stickiness and jumps and having an angular distribution with finite support, 
are examined.
Their generators are identified as Laplace operators on the graph subject to non-local \FW\ boundary conditions. 
A pathwise description is achieved for every admissible boundary condition: 
For finite jump measures, a construction of \KPS\ in the continuous setting is expanded via a technique of successive killings and revivals;
for infinite jump measures, the pathwise solution of \IM\ for the half line is analyzed and extended to the star graph.
These processes can then be used as main building blocks for Brownian motions on general metric graphs with non-local boundary conditions.
\end{abstract}

\maketitle

\section{Introduction}

The goal of the present paper is the pathwise construction of all Brownian motions on a star graph~$\cG$,
that is, to construct a Feller process such that its generator $A =  \frac{1}{2} \D$ satisfies the non-local \FW\ boundary condition 
  \begin{align*}
   \forall f \in \cC_0^2(\cG): \quad p_1 f(0) - \sum_{e \in \cE} p^{e}_2 f_e'(0) + \frac{p_3}{2} f''(0) - \int_{\cG \bs \{0\}} \big( f(g) - f(0) \big) \, p_4(dg) = 0
  \end{align*}
for given constants $p_1 \geq 0$, $p^{e}_2 \geq 0$ for each edge $e \in \cE$, $p_3 \geq 0$ and a measure $p_4$ on the punctured star graph $\cG \bs \{0\}$, with
  \begin{align*}
    p_1 + \sum_{e \in \cE} p^{e}_2 + p_3 + \int_{\cG \bs \{0\}} \big( 1 - e^{-x} \big) \, p_4 \big( d(e,x) \big) = 1.
  \end{align*}   
We now illustrate the underlying definitions,
for rigorous definitions the reader may consult section~\ref{sec:Definitions}:

A metric graph $\cG$ is a mathematical description of a set of locally one-dimensional structures, 
edges $e \in \cE$, which are ``glued together'' at vertices $v \in \cV$ by the graph's combinatorial structure,
and every edge $e \in \cE$ is isomorphic to a finite interval or  half line of length $\cR_e \in (0, +\infty]$.
In the case of a star graph, $\cG$ consists only of one vertex (which will be named $0$) and a set of edges with each one being isomorphic to $[0, +\infty)$,
that is, the star graph is then represented by the set
 \begin{align*}
  \cG = \{0\} \cup \bigcup_{e \in \cE} \big( \{e\} \times [0, +\infty) \big),
 \end{align*}
where the finite endpoint $(e,0)$ of each edge $e \in \cE$ is identified with the vertex $0$.
The canonical metric on $\cG$ is then defined by the length of the shortest possible path connecting two points on $\cG$:
inside the edges, it coincides with the Euclidean metric, while on differing edges, it is the sum of the points' distances to~$0$.
We will only consider star graphs with finite sets of edges.

A Brownian motion on a star graph $\cG$ (more generally, on any metric graph $\cG$) is defined to be a right continuous, strong Markov process on~$\cG$
which behaves on every edge like the standard one-dimensional Brownian motion, more accurately:
If a Brownian motion $X$ on the graph~$\cG$ is started inside some edge $\{e\} \times (0, +\infty)$, then the process
$X$, stopped at leaving its initial edge, must be equivalent to the one-dimensional Brownian motion, stopped when leaving the interval $(0, +\infty)$. 
In this sense, Brownian motions on star graphs are both a generalization and a restriction of classical Walsh processes: 
They may feature other boundary behavior at $0$ than just skew effects, such as stickiness or jumps. On the other hand, the skew measure may only
assume finitely many values, due to the finiteness of the set of edges.

The context of star graphs generalizes the class of Brownian motions on half lines, which has been studied extensively in the past:
Presumably, it started with first path considerations by Kac~\cite{Kac51} and Feller~\cite{Feller54} 
and Feller's and Wentzell's analytic examinations of semigroups in \cite{Feller52}, \cite{Wentzell56} and \cite{Wentzell59}.
Dynkin~\cite{Dynkin56} and Hunt~\cite{Hunt56} provided the tools for a rigorous probabilistic study,
and  L\'evy's~\cite{Levy65} and Trotter's~\cite{Trotter58} studies on the fine structure of the paths of the Brownian motions and their local times
made it possible for It\^{o} and McKean to give
the complete, pathwise description of all Brownian motions on~$\R_+$ in~\cite{ItoMcKean63};
 for a more detailed historical overview, we would like to refer the reader to~\cite[Section 2]{ItoMcKean63} and to~\cite{Peskir15}.
 %The interval setting has been further examined by Weber in \cite{Weber94}, Favini et al.\ in \cite{Favini00}, and Xiao and Liang in \cite{XiaoLiang08}.
 
Star graphs serve as the main building blocks in the study and construction of general metric graphs.
Recently, there is a growing interest in metric graphs, networks and quantum graphs, and stochastic processes thereon.
They arise in many areas of physics, chemistry and engineering applications, for an elaborate survey the reader may consult~\cite{Kuchment02}
and Kuchment's introductory article~\cite{Kuchment04}.
A collection of recent developments is found in the proceedings~\cite{AoGaiA08} and Mugnolo's monograph~\cite{Mugnolo14}.
The research of continuous processes on graph-like structures seems to be started by Baxter and Chacon in~\cite{BaxterChacon84},
who introduced the notion of diffusions on graphs and transferred some classical one-dimensional results to this setting.
Since then, a wide variety of results and techniques evolved: 
Freidlin and Wentzell investigated an averaging principle for processes on graphs in~\cite{FreidlinWentzell93},
which was further developed by Barret and von~Renesse with the help of Dirichlet methods in~\cite{BarretvRenesse2014}.
Processes on special tree structures have been examined by Dean and Jansons in~\cite{DeanJansons93} via excursion theory and by Krebs in~\cite{Krebs95} via Dirichlet forms.
 %Mugnolo and Romanelli \cite{Mugnolo07} special behaviors dictated by applications on graphs
With the help of graphs, Walsh~\cite{Walsh78} and Eisenbaum and Kaspi~\cite{Eisenbaum96} studied and extended classical one-dimensional results like local time properties.
Particular Brownian motions on graphs have been constructed and studied
 by Barlow, Pitman and Yor in~\cite{BarlowPitmanYor89} via semigroup considerations,
 by Enriquez and Kifer in~\cite{EnriquezKifer01} as weak limits of Markov chains,
 and by Georgakopoulos and Kolesko in~\cite{Georgakopoulos14} as weak limits of graph approximations.
In~\cite{Lejay03}, Lejay develops simulation methods for diffusions on graphs, which can also be applied in the Brownian context.
Further results for continuous Brownian motions on star graphs have been researched by Najnudel in~\cite{Najnudel07} and Papanicolaou et al.\ in~\cite{Papanicolaou12}.
Fitzsimmons and Kuter conducted potential theoretic investigations in the star graph setting
in~\cite{Jehring09} and~\cite{FitzsimmonsKuterWalsh}, and extended their findings to general metric graphs in \cite{FitzsimmonsKuterGraph}.

Kostrykin, Potthoff and Schrader achieved the classification and pathwise construction of 
all Brownian motions on a metric graph which are continuous (up to their lifetime) by giving a complete description 
of all continuous Brownian motions on star graphs in \cite{KPS_Walsh12A}, \cite{KPS_Walsh12B}, and then gluing them together in~\cite{KPS12}.
%i.e., all Brownian motions that only have one jump, leading to the absorbing cemetery point $\D$. 
Their works mark the starting point of this article, in which we weaken the condition of continuity to right continuity,
which allows non-local effects to take place at the boundary.
By extending the findings and the construction approaches of the above-mentioned works by Kostrykin, Potthoff and Schrader,
and of \IM's extensive analysis of the half-line case in \cite{ItoMcKean63},
we will obtain the classification and a complete pathwise construction for all right continuous Brownian motions on any star graph.

\subsection{Classification of Brownian Motions}
We give a short overview over the possible behavior any Brownian motion may feature on a star graph (more generally, on a metric graph).
By its very definition, the behavior of the process  is already fixed inside the edges,
where it must run like the standard one-dimensional Brownian motion. Therefore, the ``non-Brownian'' effects can only take place 
at the vertices of the graph and still must respect (strongly) Markovian ``characteristics''. Thus, it is feasible to classify a Brownian motion by its local behavior, 
which is reflected in its generator:

As mentioned above, the classical case of a ``metric graph'' with only one vertex and one edge---that is the half line $\R_+$---is completely understood (see~\cite{ItoMcKean63}).
Here, the generator $A$ of a Brownian motion 
 is a contraction of $\frac{1}{2} \, \D$, with $\D$ being the Laplacian on $\R_+$. 
Its domain is then uniquely characterized by a set of
constants $p_1 \geq 0$, $p_2 \geq 0$, $p_3 \geq 0$ and a measure~$p_4$ on~$(0, \infty)$, normalized by
  \begin{align*}
    p_1 + p_2 + p_3 + \int_{(0,\infty)} \big( 1 \wedge x \big) \, p_4 (dx) = 1,
  \end{align*}
 which constitute the following non-local \FW\ boundary condition: % of the generator:
 %\begin{equation}
   \begin{align} 
    \sD(A) = \Big\{ & f \in \cC_0^2(\R_+): \nonumber  \\
                    & p_1 \, f(0) - p_2 \, f'(0+) + \frac{p_3}{2} f''(0+) - \int_{(0,\infty)} \big( f(x) - f(0) \big) \, p_4(dx) = 0 \Big\}.  \label{eq:intro:generator domain halfline}
  \end{align}
% \end{equation}
  
This result is easily extended to the case of a general metric graph $\cG$. Just like in the case of the half line, the generator of a Brownian motion reads $A = \frac{1}{2} \D$,
with $\D$ now being the Laplacian on $\cG$.
For every vertex $v \in \cV$ there exist constants
 $p^v_1 \geq 0$, $p^{v,e}_2 \geq 0$ for each $e \in \cE(v)$, $p^v_3 \geq 0$ and a measure $p^v_4$ on $\cG \bs \{v\}$ with
  \begin{align*} 
    p^v_1 + \sum_{e \in \cE(v)} p^{v,e}_2 + p^v_3 + \int \big( 1 - e^{-d(v,g)} \big) \, p^v_4 \big( dg \big) = 1,
  \end{align*}
 such that the domain of $A$ satisfies
 % \begin{equation} 
  \begin{align}
   \sD(A) & \subseteq 
       \Big\{ f \in \cC^2_0(\cG) : \forall v \in \cV: \nonumber \\ 
   & \qquad           p^v_1 \, f(v) - \sum_{e \in \cE(v)} p^{v,e}_2 \, f_e'(v) + \frac{p^v_3}{2} f''(v) - \int \big( f(g) - f(v) \big) \, p^v_4(dg) = 0 \Big\},  \label{eq:intro:generator domain}
  \end{align}   
  %\end{equation}
where $\cE(v)$ is the set of edges incident with a vertex $v$, and $f_e'(v)$ is the directional derivative of $f$ at $v$ along the edge $e$.\footnote{For a star graph,
the set of vertices is just $\cV = \{0\}$, and $\cE(0) = \cE$ in this case.}

These results can be derived through various techniques: Classical proofs such as in~\cite{Wentzell56} and~\cite{Feller57} 
are based on the analysis of the underlying semigroup, 
which then were extended giving special attention on non-local boundaries in~\cite{Mandl68} and~\cite{Langer71}.
Other approaches are possible by analytic examinations of the resolvent in~\cite{Rogers83} or of the Dirichlet form such as in~\cite{Kant09} and~\cite{Fukushima14},
or by probabilistic methods via Dynkin's formulas like in~\cite{Knight81} and~\cite{ItoMcKean63}, 
or by the excursion theory of~\cite{Ito72}.
As our goal is a pathwise construction, we will be more interested in a method which obtains the generator via a probabilistic method rather than by analytic means: 
Dynkin's formula gives access to the generator directly through the local exit behavior of the process. It states that, under certain conditions, the generator $A$
of a strong Markov process $X$ on a state space~$E$ can be computed by
  \begin{align}\label{eq:Dynkins formula (generator)}
   A f(x) 
   & = \lim_{n \rightarrow \infty} \frac{\EV_v \big( f\big(X(\t_{\e_n})\big) \big) - f(x)}{\EV_x(\t_{\e_n})}, \quad f \in \sD(A), \ x \in E,
  \end{align}
with $(\e_n, n \in \N)$ being a sequence of positive numbers converging to $0$ and $\t_{\e_n}$ being the first exit time of $X$ from the closed ball $\overline{\cB_x(\e_n)}$.

 \begin{figure}[tb] 
   \centering
   \includegraphics[width=0.9\textwidth,keepaspectratio]{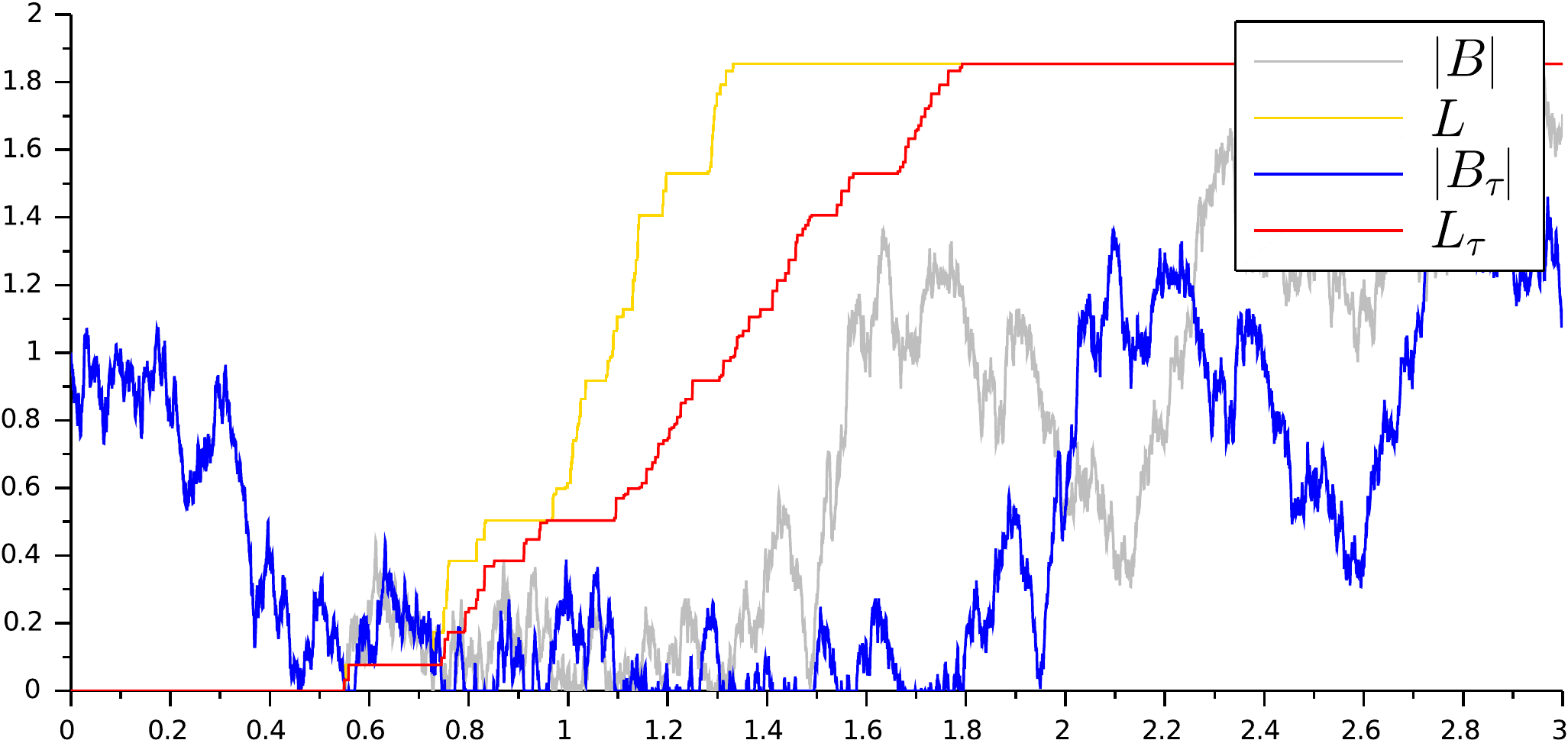}
   \caption[The sticky Brownian motion on $\R_+$]
           {The sticky Brownian motion $\big( |B_{\t(t)}|, t \geq 0 \big)$ on $\R_+$ with its local time $(L_{\t(t)}, t \geq 0)$.} 
            \label{fig:intro:BB HL continuous}
 \end{figure}

Surprisingly, the components of the ``generator data'' given in equation~\eqref{eq:intro:generator domain}
 \begin{align} \label{eq:intro:generator data}
  \big( p^v_1, (p^{v,e}_2)_{e \in \cE(v)}, p^v_3, p^v_4 \big)_{v \in \cV}
 \end{align}
have, for the most part, easy probabilistic interpretations.
 We briefly explain their effects for Brownian motions on the half line $\R_+$, where their set \eqref{eq:intro:generator data}
 of defining boundary weights reduces to $(p_1, p_2, p_3, p_4)$ of equation~\eqref{eq:intro:generator domain halfline}:
 If $B = (B_t, t \geq 0)$ is the Brownian motion on $\R$, then the reflecting Brownian motion $\abs{B} = {(\abs{B_t}, t \geq 0)}$ is 
 a Brownian motion on $\R_+$ which is characterized by its boundary set $(p_1, p_2, p_3, p_4) = (0,1,0,0)$. 
 If instead we consider the ``absorbed'' process $(B_{t \wedge \h_0}, t \geq 0)$ which results from stopping $B$ at the time $\h_0 := \inf \{ B_t = 0 \}$ of $B$ hitting $0$
 for the first time,
 it turns out that this is a Brownian motion on $\R_+$ with $(p_1, p_2, p_3, p_4) = (0,0,1,0)$.
 On the other hand, the boundary set $(p_1, p_2, p_3, p_4) = (1,0,0,0)$ is implemented by the ``Dirichlet'' process $B^D$,
  \begin{align*}
   B^D_t :=
    \begin{cases}
     B_t, & t < \h_0, \\
     \D,  & t \geq \h_0,
    \end{cases}
  \end{align*}
 constructed by killing $B$ at $\h_0$
 (this is not a Brownian motion in the sense of our definition, as Markov processes will always be assumed to be normal in this work).\footnote{We are using
   the conventional symbol $\D$ for both the cemetery point of a Markov process and the Laplace operator.
   Due to the different contexts, there should be no danger of confusion.}
 
 Thus, $p_1$, $p_2$, $p_3$ can be interpreted as the ``weights'' governing the \textdef{killing}, \textdef{reflection} and \textdef{stickiness} at the origin. 
 These effects are especially illuminated when examining the following ``mixed'' cases, as surveyed in~\cite{KPS10}:
 The ``quasi absorbed case'' $(p_1, p_2, p_3, p_4) = ( {\neq}0, 0, {\neq}0, 0)$ can be realized by stopping the Brownian motion~$B$ at the origin for an exponentially distributed random time, independent of $B$, and then killing it.
 The ``elastic case'' $(p_1, p_2, p_3, p_4) = ( {\neq}0, {\neq}0, 0, 0)$ is obtained by killing the reflecting Brownian motion $\abs{B}$ 
  when its local time at the origin exceeds some exponentially distributed random time, independent of $\abs{B}$.
 Finally, the ``sticky case'' $(p_1, p_2, p_3, p_4) = (0, {\neq}0, {\neq}0, 0)$ is achieved by ``slowing down'' the reflecting Brownian motion~$\abs{B}$ at the origin:
  With $(L_t, t \geq 0)$ being its local time at the origin, define the function $\t^{-1} \colon t \mapsto t + \frac{p_3}{p_2} \, L_t$.
  Then the ``sticky'' boundary condition is realized by the time changed Brownian motion $\big( |B_{\t(t)}|, t \geq 0 \big)$, see figure~\ref{fig:intro:BB HL continuous}.
 The complete ``local'' case $(p_1, p_2, p_3, p_4) = ( {\neq}0, {\neq}0, {\neq}0, 0)$ is a mixture of the sticky and the elastic case: It is achieved by killing the sticky Brownian motion
  $\big( |B_{\t(t)}|, t \geq 0 \big)$ once its local time $(L_{\t(t)}, t \geq 0)$ at the origin exceeds some exponentially distributed, independent random time.
 
 The measure $p_4$ now introduces jumps of the resulting process from the origin to points other than the absorbing cemetery point $\D$.
 If $p_4$ is finite, then this \textdef{jump measure} can be implemented just like the jumps of a compound Poisson process: 
 Starting with the Brownian motion realizing the local boundary condition $(p_1, p_2, p_3, 0)$, we restart this process---if it has not been killed 
 already---whenever its local time at the origin
 exceeds some independent, exponentially distributed random time with rate proportional to~$p_4((0,\infty))$, at some point chosen independently by the probability measure 
 $\frac{p_4}{p_4((0,\infty))}$, see figure~\ref{fig:intro:BB HL complete}. 
 In the case of an infinite measure~$p_4$, the description of the complete process is not as easy: As the finite case already suggests,
 the resulting process will be a Brownian motion which implements the local boundary conditions and 
 jumps out of the origin like a subordinator with L\'evy measure~$p_4$, run on the time axis of the local time.
 A detailed construction of such paths will be given later.
 
 \begin{figure}[tb] 
   \centering
   \includegraphics[width=0.95\textwidth,keepaspectratio]{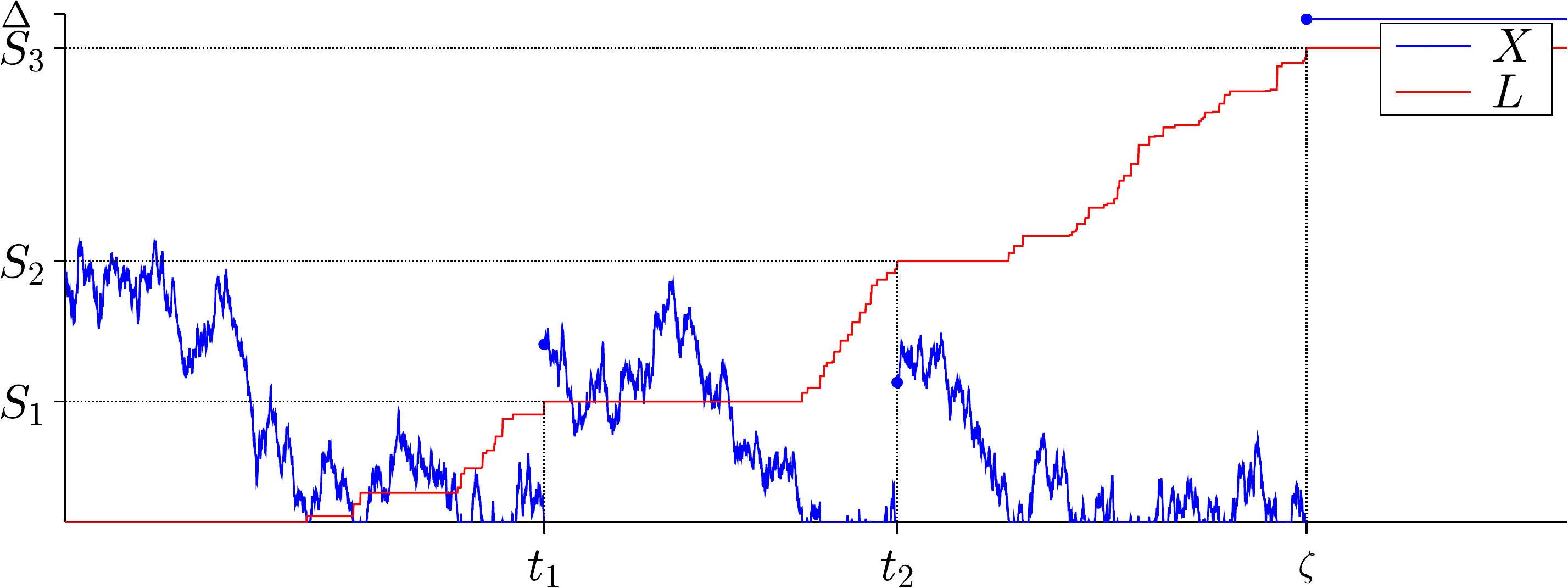}
   \caption[Implementation of jumps for Brownian motions on $\R_+$]
           {Implementation of jumps for Brownian motions on~$\R_+$: 
            Starting with the sticky Brownian motion,
            restart the process whenever its local time exceeds some level $S_n$ at a point chosen by $p_4 + p_1 \, \e_{\D}$,
            resulting a Brownian motion $X$ with boundary weights $(p_1, p_2, p_3, p_4)$.} \label{fig:intro:BB HL complete}
 \end{figure}
 
 These results can be transferred directly to the case of a metric graph $\cG$, where the boundary weights
  $\big( p^v_1, (p^{v,e}_2)_{e \in \cE(v)}, p^v_3, p^v_4 \big)$ govern the local behavior at a vertex $v \in \cV$.
 The only additional effect which arises here is that the process can usually leave the vertex $v$ on more than one edge. 
 Thus, the reflection weight $p^v_2$ is split up into partial weights $p^{v,e}_2$, $e \in \cE(v)$.
 For any excursion which exits $v$ continuously, the 
 starting edge of this excursion is then chosen independently by the distribution $\big( p^{v,e}_2 / p^v_2, e \in \cE(v) \big)$,
 with $p^v_2 := \sum_{e \in \cE(v)} p^{v,e}_2$.
 
 Accepting these rather illustrative descriptions for the moment, it is clear that in absence of the jumping measure $p_4$, the Brownian motion may be realized by a process which
 is continuous up to its lifetime. On the other hand, the case $p_4 \neq 0$ can only be achieved by a discontinuous process.
 
\subsection{Construction Approach}
As already mentioned, the boundary conditions on the edges can be implemented via path transformations of an easy prototype process like the reflecting Brownian motion
or a suitable Walsh process:
The killing parameter is introduced by killing with respect to the pseudo inverse of the local time, which turns out to be a terminal time,
or equivalently, by killing with respect to a multiplicative functional.
Stickiness can be implemented by the time change relative to the local time, which is an additive functional.
These transformations are classical and well understood, and a complete construction of this type was already obtained in \cite{KPS_Walsh12A} and \cite{KPS_Walsh12B}.
However, the implementation of jumps seems to be a non-standard problem, which has not been considered in our context yet.

For finite jump measures, we will use the technique of ``killing and reviving'' a (strong) Markov process, which proceeds as follows:
We construct the concatenation of a sequence of Markov processes $(X^n, n \in \N)$ to form a new Markov process that behaves like 
$X^1$ until this process dies, afterwards is ``revived'' as $X^2$ at some point chosen by a probability kernel which takes 
``Markovian information of $X^1$ until its death'' into account, then behaves like $X^2$ until it dies, and so on. 
Having this general concept of concatenation at our disposal, we now take independent copies of one basis process $X^0$
which dies ``conveniently'', and revive them with appropriate kernels in order to introduce the required jumps.
This technique will be shortly introduced in subsection~\ref{subsec:C_CO:identical iterations}, and then applied to the Brownian 
construction in subsection~\ref{subsec:KPS extension}.

The pathwise solution for infinite jump measures will pose a completely different challenge. 
In this case, just as in the context of a general L\'evy process, the resulting process needs to feature infinitely many ``small'' jumps in arbitrarily small time intervals,
so the jumps will not be arrangeable in time and the process cannot be constructed by the successive concatenation of a countable product of independent subprocesses. 
Here, we will employ a local, ``bare hand'' construction, utilizing the ingenious ideas of It\^{o} and McKean 
which we explain at the beginning of section~\ref{sec:G_IM} in order to put the reader in the position to understand our generalization to the star graph.
The proof of the (strong) Markov property of the resulting process will be highly non-trivial,
and we will only succeed by utilizing Galmarino's results \cite{Galmarino63} on the characterization of stopped $\sigma$-algebras.

\subsection{Applications and Upcoming Work}
As we give a complete description of the pathwise construction for every possible Brownian motion on a star graph, our results can be directly applied
to problems which are centered around the paths of these processes, such as studies on the fine structure or simulation techniques.
On the other hand, we will use Brownian motions on star graphs as prototype components to obtain the construction of Brownian motions on general metric graphs 
in an upcoming work.

\section{Definitions and Fundamental Properties}\label{sec:Definitions}

Before we begin with our constructions, we give a concise introduction to Markov processes, star graphs and Brownian motions on star graphs.
In this section, we summarize the underlying main definitions and collect some of the results.

\subsection{Markov Processes}\label{subsec:Definitions Markov}

We understand a Markov process $X$ on a Radon space $E$ (equipped with a $\sigma$-algebra~$\sE$)  to be defined   
in the canonical sense of the standard works of Dynkin~\cite{Dynkin65}, Blumenthal--Getoor~\cite{BlumenthalGetoor69} and Sharpe~\cite{Sharpe88},
that is, as a sextuple 
 \begin{align*}
  X = \big( \Omega, \sG, (\sG_t, t \geq 0), (X_t, t \geq 0), (\Theta_t, t \geq 0), (\PV_x, x \in E) \big)
 \end{align*}
with the following properties:
$(X_t, t \geq 0)$ is a right continuous, $E$-valued stochastic process on the measurable space $(\Omega, \sG)$,
adapted to the filtration $(\sG_t, t \geq 0)$, and equipped with shift operators $(\Theta_t, t \geq 0)$ on $\Omega$.
$(\PV_x, x \in E)$ is a family of probability measures satisfying $X_0 = x$ $\PV_x$-a.s.\ for all $x \in E$ (normality of the process),
such that for all $t \geq 0$, $B \in \sE$, $x \mapsto \PV_x(X_t \in B)$ is measurable
and the Markov property holds:\footnote{For any $\sigma$-algebra $\sE$, we define $b\sE$, $p\sE$ to be
  the sets of all $\sE$-measurable functions which are bounded, non-negative respectively, 
  as well as $bp\sE := b\sE \cap p\sE$.}\textsuperscript{,}\footnote{For convenience,
  we omit the qualifier ``a.s.''~in equations which contain conditional expectations.}
 \begin{align}\label{eq:Markov property (simple)}
  \forall x \in E, s, t \geq 0, f \in b\sE: \quad \EV_x \big( f(X_{s+t}) \,\big|\, \sG_s \big) = \EV_{X_s} \big( f(X_t) \big).
 \end{align}
 
Every Markov process $X$ has an associated semigroup $(T_t, t \geq 0)$ and resolvent $(U_\a, \a > 0)$, defined for all $f \in p\sE \cup b\sE$ by
 \begin{align*}
  T_t f(x) := \EV_x \big( f(X_t) \big), \quad U_\a f(x) := \EV_x \Big( \int_0^\infty  e^{-\alpha t} \, f(X_t) \, dt \Big), \quad x \in E.
 \end{align*}
 
If $(X_t, t \geq 0)$ is right continuous, the Markov property~\eqref{eq:Markov property (simple)} is equivalent to its Laplace-transformed version,
that is, for all $\a > 0$, $s \geq 0$, $f \in b\cC(E)$: %, $J \in b \sG_s$:
 \begin{align}\label{eq:Markov property (resolvent)} 
  %\forall \a > 0, s \geq 0, f \in b\cC(E), J \in b \sG_s: \quad 
  % \EV_x \Big( \int_0^\infty e^{-\a t} \, f(X_{s+t}) \, dt \cdot J \Big) = \EV_x \big( U_\a f(X_s) \cdot J \big).
  \EV_x \Big( \int_0^\infty e^{-\a t} \, f(X_{s+t}) \, dt \,\Big|\, \sG_s \Big) = U_\a f(X_s).
 \end{align}
 
A Markov process $X$ is said to be strongly Markovian with respect to a filtration $(\sG_t, t \geq 0)$, if for any $(\sG_t, t \geq 0)$-stopping time $\t$,
 \begin{align}\label{eq:strong Markov property (simple)}
  \forall x \in E, s, t \geq 0, f \in b\sE: \quad \EV_x \big( f(X_{t+\t}) \,\big|\, \sG_{\t+} \big) = \EV_{X_s} \big( f(X_\t) \big),
 \end{align}
which then can be lifted to  the universal completion  $\sF$ of $\sigma(X_s, s \geq 0)$ by using the monotone class theorem:
 \begin{align}\label{eq:strong Markov property (strong)}
    \forall x \in E, Y \in b\sF: \quad \EV_x \big( Y \circ \Theta_\tau \, \1_{\{\tau < \infty\}} \,\big|\, \sG_{\tau+} \big) = \EV_{X_\tau} \big( Y \big) \, \1_{\{\tau < \infty\}}.
 \end{align}
 
Given a strong Markov process $X$, its resolvent can be localized at any stopping time $\tau$ with the help of Dynkin's formula \cite[Section 5.1]{Dynkin65}
 \begin{align}\label{eq:Dynkins formula (resolvent)}
   U_\alpha f(x) = \EV_x \Big( \int_0^\tau  e^{-\alpha t} \, f(X_t) \, dt \Big) + \EV_x \big( e^{-\alpha \tau} \, U_\alpha f(X_\tau) \, \1_{\{\tau < \infty\}} \big).
 \end{align}
%and the generator can be localized as follows:
% \begin{align}
%  
% \end{align}

A Markov process $X$ is a Feller process, if its semigroup is $\cC_0$-Feller,\footnote{For a locally compact space~$E$ with countable base,
   $\cC_0(E)$ is the set of all continuous functions which vanish at infinity. The space of all continuous and bounded functions on~$E$ is denoted by~$b\cC(E)$.} that is, if 
 \begin{enumerate}
  \item $T_t \cC_0(E) \subseteq \cC_0(E)$ for all $t \geq 0$, and   \label{itm:Feller (semigroup)}
  \item $\lim_{t \downarrow 0} T_t f(x) = f(x)$ for all $f \in \cC_0(E)$, $x \in E$. \label{itm:Feller (normality)}
 \end{enumerate}
Here, \ref{itm:Feller (normality)} is already implied by the assumed right continuity and normality of any Markov process.
Furthermore, it is well-known (cf.~\cite[Appendix~B]{KPS12}) that \ref{itm:Feller (semigroup)}
can be equivalently replaced by the corresponding condition of the resolvent, that is,
 \begin{align}\label{eq:Feller (resolvent)}
   U_\a \cC_0(E) \subseteq \cC_0(E) \quad \text{for all $\a > 0$}.
 \end{align}  

Every Feller process $X$ is uniquely characterized by its weak $\cC_0$-generator
 \begin{align}
   A \colon \sD(A) \rightarrow \cC_0, \quad A f(x) := \lim_{t \downarrow 0} \frac{T_t f(x) - f(x)}{t},
 \end{align}
with its domain $\sD(A)$ being to set of all $f \in \cC_0$ for which the right-hand limit exists and constitutes a function in $\cC_0$. 

Whenever it is convenient and possible, we will treat a Markov process in the context of right processes (which necessitates the switch to the usual hypotheses)
to ensure that transformations like killing, time change or revival produce a strong Markov process again (cf.~\cite[Chapter~II]{Sharpe88}).

For the most part, however, we will work in the basic setting as described above. 
Then, Galmarino's theorem \cite{Galmarino63} (see also \cite[Chapter~IV, 99--101]{DellacherieMeyerA}, \cite[Theorem~3.2.13]{Knight81})
gives the following characterization of the stopped $\sigma$-algebra $\sF^0_\tau$ 
of the canonical filtration $\sF^0_t = \sigma(X_s, s \leq t)$, $t \geq 0$, of a right continuous stochastic process $(X_t, t \geq 0)$ on $\O$
for an $(\sF^0_t, t \geq 0)$-stopping time $\tau$:
 \begin{align}\label{eq:galmarino}
  \sF^0_\tau = \sigma(X_{t \wedge \tau}, t \geq 0),
 \end{align}
in case there exist stopping operators $(\a_t, t \geq 0)$ on $\O$, satisfying $X_s \circ \a_t = X_{s \wedge t}$ and $\a_s \circ \a_t = \a_{s \wedge t}$ for all $s, t \geq 0$.

At times, we need some basic properties of L\'evy processes. Furthermore, we will make use of centering and translation operators $\G$, $(\g_x, x \in E)$ on $\O$, satisfying
 \begin{align*}
  \forall t \geq 0, x, y \in E: \quad X_t \circ \G = X_t - X_0, \quad X_t \circ \g_x = X_t + x, \quad \g_x \circ \g_y = \g_{x+y}.
 \end{align*}
Then, similar to shift operators $(\T_t, t \geq 0)$, utilizing the spatial homogeneity of a L\'evy process $X$, we have for all $x, y \in E$, $F \in b\sF$:
 \begin{align}\label{eq:centering translation for Levy}
  \EV_x(F \circ \G) = \EV_0(F), \quad \EV_x (F \circ \g_y) = \EV_{x+y}(F).
 \end{align}
Just as in the case of shift operators, there exist natural centering and translation operators in the path-space setting, namely
 \begin{align}\label{eq:natural centering and translation}
   \T_t(\o) := \o(t + \,\cdot\,), \quad \G(\o) := \o - \o(0), \quad \g_x(\o) := \o + x, ~ x \in E.
 \end{align}

\subsection{Star Graphs}

A \textdef{star graph} is a metric graph with only one vertex $0$ and a finite set of (external) edges $\cE$, that is, it is represented by
 \begin{align*}
  \cG = \{0\} \cup \bigcup_{e \in \cE} \big( \{e\} \times [0, +\infty) \big),
 \end{align*}
with the endpoint $(e,0)$ of each edge $e \in \cE$ being identified with the vertex $0$. 
The notion of shortest distances induced by the Euclidean metric on the edges establishes a metric on $\cG$.
Inside $\cG \bs \{0\}$, the topology of the edges equals the Euclidean one-dimensional topology of intervals,
as the open $\e$-balls read
 \begin{align*}
  \forall e \in \cE, x > 0, \e < x: \quad \cB_\e\big( (e,x) \big) = \{e\} \times (x-\e, x+\e),
 \end{align*}
while on the star vertex, the edges are glued together:
 \begin{align*}
  \forall \e > 0:  \quad \cB_\e(0) = \bigcup_{e \in \cE} \{e\} \times [0, \e).
 \end{align*}
In particular, $\cG$ is a Polish space.

Every real-valued function $f$ on a star graph $\cG$ is represented by a collection of functions $(f_e, e \in \cE)$ for $f_e \colon [0, +\infty) \rightarrow \R$,
with ${f_e(x) = f \big( (e,x) \big)}$, $x \in [0, +\infty)$. As the endpoints of the edges are identified, the values at $0$ must coincide:
 \begin{align*}
  \forall e \in \cE: \quad f_e(0) = f \big( (e,0) \big) = f(0).
 \end{align*}

In every small neighborhood of a non-vertex point $g \in \cG \bs \{0\}$, a real valued function~$f$ on~$\cG$ can locally be interpreted as a function on some interval of~$\R$. Thus,
the differentiability of~$f_e$ at~$x$ induces the notion of differentiability of~$f$ at~$g = (e,x) \in \cG \bs \{0\}$. The concept of differentiation at the vertex is as follows:

\begin{definition}
 Let $f \colon \cG \rightarrow \R$ be a function on $\cG$, and $e \in \cE$.
 Then the \textdef{directional derivative of $f$ at $0$ along $e$} is defined by
  \begin{align*}
   f_e'(0) :=  \lim_{x \rightarrow 0, x > 0} f'(e, x).
  \end{align*}
 whenever the right-hand side exists.
\end{definition}

\begin{definition} \label{def:G_MG:C2 functions}
 Let $\cC^{0,2}_0(\cG)$ be the subspace of all functions $f$ in $\cC_0(\cG)$, which are twice continuously differentiable on $\cG \bs \{0\}$, such that 
 for every %$v \in \cV$, 
 $e \in \cE$,
 the limit
  \begin{align*}
   f_e''(0) := \lim_{x \rightarrow 0, x > 0} f_e''(x)
  \end{align*}
 exists, and $f_e''$ vanishes at infinity.
 Let $\cC^2_0(\cG)$ be the subset of those functions~$f$ in~$\cC^{0,2}_0(\cG)$, for which $f''$ extends from $\cG \bs \{0\}$ to a function in~$\cC_0(\cG)$.
\end{definition}

% By definition, a function $f \in \cC^{0,2}_0(\cG)$ lies in $\cC^2_0(\cG)$, if and only if, for every $v \in \cV$,
% the second derivatives at $v$ coincide, that is, if $f_k''(v) = f_e''(v)$ holds for all $k,l \in \cE(v)$,
% and in this case, we will just write $f''(v)$ for this value. 
% If $f \in \cC^2_0(\cG)$, then, for any edge $l \in \cE$, the limits of the first derivatives at its endpoint(s)
% $\lim_{x \downdownarrows 0} f_e'(x)$ (and $\lim_{x \upuparrows \cR_e} f_e'(x)$, if~$l \in \cI$) must exist,
% which can easily be seen by the fundamental theorem of calculus:
%  \begin{align*}
%   f_e'(x) = - \int_x^b f_e''(t) \, dt + f_e'(b), \quad x, b \in (0, \cR_e).
%  \end{align*}
% However, these limits on various edges do not need to coincide at their mutual vertex:
% In general, the first derivate
% $f'$ of $f \in \cC^2_0(\cG)$ does not extend from $\cC_0(\cG \bs \{0\})$ to a function in $\cC_0(\cG)$.

We will mainly be concerned with the following operator on $\cC^2_0(\cG)$:

\begin{definition}
 The \textdef{Laplacian} $\D$ on $\cG$ is defined by 
  \begin{align*}
   \D \colon \cC^2_0(\cG) \rightarrow \cC_0(\cG), ~ f \mapsto \D(f) := f''.
  \end{align*}
\end{definition}

\subsection{Brownian Motions on Star Graphs} \label{subsec:G_BM:def}

%Following the definitions of the half-line and interval cases \ref{def:B_HL:BM}, \ref{def:B_IN:BM}, and 
Extending the definition of~\cite{KPS12} and \cite[Chapter~6]{Knight81} to the discontinuous setting of \cite{ItoMcKean63}, we 
define a Brownian motion on a star graph $\cG$ (more generally, on a metric graph~$\cG$) to be a right continuous, strong Markov process on $\cG$ 
which behaves on every edge like the standard one-dimensional Brownian motion. That is, the local coordinate of such a process, if stopped once it leaves its starting edge,
needs to be equivalent to the Brownian motion on $\R$, stopped when leaving the corresponding interval of the process' initial edge:

\begin{definition} \label{def:G_BM:BM}
 Let \processX\ be a right continuous, strong Markov process on a metric graph $\cG$. 
 $X$ is a \textdef{Brownian motion on~$\cG$}, if
 for all $g = (e,x) \in \cG$, the random time 
  \begin{align*}
   H_X := \inf \big\{ t \geq 0: X_t \notin e^0 \big\}, \quad 
   \text{with } e^0 = \{e\} \times ( 0, \cR_e ),
  \end{align*}
 is a $(\sG_t, t \geq 0)$-stopping time,
 and for all $n \in \N$, $f_1, \ldots, f_n \in b\sB(\cG)$, $t_1, \ldots, t_n \in \R_+$,
  \begin{align*}
   \EV_{(e,x)} \big( f_1(X_{t_1 \wedge H_X}) \cdots f_n(X_{t_n \wedge H_X}) \big)
   = \EV^B_{x} \big( f_1(e, B_{t_1 \wedge H_B}) \cdots f_n(e, B_{t_n \wedge H_B}) \big)
  \end{align*}
 holds, with $B$ being the Brownian motion on $\R$ and $H_B := \inf \big\{ t \geq 0: B_t \notin ( 0, \cR_e ) \big\}$. 
\end{definition}

The technical requirement of the first hitting time $H_X$ of the closed set $\comp e^0$
being a stopping time is always satisfied if we are working in the context of usual hypotheses (cf.~\cite[Sections~10, A.5]{Sharpe88}).
It can also be achieved if we ensure the continuity of the process $X$ until $H_X$ (see~\cite[Theorem~49.5]{Bauer96}),
that is, continuity while the process runs inside any edge. While the latter condition is not implied by the above definition,
it is a desirable property which may be implemented by constructing a Brownian motion on a metric graph
with continuous excursions of a ``standard'' one-dimensional Brownian motion, as done in sections~\ref{sec:KPS extension} and \ref{sec:G_IM}.

% \begin{theorem} \label{theo:G_BM:BM is Feller with gen Laplace}
%  Let $X$ be a Brownian motion on $\cG$ with generator $A$.
%  Then $X$ is a Feller process, uniquely determined by its generator $A = \frac{1}{2} \D$, with $\sD(A) \subseteq \cC^2_0(\cG)$.
% \end{theorem}

We give the main result on the classification of Brownian motions on star graphs:

\begin{theorem}  \label{theo:G_BM:Fellers theorem on star graph}
 Let $X$ be a Brownian motion on star graph $\cG$ with star point $0$.  Then $X$ is a Feller process with generator $A = \frac{1}{2} \D$,
 and there exist constants
 $p_1 \geq 0$, $p^{e}_2 \geq 0$ for each $e \in \cE$, $p_3 \geq 0$ and a measure $p_4$ on $\cG \bs \{0\}$ with
  \begin{align*}
    p_1 + \sum_{e \in \cE} p^{e}_2 + p_3 + \int_{\cG \bs \{0\}} \big( 1 - e^{-x} \big) \, p_4 \big( d(e,x) \big) = 1,
  \end{align*}
 and
  \begin{align}\label{eq:G_BM:infinite jumpmeasure}
    p_4 \big( \cG \bs \{0\} \big) = +\infty, \quad \text{if} \quad  \sum_{e \in \cE} p^{e}_2 + p_3 = 0,
  \end{align}
 such that the domain of $A$ reads
  \begin{align*}
   \sD(A)  = 
       \Big\{ & f \in \cC^2_0(\cG): \\ 
              &       p_1 f(0) - \sum_{e \in \cE} p^{e}_2 f_e'(0) + \frac{p_3}{2} f''(0) - \int_{\cG \bs \{0\}} \big( f(g) - f(0) \big) \, p_4(dg) = 0 \Big\}. 
  \end{align*}
 Furthermore, $X$ is uniquely characterized by this set of normalized constants.
\end{theorem}

The above theorem can be proved similarly to the half-line setting 
with the help of Dynkin's formula for the generator~\eqref{eq:Dynkins formula (generator)}
(cf.~\cite[Lemma~6.2]{Knight81}, \cite[Section~8]{ItoMcKean63}, and \cite[Lemma 3.2]{KPS12}), 
giving special attention to the graph topology. A proof for general metric graphs can be found in \cite[Section~20.3]{Werner16}.
Notice that we use the equivalent normalization $1 - e^{-x}$, instead of the classical $1 \wedge x$, for the jump measure~$p_4$, 
which will turn out to be more suitable in our context.

In general, the approach via Dynkin's formula only gives necessity of the boundary condition. For later use,
we prove the following result which assures sufficiency:
\begin{lemma} \label{lem:G_BM:totality on star graph}
 Let $X$ be a Brownian motion on a star graph $\cG$ with star point $0$, and let 
 $p_1 \geq 0$, $p^{e}_2 \geq 0$ for each $e \in \cE$, $p_3 \geq 0$ and a measure $p_4$ on $\cG \bs \{0\}$ be given with
  \begin{align*}
    p_1 + \sum_{e \in \cE} p^{e}_2 + p_3 + \int_{\cG \bs \{0\}} \big( 1 - e^{-x} \big) \, p_4 \big( d(e,x) \big) > 0,
  \end{align*}
 such that the generator of $X$ is $A = \frac{1}{2} \D$ and its domain satisfies $\sD(A) \subseteq \sD$, with
  \begin{align*}
   \sD  := 
       \Big\{ & f \in \cC^2_0(\cG) : \\ 
   &           p_1 f(0) - \sum_{e \in \cE} p^{e}_2 f_e'(0) + \frac{p_3}{2} f''(0) - \int_{\cG \bs \{0\}} \big( f(g) - f(0) \big) \, p_4(dg) = 0 \Big\}.
  \end{align*}
 Then $\sD(A) = \sD$.
\end{lemma}
\begin{proof}
 %As in the proof for the interval case, we will employ Lemma~\ref{lem:A_FP:totality of generator}:
 For $\a > 0$, consider $\frac{1}{2} \D = \a f$ with $f \in \sD$. It suffices to show that $f \equiv 0$ is the only possible solution
 (see, e.g., \cite[Lemma~1.1, Theorem~1.2]{Dynkin65}).
 
 The function $f$ solves the differential equation $\frac{1}{2} \D = \a f$ on every edge, so it must be of the form
  \begin{align*}
    f(e,x) = c^e_1 \, e^{-\sqrt{2 \a} x} + c^e_2 \, e^{\sqrt{2 \a} x}, \quad e \in \cE, ~ x \geq 0,
  \end{align*}
 for some $c^e_1, c^e_2 \in \R$, for each $e \in \cE$.
 Since $f$ needs to vanish at infinity (because $f \in \cC_0(\cG)$), it is $c^e_2 = 0$ for all $e \in \cE$.
 But then, in order to be continuous at the star vertex,
 all the $c^e_1$ need to coincide. Therefore, setting $c^e_1 = c$ for all $e \in \cE$ results in
  \begin{align*}
    f(e,x) = c \, e^{-\sqrt{2 \a} x}, \quad e \in \cE, ~ x \geq 0.
  \end{align*} 
 As $f \in \sD(A) \subseteq \sD$, the boundary condition for $f$ now yields 
  \begin{align*}
   c \, \Big( p_1 + \sqrt{2 \a} \, \sum_{e \in \cE} p^e_2 + \a \, p_3 + \int_{\cG \bs \{0\}} \big( 1 - e^{-\sqrt{2 \a} x} \big) \, p_4 \big( d(e,x) \big)  \Big)
   = 0,
  \end{align*}
 which is only possible for $c = 0$, because all of the summands in the parentheses are non-negative, but must add up to a positive number.
 
 Thus $\frac{1}{2} \D f = \a f$, $f \in \sD$, is only solved by $f = 0$, completing the proof.\sqed
\end{proof}

We are ready for the pathwise constructions. For convenience,
we have collected basic information on some prototype Brownian motions on the half line, the Walsh process, and easy, but non-standard results on Brownian
local time in Appendix~\ref{app:Preliminaries}.

\section{Construction for Finite Jump Measures}\label{sec:KPS extension}

 The upcoming, extensive construction in section~\ref{sec:G_IM} is only necessary for jump measures~$p_4$ which admit $p_4 \big( \cG \bs \{0\} \big) = +\infty$.
 If~$p_4$ is a finite measure on $\cG \bs \{0\}$, there is a much simpler way to construct a Brownian motion~$X$ on the star graph~$\cG$ with generator domain
  \begin{align*}
    \sD(A) =  \Big\{   & f \in \cC^2_0(\cG) : \\ 
                       & p_1 f(0) - \sum_{e \in \cE} p^{e}_2 f_e'(0) + \frac{p_3}{2} f''(0) - \int_{\cG \bs \{0\}} \big( f(g) - f(0) \big) \, p_4(dg) = 0 \Big\},
  \end{align*}
 which we briefly cover now. To this end, we prepare a technique for introducing isolated jumps:

\subsection{The Technique of Successive Revivals}\label{subsec:C_CO:identical iterations}

We will shortly review the technique of concatenation:
Given a sequence of right processes $X^n$ on~$\O^n$ with lifetimes~$\z^n$ on disjoint state spaces~$E^n$, $n \in \N$, we can glue them together to
a process~$X$ on~$\bigcup_n E^n$ by setting on $\Omega := \prod_{n \in \N} \Omega^n$, 
for~$\omega := (\omega^n, n \in \N) \in \Omega$, $t \geq 0$, 
 \begin{align*}
  X_t(\omega) :=
    \begin{cases}
      X^1_t(\omega^1),				 & t < \zeta^{1}(\omega^1), \\
      X^2_{t - \zeta^{1}(\omega^1)}(\omega^2),		 & \zeta^{1}(\omega^1) \leq t < \zeta^{1}(\omega^1) + \zeta^{2}(\omega^2), \\
%      X^3_{t - \zeta^{1}(\omega^1) - \zeta^{2}(\omega^2)}(\omega^3),		 & \zeta^{1}(\omega^1) + \zeta^{2}(\omega^2) \leq t < \zeta^{1}(\omega^1) + \cdots + \zeta^{3}(\omega^3), \\
      X^3_{t - \left(\zeta^{1}(\omega^1) + \zeta^{2}(\omega^2)\right)}(\omega^3),		 & \zeta^{1}(\omega^1) + \zeta^{2}(\omega^2) \leq t < \zeta^{1}(\omega^1) + \cdots + \zeta^{3}(\omega^3), \\
      ~ \vdots					 & ~ \vdots \\
      \Delta,  & t \geq \sum_{n \in \N}\zeta^{n}(\omega^n).
    \end{cases}
 \end{align*}
The revival point after each death is chosen by ``memoryless'' kernels~$K^n$ which take the ``information'' of $X^n$ until its death into account,
so-called \textdef{transfer kernels}~$K^n$ from~$X^n$ to~$(X^{n+1}, E^{n+1})$, which are introduced in~\cite[Section~14]{Sharpe88}.
We will not define them here, because deterministic transfer kernels of the form 
 \begin{align*}
  K^n(x, dy) = q(dy), \quad x \in E^n,
 \end{align*}
for a probability measure $q$ on $E^{n+1}$ will turn out to be sufficient for our applications. Given such kernels $K^n$, $n \in \N$, the distribution of the
concatenated process~$X$ is given by the initial measures
 \begin{align*}
   & \PV_x(d\omega^1, \ldots, d\omega^{n-1}, d\omega^n, d\omega^{n+1}, \ldots) \\
   & := \delta_{[\Delta^1]}(d\omega^1) \cdots \delta_{[\Delta^{n-1}]}(d\omega^{n-1}) \,
      \PV^n_x(d\omega^n) \, K^n(\omega^n, dx^{n+1}) \,
      \PV^{n+1}_{x^{n+1}}(d\omega^{n+1}) \cdots
 \end{align*}
for any $x \in E^n$, $n \in \N$.

In~\cite[Section~2.3]{WernerConcat} we extended the results of~\cite[Section~14]{Sharpe88} on the concatenation of right processes and achieved
the following result:
\begin{theorem} \label{theo:concatenation countable} 
 Let $(X^n, n \in \N)$ be a sequence of right processes on disjoint spaces $(E^n, n \in \N)$, such that the topological union
 $E := \bigcup_{n \in \N} E^n$ is a Radon space, and let a transfer kernel $K^n$ from $X^n$ to $(X^{n+1}, E^{n+1})$ be given for each $n \in \N$.
 Then the concatenation $X$ of the processes $(X^n, n \in \N)$ via the transfer kernels $(K^n, n \in \N)$
 is a right process on $E$. With $R^n := \inf \{ t \geq 0: X_t \in E^{n+1} \}$, for all $n \in \N$, $x \in \bigcup_{j=1}^n E^j$, $f \in b\sE^{n+1}$,
  \begin{align*}
   \EV_x \big( f(X_{R^n}) \, \1_{\{R^n < \infty \}} \, \big| \, \sF_{R^n-} \big) = K^n f \circ \pi^n \, \1_{\{R^n < \infty \}}. 
  \end{align*}
\end{theorem}

Now, let a single right process~$X^0$ on~$E$ be given, together with a transfer kernel~$K^0$ from~$X^0$ to~$(X^0, E)$.
Define the sequences
 \begin{align}\label{eq:pasting processes}
  X^n := \{n\} \times X^{0}, \quad K^n := \delta_{n+1} \otimes K^{0}, \quad n \in \N.
 \end{align}
Then the $X^n$, $n \in \N$, are right processes on disjoint spaces $E^n := \{n\} \times E$, 
$K^n$ are transfer kernels from $X^n$ to $(X^{n+1}, E^{n+1})$, 
and the concatenation $X$ is a right process on~$\tE := \bigcup_n E^n = \N \times E$.
Now discard the first coordinate by applying the projection $\pi \colon \N \times E \rightarrow E$ to~$X$.
By checking the consistency condition for state-space transformations (see~\cite[Section~13]{Sharpe88}, \cite[Section~3]{WernerConcat}),
we see that $X^q := \pi(X)$ is a right process on~$E$. It is the \textdef{instant revival process}
constructed of independent copies of $X^0$, which are revived after each death by the \textdef{revival distribution} $q$.

\subsection{Construction of Brownian motions on Star Graphs (Finite Measure)}\label{subsec:KPS extension}

 Let $p_1 \geq 0$, $p^e_2 \geq 0$ for each $e \in \cE$, $p_3 \geq 0$, and $p_4$ be a finite measure on $\cG \bs \{0\}$, normalized by
  \begin{align*}
    p_1 + \sum_{e \in \cE} p^{e}_2 + p_3 + \int_{\cG \bs \{0\}} \big( 1 - e^{-x} \big) \, p_4 \big( d(e,x) \big) = 1, \quad \text{with } p_2 := \sum_{e \in \cE} p^{e}_2.
  \end{align*}

 If $p_2 > 0$, following the construction of \cite{KPS_Walsh12A} and \cite{KPS_Walsh12B}, start with the Walsh process~$W$ on~$\cG$ with reflection weights $(q^e_2 = p^e_2/p_2, e \in \cE)$
 and local time $(L_t, t \geq 0)$ at the vertex~$0$.
 Then implement the stickiness parameter~$p_3$ by ``slowing down''~$W$ at the vertex via the canonical approach of time change, as given in~\cite[Section 2]{KPS_Walsh12B}:
 For $\gamma := p_3/p_2$, introduce the new time scale~$\tau$ by defining its inverse by
  \begin{align*}
   \tau^{-1} \colon \R_+ \rightarrow \R_+, \quad t \mapsto t + \gamma \, L_t,
  \end{align*}
 and consider the sticky Walsh process
  \begin{align*}
   W^s_t := W_{\tau(t)}, \quad t \geq 0,
  \end{align*}
 with its new local time $L^s_t = L_{\tau(t)}$, $t \geq 0$, as seen in \cite[Equation (2.22)]{KPS_Walsh12B}.
 Next, following \cite[Section 3]{KPS_Walsh12B}, introduce an exponentially distributed random variable~$S$ with rate $\beta := \frac{p_1 + p_4( \cG \bs \{0\})}{p_2}$,
 independent of $W^s$,
 and kill $W^s$ when its local time exceeds~$S$, that is, at the random time
  \begin{align*}
   \zeta_{\beta, \gamma} := \inf \{ t \geq 0: L^s_t > S \},
  \end{align*}
 to obtain the process
  \begin{align*}
   W^g_t :=
    \begin{cases}
     W^s_t, & t < \zeta_{\beta, \gamma}, \\
     \D, & t \geq \zeta_{\beta, \gamma}.
    \end{cases}
  \end{align*}
 Then \cite[Theorem 3.7]{KPS_Walsh12B} shows that $W^g$ is a Brownian motion on $\cG$ with generator 
  \begin{align*}
    \sD(A^g) =  \Big\{ & f \in \cC^2_0(\cG) : \\ 
                       & \big( p_1 + p_4( \cG \bs \{0\} ) \big) f(0) - \sum_{e \in \cE} p^{e}_2 f_e'(0) + \frac{p_3}{2} f''(0) = 0 \Big\}.
  \end{align*}
  
 Now adjoin an absorbing, isolated point $\sq$ to the state space $\cG$ and let $X^q$ be the 
 instant revival process resulting from $W^g$, as explained in subsection~\ref{subsec:C_CO:identical iterations}, where the original process~$W^g$ is revived, whenever it dies, 
 with the transfer kernel
  \begin{align*}
   \forall g \in \cG: \quad K^0(g, \,\cdot\,) := q, \quad \text{with} \quad q := \frac{p_1 \, \e_\sq + p_4}{p_1 + p_4(\cG \bs \{0\})}.
  \end{align*}
 The revival procedure transforms the generator in the expected way, that is, the killing parameter is resolved in the jump (revival) measure
 (a proof is found in Appendix~\ref{app:killing and revival}):
 \begin{lemma} \label{lem:G_IM:Feller resolvent of revived X}
 Let $X^\bullet$ be a Brownian motion on the star graph $\cG$ with generator
  \begin{align*}
   \sD(A^\bullet)  = 
       \Big\{ & f \in \cC^2_0(\cG): \\ 
              &       p_1 f(0) - \sum_{e \in \cE} p^{e}_2 f_e'(0) + \frac{p_3}{2} f''(0) - \int_{\cG \bs \{0\}} \big( f(g) - f(0) \big) \, p_4(dg) = 0 \Big\}.
  \end{align*}
 and lifetime~$\zeta$. 
 Let $q$ be a probability measure on $\cG$, and 
 $X^q$ be the identical copies process, 
 resulting from successive revivals (see subsection~\ref{subsec:C_CO:identical iterations}) of~$X^0 := X^\bullet$  with the revival kernel~$K^0 = q$.
 If $\varphi_\a := \EV_{\,\cdot} \big( e^{-\a \z} \big) \in \cC_0(\cG)$ and $1 - \varphi_\a \in \sD(A^\bullet)$ for all $\a > 0$,
 then $X^q$ is a Brownian motion on $\cG$ with generator
  \begin{align*}
   \sD(A^q) = 
       \Big\{ &  f \in \cC^2_0(\cG) : \\ 
              & - \sum_{e \in \cE} p^{e}_2 f_e'(0) + \frac{p_3}{2} f''(0) - \int_{\cG \bs \{0\}} \big( f(g) - f(0) \big) \, (p_4 + p_1 \, q)(dg) = 0 \Big\}. 
  \end{align*}
\end{lemma}
Here, the requirements on the functions $\varphi_\a$, $\a > 0$, are fulfilled, as seen in \cite[Lemma~1.12]{KPS_Walsh12A} and \cite[Corollary~3.5]{KPS_Walsh12B},
 for $g = (e,x) \in \cG$, we have:
  \begin{align*}
   \varphi_\a(g) 
   = \EV_g \big( e^{-\a \h_0} \big) \, \EV_0 \big( e^{-\a \z} \big) 
   = e^{-\sqrt{2\a}x} \, \frac{p_1 + p_4( \cG \bs \{0\})}{\big(p_1 + p_4( \cG \bs \{0\})\big) + \sqrt{2\a} \, p_2 + \a \, p_3}.
  \end{align*}
 Thus, $X^q$ is a Brownian motion on $\cG \cup \{\sq\}$ with generator
  \begin{align*}
   & \sD(A^q) \subseteq  \ \Big\{  f \in \cC^2_0(\cG) : \\ 
   & \quad                      - \sum_{e \in \cE} p^{e}_2 f_e'(0) + \frac{p_3}{2} f''(0) 
                                - \int_{(\cG \bs \{0\}) \cup \{\sq\}} \big( f(g) - f(0) \big) \, \big( p_1 \, \e_\sq + p_4 \big)(dg) = 0 \Big\}.
  \end{align*}
  
 Finally, map the absorbing set $\{\sq\}$ to $\D$ by considering
 \begin{align} \label{eq:G_GC:killing mapping}
  \psi \colon \cG \rightarrow \cG \bs F, ~ x \mapsto \psi(x) :=
    \begin{cases}
     x,  & x \in \cG \bs F, \\
     \D, & x \in F,
    \end{cases}
 \end{align}
for the isolated set $F := \{\sq\}$.
It is easy to verify the requirements of~\cite[Theorem~(13.5)]{Sharpe88} in order to show that $\psi(X^q)$ is a right (and thus, strongly Markovian) process.
In general, by mapping an isolated set $F$ to the cemetery point~$\D$, all jumps to~$F$ are transformed to immediate killings,
so~$p_4(F)$ is transformed to an additional killing weight~$p_1$ (see Appendix~\ref{app:killing and revival} for the proof): 
\begin{lemma} \label{lem:G_GC:killing on absorbing set, generator data}
 Let $X$ be a Brownian motion on $\cG$ with generator 
  \begin{align*}
   \sD(A^X)
   \subseteq \Big\{ & f \in \cC^2_0(\cG): \\
       &  p_1 f(0) - \sum_{e \in \cE} p^{e}_2 f_e'(0) + \frac{p_3}{2} f''(0) 
         - \int_{\cG \bs \{0\}} \big( f(g) - f(0) \big) \, p_4(dg) = 0 \Big\},
  \end{align*}
 and $F \subsetneq \cG$ be an isolated, absorbing set for~$X$.
 Let $Y := \psi(X)$ be the process on~$\cG \bs F$ resulting from killing~$X$ on~$F$, 
 with $\psi$ as given in equation~\eqref{eq:G_GC:killing mapping}.
 Then the domain of the generator of~$Y$ satisfies
  \begin{align*}
   & \sD(A^Y) 
   \subseteq  \Big\{  f \in \cC^2_0(\cG \bs F):\\
        &  \big( p_1 + p_4(F) \big) f(0) - \sum_{e \in \cE} p^{e}_2 f_e'(0) + \frac{p_3}{2} f''(0) 
         - \int_{\cG \bs (F \cup \{0\})} \big( f(g) - f(0) \big) \, p_4(dg) = 0 \Big\}.
  \end{align*} 
\end{lemma}

 Thus, the generator of the mapped Brownian motion $X := \psi(X^q)$ on $\cG$ satisfies
    \begin{align*}
    \sD(A) \subseteq  \Big\{ & f \in \cC^2_0(\cG) : \\ 
                              & p_1 f(0) - \sum_{e \in \cE} p^{e}_2 f_e'(0) + \frac{p_3}{2} f''(0) 
                                - \int_{\cG \bs \{0\}} \big( f(g) - f(0) \big) \, p_4(dg) = 0 \Big\}.
  \end{align*}
 However, $X$ is a Brownian motion on a star graph, so Lemma~\ref{lem:G_BM:totality on star graph} asserts that $\sD(A)$ indeed equals the right-hand set. 
 
 If $p_2 = 0$ and $p_3 > 0$ (the case $p_2 = 0$ and $p_3 = 0$ is impossible if $p_4$ is finite, as seen in~\eqref{eq:G_BM:infinite jumpmeasure}), 
 the resulting process is simpler. 
 In this case, the construction follows exactly the same lines as above,
 except that instead of considering a standard Walsh process $W$, we start with
 a Walsh process $W^a$ absorbed at the vertex $0$, which is then killed when $W^a$ has stopped at $0$ for an independent, exponentially distributed time with rate
 $\beta := \frac{p_1 + p_4( \cG \bs \{0\} )}{p_3} $. As seen in \cite[Subsection 1.4]{KPS_Walsh12A}, the domain of the resulting Brownian motion $W^e$ reads
  \begin{align*}
   \sD(A^e) & =  \Big\{ f \in \cC^2_0(\cG) :  \beta f(0) +  \frac{1}{2} f''(0) = 0 \Big\} \\
            & =  \Big\{ f \in \cC^2_0(\cG) :  \big( p_1 + p_4( \cG \bs \{0\} ) \big) f(0) +  \frac{p_3}{2} f''(0) = 0 \Big\}.
  \end{align*}
 The remaining construction for the implementation of the jumps then proceeds as above.

\section{Construction for Infinite Jump Measures}\label{sec:G_IM}

The above-described method of successive revivals only implements finite jump measures~$p_4$. In the infinite case, we need a completely different approach.

\subsection{It\^o\textendash{}McKean's Construction} \label{subsec:B_HL:I-M}

In \cite{ItoMcKean63}, It\^o and McKean obtained  a complete pathwise construction of a Brownian motion on $\R_+$ for any given set $(p_1, p_2, p_3, p_4)$
of boundary conditions.
Especially, they managed to implement jumps for an infinite measure $p_4$. In this case,
the Brownian motion, when started at or hitting the origin, needs to perform infinitely many small jumps in some arbitrarily small time interval. Thus, just like when considering 
excursions of the one-dimensional Brownian motion from any point, it is not possible to enumerate them in temporal order to construct the complete process
via successive independent copies of killed Brownian motions, as described in section~\ref{sec:KPS extension}.
However, It\^o and McKean managed to give an ingenious transformation formula for the paths of a reflecting Brownian motion
with the help of a subordinator with L\'evy measure $p_4$ in order to implement the correct excursions from the origin, which will 
be explained in the following.

% The essential ``new idea'' of \cite{ItoMcKean63} is the solution on the 
% he will happily share their story told by McKean in \:
% 
%  \textit{``But what if $p_4( (0,\infty)) = + \infty$? That was mysterious. Luckily, It\^{o} saw at once that it must describe ``jumps'' of a new kind, produced by
%   the increasing ``differential'' process \textelp{}, and as we were flying one day, to Fukuoka I think, It\^{o} kept drawing pictures, one after another,
%   trying to see how these jumps could be interlaced with the Brownian path. After a while he got it; after a longer while I got it, too, and the rest was plain sailing \textelp{}.''
%  }
    
Without being able to verify whether the following chain of arguments really led them to their solution 
(a little bit on the history of their findings can be looked up in~\cite[Section 4.7]{McKeanSelecta}), we try to motivate their approach:
As jumps are only possible if the process is at the origin, they appear on the timeline of the local time at the origin. Furthermore, there is at most one jump
at a time, and jumps need to be independent, in the sense that they need to exhibit a Markovian character, as any Brownian motion on $\R_+$ is strongly Markovian.
By the renewal characterization for point processes (cf.~\cite[Theorem~3.1]{Ito72}), it is therefore natural to expect that the jumps are guided by a Poisson point process
(or equivalently, by a subordinator) with L\'evy measure $p_4$, on the timeline of the local time.
Thus, starting with a reflecting Brownian motion $\abs{B}$, we try to superpose $\abs{B}$ with a subordinator $P$:
The na\"ive approach of considering the process $t \mapsto \abs{B_t} + P(L_t)$ fails, as shown in figure \ref{fig:B_HL:IM naive},
because, after each jump, the process must behave like a standard Brownian motion---in contrast to a reflecting one---until the next hit of the origin. 

\begin{figure}[tb] 
   \includegraphics[width=0.9\textwidth,keepaspectratio]{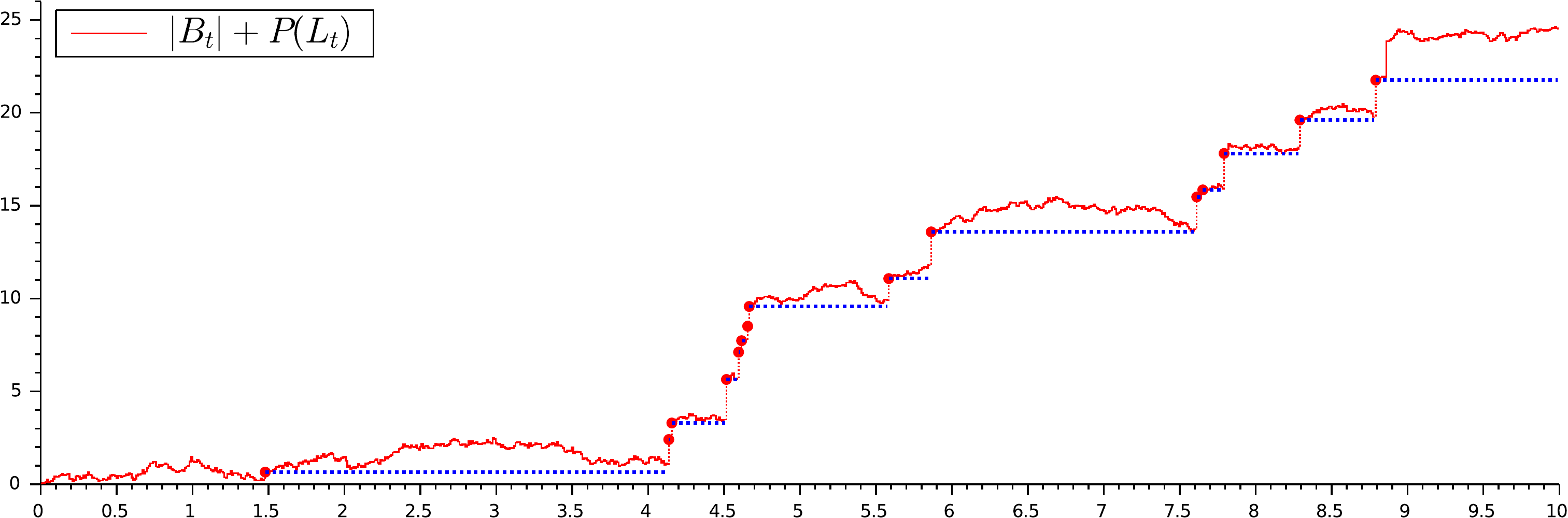}
   \includegraphics[width=0.9\textwidth,keepaspectratio]{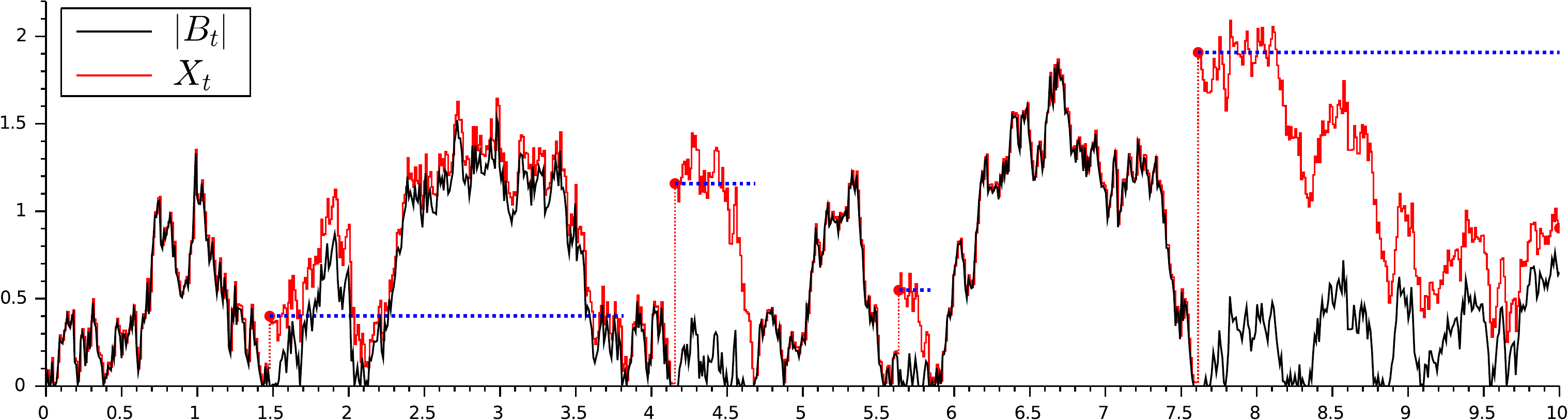}
   \caption[Construction approach for Brownian motions on $\R_+$]
           {Construction approach for Brownian motions on $\R_+$: The idea of $t \mapsto \abs{B_t} + P(L_t)$, illustrated in the above graph, fails, as 
            the process must switch to a standard Brownian motion after each jump until the next hit of zero, as shown in the graph of~$X$ below.
            The blue lines mark the starting heights of the ``jump excursions'',
            as well as (in the graph for $X$) the time elapsed until the switchover back to the reflecting Brownian motion.} \label{fig:B_HL:IM naive}
\end{figure}

Therefore, the goal is to find a way to toggle between reflecting Brownian motion and standard Brownian motion on the level of paths.
As seen in L\'evy's characterization of the local time (cf.~Theorem~\ref{theo:B_LT:levys local time}), for a reflecting Brownian motion~$\abs{B}$ with local time~$L$ at the origin,
the process $\abs{B} - L$ behaves like a standard Brownian motion. So the main idea is to toggle the paths $\abs{B}$ and $\abs{B} - L$,
more accurately: Start with~$\abs{B}$ until the first jump is introduced by~$P$, say of height $h > 0$, then switch to the ``jump excursion'' $h + \abs{B} - L$
until this part hits the origin again, then toggle back to~$\abs{B}$, and so on. As $\abs{B}$ is non-negative and~$L$ only grows when~$\abs{B}$ is at the origin, the partial process
$h + \abs{B} - L$ hits zero exactly when $L$ is increased by~$h$. Following this thought, the prototype of the process should be of the form 
 $t \mapsto \abs{B_t} - L_t + F(L_t)$ for some random function~$F$ which is the identity while the reflecting Brownian motion needs to be in place,
 jumps by~$h$ whenever a jump excursion with jump height $h > 0$ needs to be started, and then is constant for $h$ units of time.
 Such a function is obtained by the choice $F = P \circ P^{-1}$, with $P^{-1}$ being the right continuous pseudo-inverse of the subordinator $P$:
  \begin{align*}
   P^{-1} \colon [0, \infty) \rightarrow [0, \infty], \quad P^{-1}(t) := \inf \{ s \geq 0: P(s) > t \}.
  \end{align*}
 Pseudo-inverses and functions of the form $P \circ P^{-1}$ are examined in detail in Appendix~\ref{app:G_IM:remarks}.
 For now, we recommend the graphs of figure \ref{fig:B_HL:IM graphs} to the reader: The upper right hand graph contains the jumps of the Poisson point process (in black)
 and its associated subordinator $P$ with an additional deterministic drift (in red),
 the lower left hand graph shows the resulting process $P \circ P^{-1}$ which exactly features the properties stated above,
 that is, being a diagonal, interrupted by upper isosceles triangles.
 In summary, \IM's solution is the process
  \begin{align*}
   X_t = \abs{B_t} - L_t + P \big( P^{-1}(L_t) \big), \quad t \geq 0,
  \end{align*}
 which is shown in the lower right hand graph of figure \ref{fig:B_HL:IM graphs}. Proving that this process is a (strong) Markov process and indeed introduces the correct 
 jump measure $p_4$ is not an easy task and is done in \cite[Sections 13--15]{ItoMcKean63}. We will take up this challenge 
 in subsections~\ref{subsec:G_IM:definitions}--\ref{subsec:G_IM:strong Markov of X}
 when we extend \IM's construction to the star graph. Afterwards, the missing killing and stickiness parameters $p_1$ and $p_3$ can be introduced by
 the standard procedure of ``slowing down'' the process $X$ by time changing it with respect to its local time at the origin, and then kill it
 once its new local time exceeds some independent, exponentially distributed random time, see  \cite[Sections 10, 15]{ItoMcKean63} or subsection~\ref{subsec:G_IM:general BB on star graph}.
 
\begin{figure}[tb] 
   %%\disableHboxWarning
   \hspace*{-2em}
   \includegraphics[height=4.5cm,keepaspectratio]{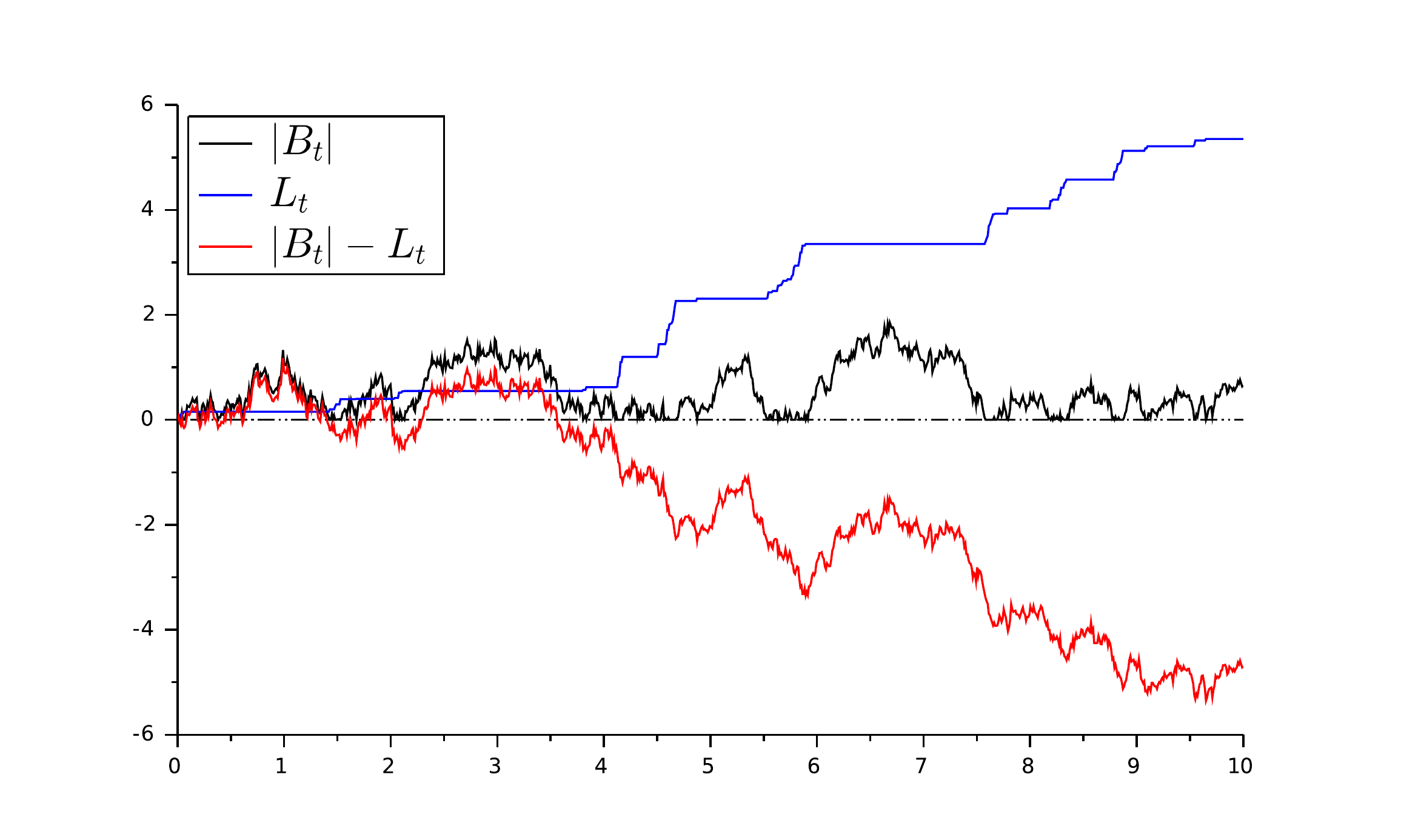}
   \includegraphics[height=4.5cm,keepaspectratio]{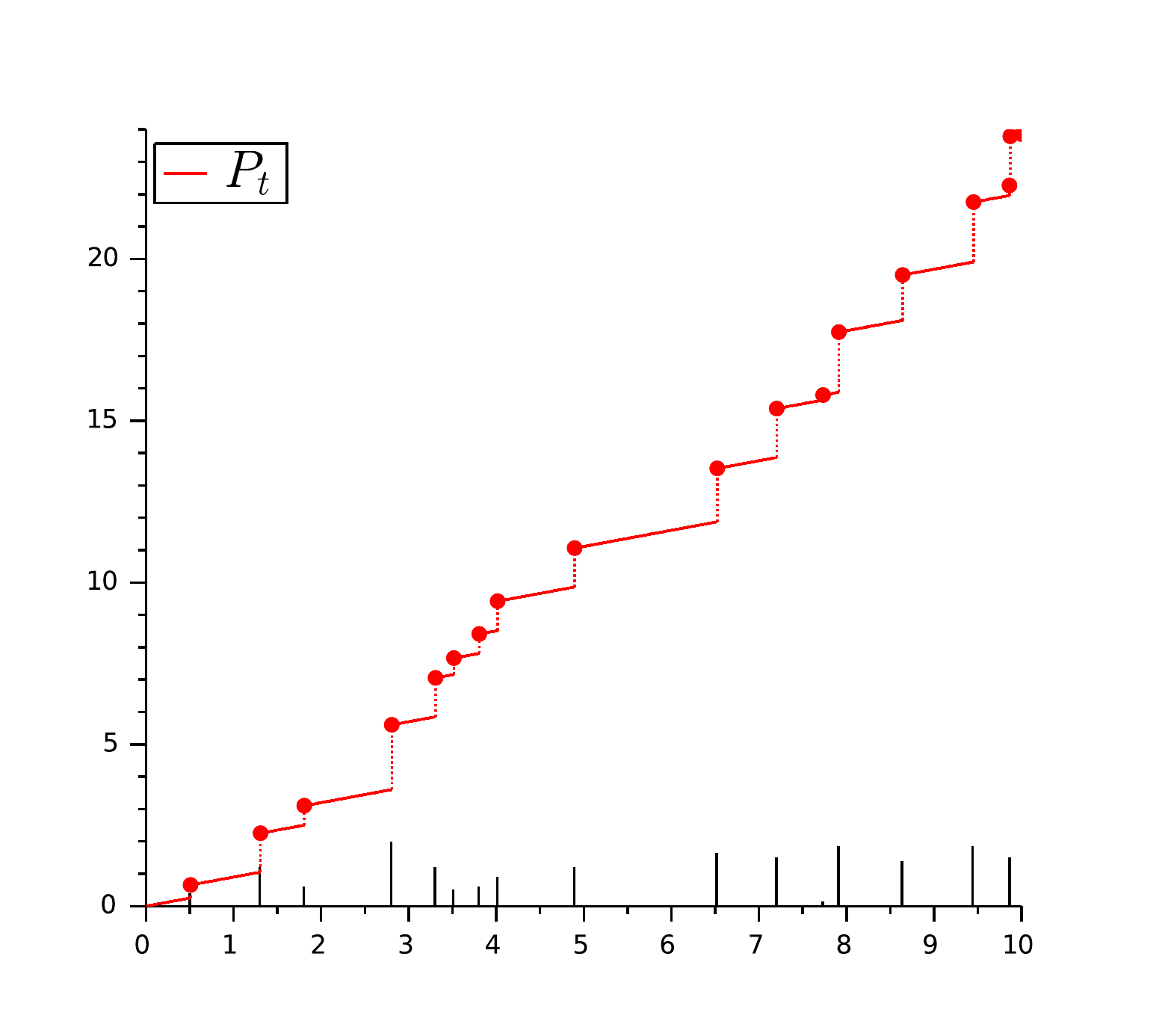} \\
   \hspace*{-2em}
   \hspace*{0.37cm}
   \includegraphics[height=4.5cm,keepaspectratio]{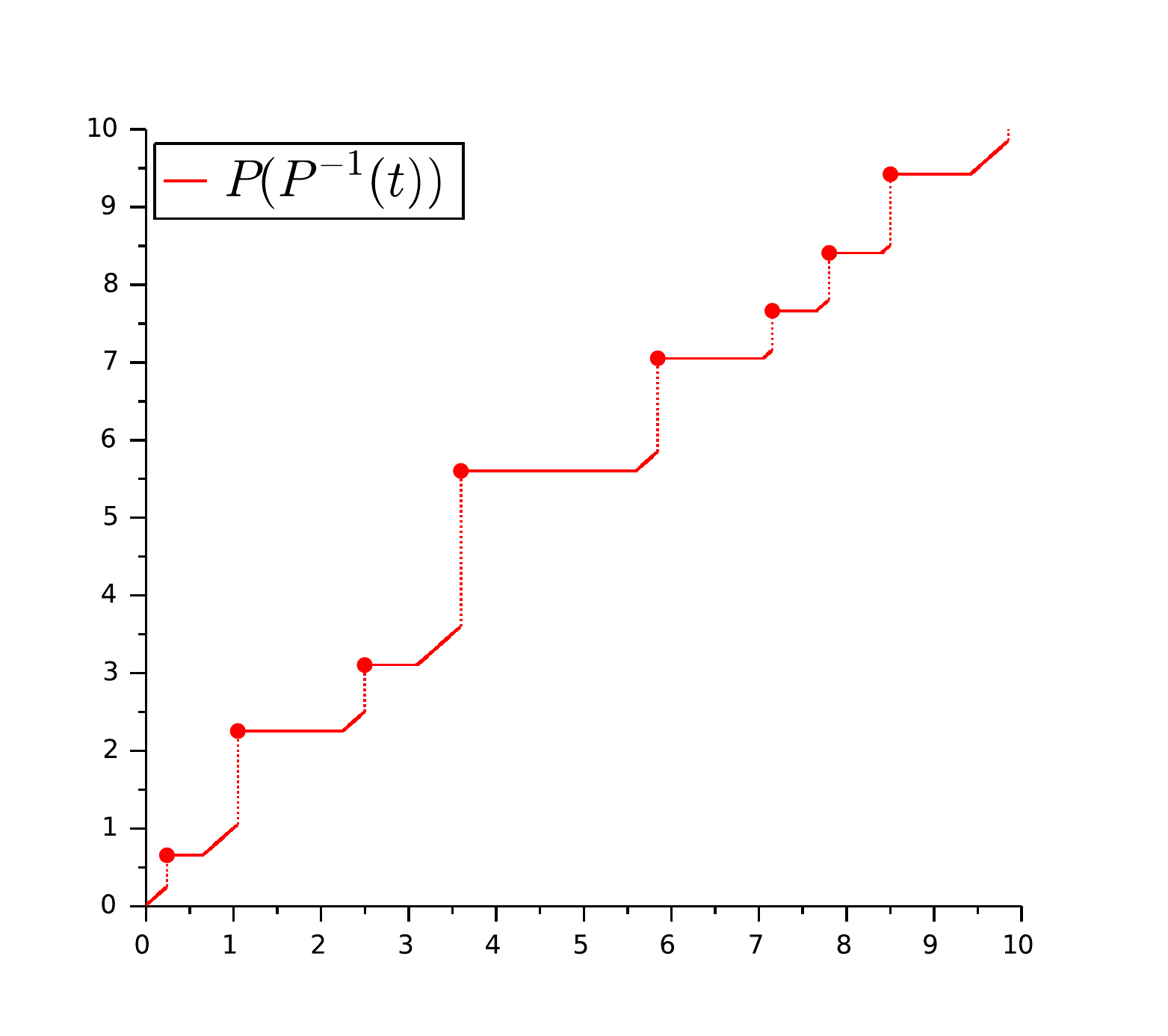}
   \includegraphics[height=4.5cm,keepaspectratio]{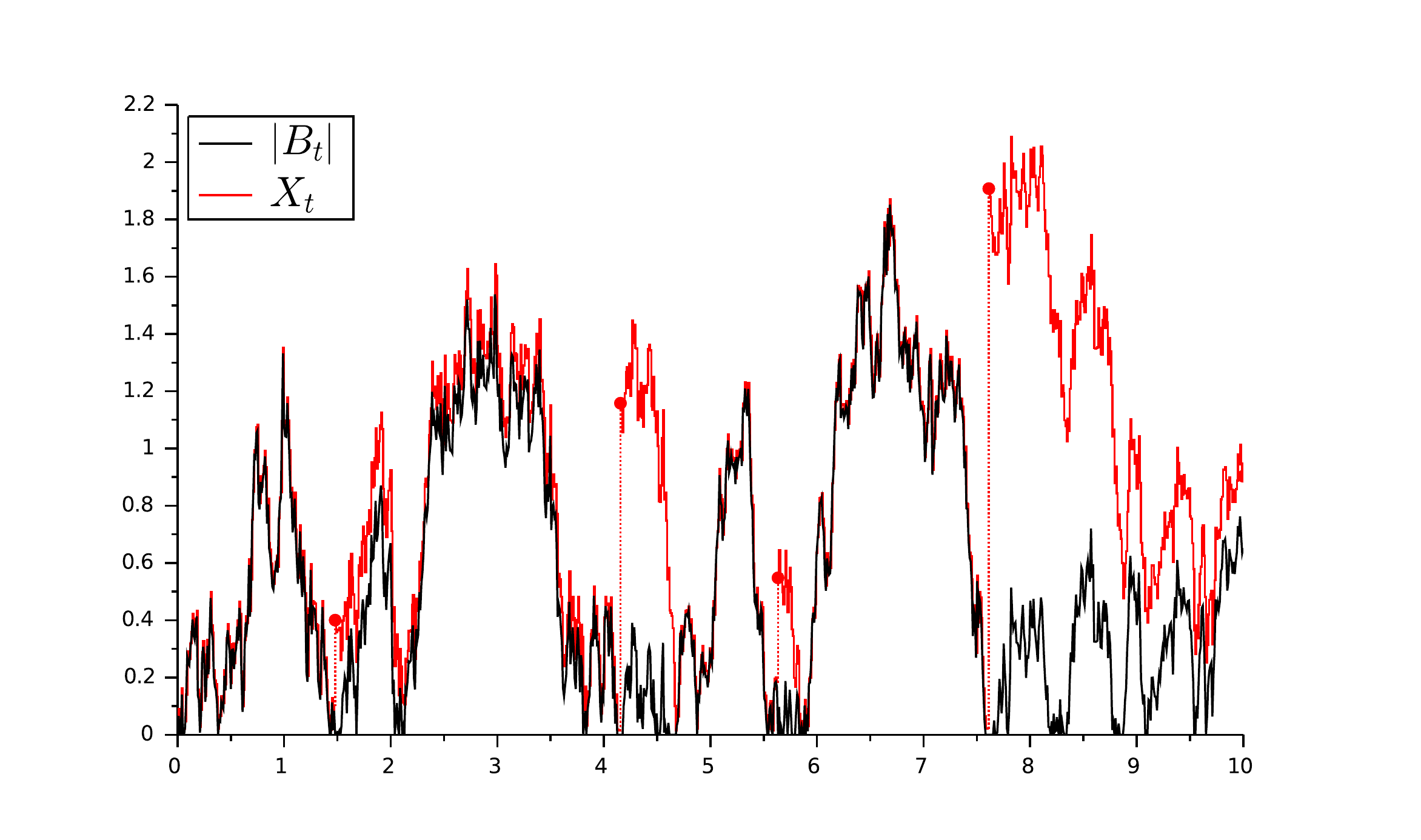}
   \hspace*{-1em}
   \caption{\IM's construction of Brownian motions on $\R_+$} \label{fig:B_HL:IM graphs}
\end{figure}

We end the treatment of the half-line case by noting that, of course, there are other ways to analyze and construct Brownian motions on $\R_+$.
A natural approach is via It\^{o} excursion theory, see, e.g., \cite{Rogers89}, \cite[Section~VI.57]{RogersWilliams2}, or \cite{Blumenthal92}.

\subsection{Construction}

We are going to construct all Brownian motions on a star graph by extending the \IM's approach of \cite{ItoMcKean63} for the half-line case, 
which was just explained above. 

In all that follows, let $\cG$ be a fixed star graph with star vertex $0$ and set of edges~$\cE$. 
For keeping notations readable in the following construction, we will assume that $\cE = \{ e_1, \ldots, e_n \}$ holds with $n = \abs{\cE}$.
As usual, we consider the representation 
  \begin{align*}
   \cG = \{0\} \cup \bigcup_{i=1}^n \big( \{e_i\} \times [0, \infty) \big)
  \end{align*}
of the graph $\cG$, with all initial points $(e,0)$, $e \in \cE$, being identified with the vertex~$0$.

Furthermore, with regard to the assertions of Feller's Theorem~\ref{theo:G_BM:Fellers theorem on star graph}, we assume that 
we are given a fixed set of boundary weights
 \begin{align*}
  p_1 \geq 0, ~ p^e_2 \geq 0 \text{ for each } e \in \cE, ~ p_3 \geq 0, ~ p_4 \text{ measure on $\cG \bs \{0\}$}, 
 \end{align*}
satisfying $\int_{\cG \bs \{0\}} \big( 1 - e^{-x} \big) \, p_4\big(d(e,x)\big) < +\infty$ and 
 \begin{align*}
  p_4 \big( \cG \bs \{0\} \big) = +\infty, \quad \text{if} \quad  p_3 = 0  \text{ and } p^e_2 = 0\text{ for all $e \in \cE$}.
 \end{align*}
We define the total and partial reflection weights
 \begin{align*}
  p_2 := \sum_{e \in \cE} p^e_2, \quad
  \text{ for } e \in \cE: ~
  q^e_2 := \begin{cases}
            p^e_2 / p_2, & \text{if } p_2 \neq 0, \\
            1/n, & \text{if } p_2 = 0,
           \end{cases}
 \end{align*}
and decompose the jump measure on the separate edges by introducing for each edge~$e \in \cE$ a measure $p^e_4$ on $(0, +\infty)$ by
 \begin{align*}
  p^e_4 (A) := p_4 \big( \{e\} \times A \big), \quad A \in \sB \big( (0, +\infty) \big).
 \end{align*}
Then the measures $p^e_4$, $e \in \cE$, also satisfy
 \begin{align*}
  \int_{(0, \infty)} \big( 1 - e^{-x} \big) \, p^e_4(dx) < +\infty,
 \end{align*}
and $p^e_4 \big( (0, \infty) \big) = +\infty$ holds for some $e \in \cE$, if $p_3 = 0$ and $p^e_2 = 0$ for all $e \in \cE$.

\begin{remark}
 The reader may notice that we do not require the normalization of the parameters $\big( p_1, (p_2^e)_{e \in \cE}, p_3, p_4 \big)$.
\end{remark}

In the following, we will always assume that $p_2 > 0$ or that $p_4$ is infinite.

% \subsection{It\^o-McKean's Konstruktion f\"ur den Sterngraphen}
% 
% \textbf{Ideen:}
% \begin{itemize}
%  \item Walsh-Prozess $(W_t, t \geq 0)$ auf einem Sterngraph l\"asst sich darstellen als 
%      \begin{align*} W_t = \pi_\sE(W_t) \times \pi_{\R_+}(W_t), \end{align*}
%        die erste Koordinate $\pi_\sE(W_t)$ ist eine "`Kanten-Markovkette"', \\
%        die zweite Koordinate $\pi_{\R_+}(W_t)$ ist eine reflektierende Brownsche Bewegung, \pause
%  \item \"ubernehme also It\^o-McKean's Konstruktion f\"ur die zweite Koordinate, \pause
%  \item dann ben\"otigt man einen neuen Kantenprozess, der w\"ahrend einer "`Sprungexkursion"' die alte "`Kanten-Markovkette"' "`\"uberschreibt"', \pause
%  \item f\"ur wird Information ben\"otigt auf welche Kante wie hoch gesprungen wurde.
% \end{itemize}

\subsection{Definitions} \label{subsec:G_IM:definitions}
  The main ingredients for our construction will be a Walsh process $W$ on $\cG$ and a family of subordinators $(Q^e, e \in \cE)$,
  which are used to control the jumps to the respective edges. We are introducing them on an appropriate, common space now:

%\begin{itemize}
 %\item 
    Let $\hW = \big( \O^W, \sF^W, (\sF^W_t)_{t \geq 0}, (\hW_t)_{t \geq 0}, (\hT^W_t)_{t \geq 0}, (\PV^W_{(e,x)})_{(e,x) \in \cG} \big)$ 
    be a Walsh process\footnote{A collection of standard results on Walsh processes can be found in Appendix~\ref{app:Walsh}.}
    on~$\cG$ with edge weights $q^e_2 = p^e_2 / p_2$, $e \in \cE$,\footnote{If $p_2 = 0$ (this requires $p_4 = +\infty$),
    then consider a Walsh process with arbitrary weight distribution, for instance use $q^e_2 = 1/n$ for all $e \in \cE$.
    Any choice leads to the correct boundary condition, as will be seen in subsection~\ref{subsec:G_IM:resolvent and generator}.}
    and $(\hL_t, t \geq 0)$ be the local time of $\hW$ at $0$.
    We have for all $s,t \geq 0$:
      \begin{align*}
       \hW_s \circ \hT^W_t = \hW_{s+t}, \quad \hL_s \circ \hT^W_t = \hL_{s+t} - \hL_t.
      \end{align*}

    For each $e \in \cE$, let 
    $\hQ^e = \big( \O^{Q,e}, \sF^{Q,e}, (\sF^{Q,e}_t)_{t \geq 0}, (\hQ^e_t)_{t \geq 0}, (\hT^{Q,e}_t)_{t \geq 0}, (\PV^{Q,e}_{q})_{q \in \R} \big)$
    be a subordinator with L\'evy measure $p^e_4$ and drift $0$ realized as canonical coordinate process on the space $\O^{Q,e}$ of all c\`adl\`ag functions.
    We then have natural translation and centering operators $(\hg^{Q,e}_q, q \in \R)$ and $\hG^{Q,e}$
    at our disposal, see~\eqref{eq:natural centering and translation}.
    
    Let $\hQ := (\hQ^{e_1}, \ldots, \hQ^{e_n})$ be the Cartesian product of the processes ${(\hQ^e, e \in \sE)}$,
    that is, %\FINALVSPACE{-0.1em}
     \begin{align*}
      \hQ = \big( \O^{Q}, \sF^{Q}, (\sF^{Q}_t)_{t \geq 0}, (\hQ_t)_{t \geq 0}, (\hT^{Q}_t)_{t \geq 0}, (\PV^{Q}_{(q^1, \ldots, q^n)})_{(q^1, \ldots, q^n) \in \R^n} \big)
     \end{align*}
    with sample space $\O^Q := \prod_{e \in \sE} \O^{Q, e}$, $\sigma$-algebra $\sF^{Q} := \bigotimes_{e \in \sE} \sF^{Q,e}$,
    the process being defined by $\hQ_t := (\hQ^{e_1}_t, \ldots, \hQ^{e_n}_t)$ for any $t \geq 0$, equipped with its natural filtration $(\sF^{Q}_t, t \geq 0)$,
    shift operators $\hT^{Q}_t := \hT^{Q,{e_1}}_t \times \cdots \times \hT^{Q,{e_n}}_t$, $t \geq 0$, 
    translation operators $\hg^Q_{(q^1, \ldots, q^n)} := \hg^{Q,{e_1}}_{q^1} \times \cdots \times \hg^{Q,{e_n}}_{q^n}$, $q^1, \ldots, q^n \in \R$,
    centering operator $\hG^Q := \hG^{Q,1} \times \cdots \times \hG^{Q,n}$,
    as well as initial measures $\PV^Q_{(q^1, \ldots, q^n)} := \PV^{Q, {e_1}}_{q^1} \otimes \cdots \otimes \PV^{Q, {e_n}}_{q^n}$ for all~$q^1, \ldots, q^n \in \R$.
    
    By construction, the processes $\hQ^{e_1}, \ldots, \hQ^{e_n}$ are independent, so %by Lemma~\ref{lem:A_LP:simultaneous jumps of independent processes},
    the set $N$ of simultaneous jumps of $\hQ^{e_1}, \ldots, \hQ^{e_n}$ is a measurable null set. As the natural shift, translation and centering operators
    do not change or introduce new discontinuities, they map $\O^Q \bs N$ into itself. Therefore, we are able to restrict the process $\hQ$ together with all its operators
    to $\O^Q \bs N$, naming this new sample space $\O^Q$ again. Thus, at most one of the processes $\hQ^1, \ldots, \hQ^n$ has a jump at any given time $t > 0$. 
%       \begin{align*}
%         & \O^Q := \prod_{e \in \sE} \O^{Q, e}, & \\
%         & \sF^{Q} := \bigotimes_{e \in \sE} \sF^{Q_e}, & \\
%         & \sF^{Q}_t := \sigma \big( f(Q^1_t, \ldots, Q^n_t), f \in b\sB(\R)^{\otimes n} \big), ~ t \geq 0, & \\
%         & \PV^Q_{(q^1, \ldots, q^n)} := P^{Q, 1}_{q^1} \otimes \cdots \otimes  P^{Q, n}_{q^n}, ~ q^1, \ldots, q^n \in \R, & \\
%         & \hT^{Q}_t := \hT^{Q,1}_t \times \cdots \times \hT^{Q,n}_t, ~ t \geq 0.&
%       \end{align*}
 %\item
 
     Now combine the Walsh process and the subordinators independently in one space by defining the product space $\O := \O^W \times \O^Q$
     with $\sigma$-algebra $\sF := \sF^W \otimes \sF^Q$
       %with filtration $\sF_t := \sF^{W}_t \otimes \sF^{Q}_t$, $t \geq 0$, 
       and product measures
      \begin{align*}
       \PV_{(e,x), (q^1, \ldots, q^n)} := \PV^W_{(e,x)} \otimes \PV^Q_{(q^1, \ldots, q^n)}, \quad (e,x) \in \cG, (q^1, \ldots, q^n) \in \R^n.
      \end{align*}
     As we will typically want to start the subordinators at the origin, we set %\FINALVSPACE{-0.13em}
      \begin{align*}
       \PV_g := \PV_{g, (0, \ldots, 0)}, \quad g \in \cG. 
      \end{align*}
     With the canonical projections $\p^W \colon \O^W \times \O^Q \rightarrow \O^W$,
     $\p^Q \colon \O^W \times \O^Q \rightarrow \O^Q$,
     and $\p^{Q,e} \colon \O^{Q,1} \times \cdots \times \O^{Q,n} \rightarrow \O^{Q,e}$, $e \in \cE$, 
       we set for any $t \geq 0$, $q \in \R^n$: %\FINALVSPACE{-0.13em}
        \begin{align*}
         & W_t := \hW_t \circ \p^W, \quad L_t := \hL_t \circ \p^W, \quad \T^W_t := (\hT^W_t \circ \p^W) \times \p^Q, & \\
         & Q_t := \hQ_t \circ \p^Q, \quad Q^e_t := \hQ_t \circ \p^{Q,e} \circ \p^Q, ~ e \in \cE, \\
         & \T^Q_t := \p^W \times (\hT^Q_t \circ \pi^Q), \quad \g^Q_q := \p^W \times (\hg^Q_q \circ \pi^Q), \quad \G^Q := \p^W \times (\hG^Q \circ \pi^Q).  &  
        \end{align*}
 Define the processes $(P_t, t \geq 0)$ and $(P^e_t, t \geq 0)$, $e \in \cE$, by %\FINALVSPACE{-0.13em}
%    \begin{align*}
%      P_e(t)         & := p_2 \, t + \int l ~ N_{4,i}([0,t] \times dl) + \sum_{j \neq i} \int l ~ N_{4,j}([0,t] \times dl), \quad i \in \cE, \\
%      P(t) := P_0(t) & := p_2 \, t + \sum_{i \in \cE} \int l ~ N_{4,j}([0,t] \times dl) \\
%             & = p_2 \, t + \int l \,\Big(\sum_{i \in \cE} N_{4,j} \Big) ([0,t] \times dl), 
%     \end{align*}        
%    with $N_{4,i}$ being the Poisson point process of $Q^i$, so
   \begin{align*}
     P_e(t) & := p_2 \, t + Q^e(t) + \smashoperator[r]{ \sum\limits_{f \in \cE, f \neq e}} Q^f(t-) , \quad e \in \cE, \\
     P(t) := P_0(t)  & := p_2 \, t + \sum_{e \in \cE} Q^e(t),
   \end{align*}       
   where, as usual, we set $Q^e(t-) := \lim_{s \upuparrows t} Q^e(s)$ for $t > 0$, and $Q^e(0-) := Q^e(0)$.
 %\item
 Furthermore, for any vector $\eta = (\eta^e, e \in \cE)$ of real numbers with $\eta^e \leq 0$ for all except at most one $e \in \cE$,
       we construct a function $E( \eta ) \colon \cG \rightarrow \cE$ by setting %\FINALVSPACE{-0.13em} 
        \begin{align*}
         E( \eta ) (l,x)
         := \begin{cases}
             e, & \exists e \in \cE: \eta^e > 0, \\
             l, & \forall e \in \cE: \eta^e \leq 0.
            \end{cases}
        \end{align*}
% 	\begin{align*}
% 	  E(\eta^i, i \in \cE)
% 	  := \sum_{i \in \cE} \1_{\R_{>0}}(\eta^i) \cdot 1^i_\cE + \prod_{i \in \cE} \1_{R_{\leq 0}}(\eta^i) \cdot \id_\cE, 
% 	\end{align*}
%        with $\id_\cE \colon \cG \rightarrow \cE, ~ (e,x) \mapsto e$, $1^i_\cE \colon \cG \rightarrow \cE, ~ (e,x) \mapsto i$
%        (this function is only well-defined if there is at most one $i \in \cE$ for which $\eta^i > 0$, this will always be the case in our use of the function,
%         see lemma \ref{x}).
 %\item
 For all $e \in \cE$, define the processes $(\eta^e_t, t \geq 0)$ by %\FINALVSPACE{-0.13em}
	\begin{align*} 
	  \eta^e_t := \big( P_e P^{-1} - \id \big) (L_t), \quad t \geq 0. 
	\end{align*}
 %\item
 Finally, we define the stochastic process $(X_t, t \geq 0)$ on $\cG$, by setting
	\begin{align*}
	  X_t := \big( E(\eta^e_t, e \in \cE) \circ W_t, \eta_t + \abs{W_t} \big), \quad t \geq 0.
	\end{align*}
For later use, we also set
 \begin{align*}
   \varrho_t := P^{-1}(L_t) = \inf \{ s \geq 0: P_s > L_t \}, \quad t \geq 0.
 \end{align*}

\subsection{Remarks on the Definition} \label{subsec:G_IM:remarks}
The process $(X_t, t \geq 0)$  will turn out to be a Brownian motion on $\cG$ which realizes the reflection parameters $(p^e_2, e \in \cE$)
 and the jump measure $p_4$ in the boundary condition of the generator.
 It is a generalization of \IM's construction on the half line, which was explained in subsection~\ref{subsec:B_HL:I-M}.
 
Indeed, the local coordinate of $X_t$ is, by definition, just \IM's ``basic'' Brownian motion on the half line, namely $P P^{-1}(L_t) - L_t + \abs{W_t}$.
However, we need to adjust their construction by a process which controls the edges of the Brownian motion on the star graph:
 We cannot use the edge process of $(W_t, t \geq 0)$, as this would change the edge whenever $\abs{W_t}$ is at $0$,
 even if the translated excursion $P P^{-1}(L_t) - L_t + \abs{W_t}$ is not finished yet, see figure~\ref{fig:G_IM:complete processes}.

\begin{figure}[tb] 
   \centering
   \includegraphics[keepaspectratio,width=0.98\textwidth]{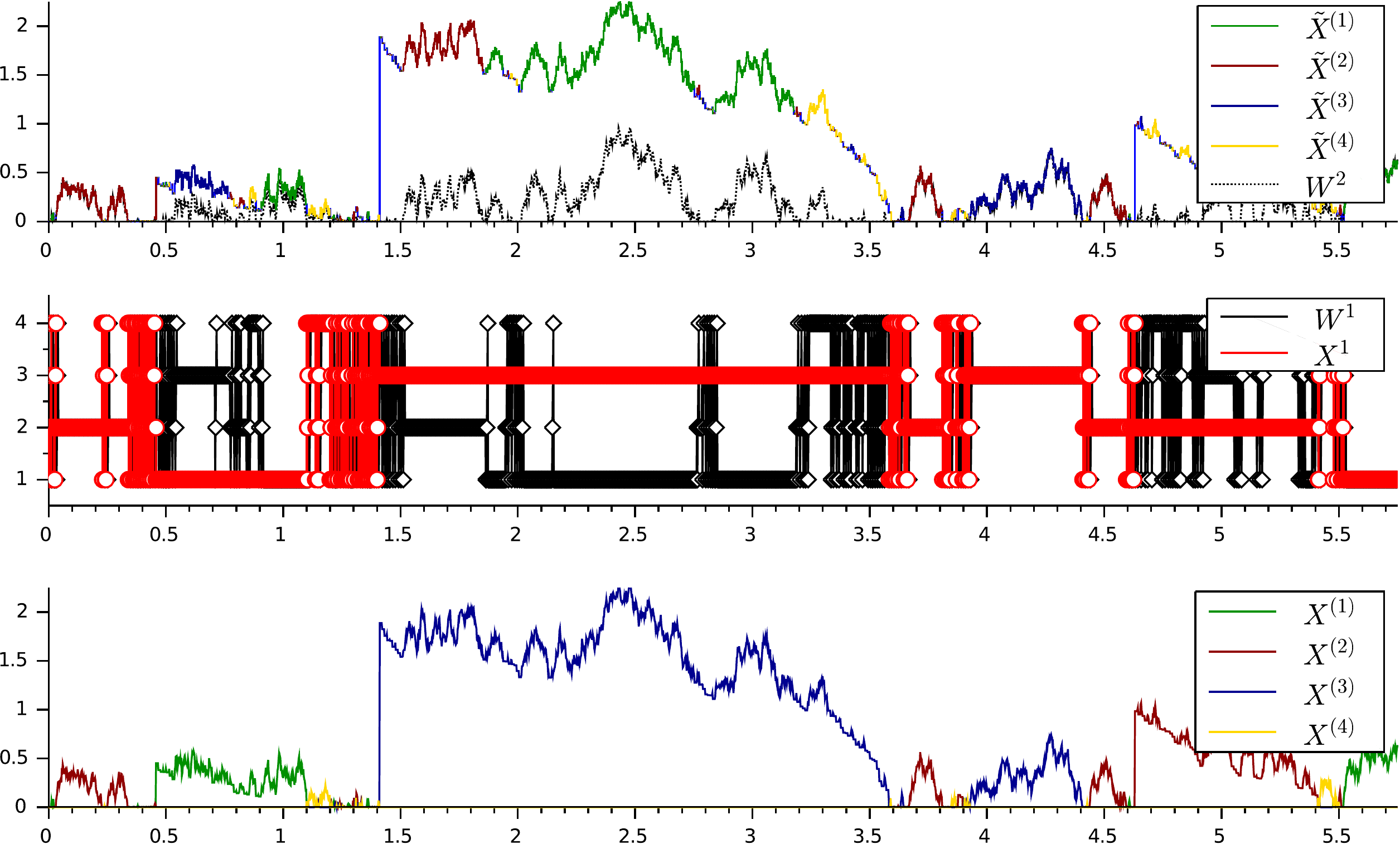}
   \caption[Construction approach for Brownian motions on a star graph]
           {Construction approach for Brownian motions on a star graph: 
            %Illustration of the Walsh process $W_t = (W^1_t, W^2_t)$ and the resulting Brownian motion with jumps 
           %   $X_t = (X^1_t, X^2_t) = \big( E(\eta^e_t, e \in \cE) \circ W_t, P P^{-1}(L_t) - L_t + \abs{W_t} \big)$.
            The incorrect process $\tX_t = \big( W^1_t, P P^{-1}(L_t) - L_t + \abs{W_t} \big)$, pictured in the first graph,
            already implements the desired radial process,
            however switches edges during jump excursions whenever the original process $W$ hits the vertex $0$.
            Thus, the edge process~$W^1$ must be transformed to~$X^1$ in order to ``hold'' the current edge during jump excursions.
            $X^{(i)}$, $\tX^{(i)}$ represent the process parts of $X$, $\tX$ on the corresponding edges $e_i \in \cE$, $i \in \{1,2,3,4\}$. } \label{fig:G_IM:complete processes}
\end{figure}  
 
 Therefore, we need to ``overwrite'' the edge process of $(W_t, t \geq 0)$ to being constant on some edge $e \in \cE$,
 as long as there is a ``jump excursion'' on this edge. There does not seem to be a straight-forward way to define such an ``overwriting process''.
 Our solution is the introduction of the auxiliary processes $P_e$, $e \in \cE$, which are modifications of the process $P$, namely, being 
 right continuous at the jump times of their own edge $e$, and left continuous at jump times of the other edges. 
 Therefore, on jump excursions on their own edge, $P_e P^{-1}$ will have ``upper triangles'' (which is equivalent to $\eta^e_t > 0$) just as $P P^{-1}$,
 but ``lower triangles'' (which is equivalent to $\eta^e_t < 0$) on jump excursions to other edges, see figure~\ref{fig:G_IM:eta processes}
 and Remark~\ref{rem:G_IM:behavior process X}.
 Thus, it is possible to derive the current edge of a jump excursion 
 from the paths of $t \mapsto P_e P^{-1}(L_t)$, $e \in \cE$, or equivalently from the excursion processes $(\eta^e_t, t \geq 0)$, $e \in \cE$.
 
\begin{figure}[tb] 
   \includegraphics[width=\textwidth,keepaspectratio]{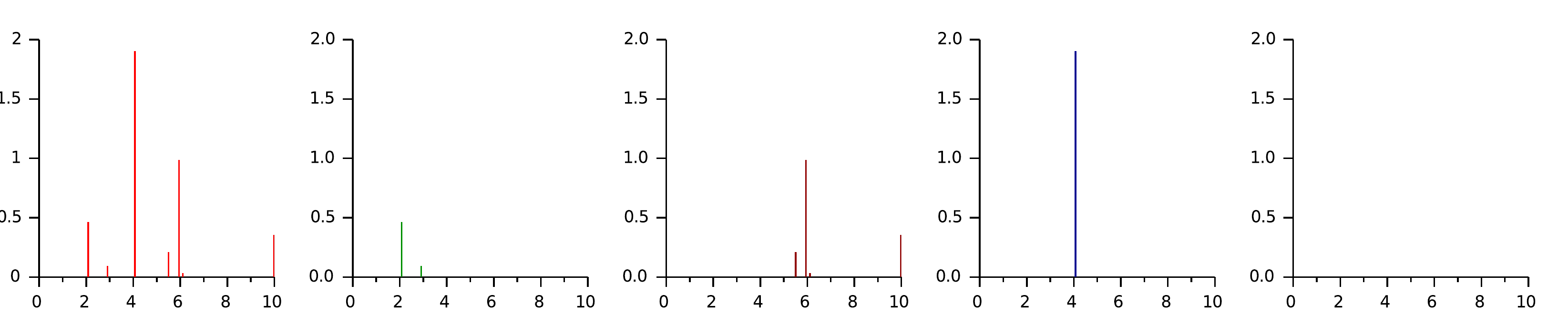}
   
   \includegraphics[width=\textwidth,keepaspectratio]{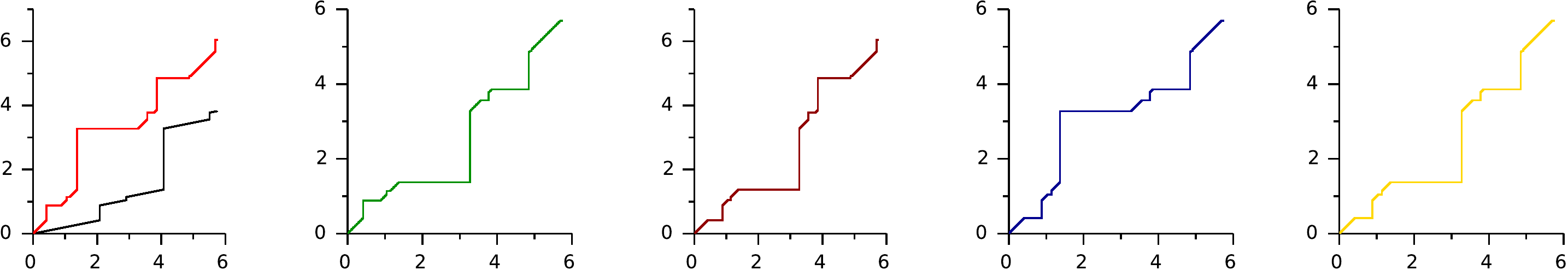}
   \caption[Illustration of the extension of \IM's approach to the star graph]
           {Illustration of the extension of \IM's approach to the star graph: 
            The left-hand graphs show the jumps of the complete subordinator $Q$ and the resulting processes $P$ (black) and
            $P P^{-1}$ (red) used in \IM's construction in the half-line case.
            The other graphs show (in the first row) the decomposition into the subordinator parts $Q^e$ of the corresponding edges $e \in \cE$ 
            (with color assignment like in figure~\ref{fig:G_IM:complete processes}),
            as well as (in the second row) the resulting processes $P_e P^{-1}$,
            which feature the needed ``upper triangles'' for ``own jumps'' and ``lower triangles'' for ``other jumps''.} \label{fig:G_IM:eta processes}
\end{figure} 
 
The process $E(\eta^e_t, e \in \cE)$ thus chooses which (if any) of the jump excursion times $\eta^e_t, e \in \cE$, is currently greater than zero
(that is, which ``triangle'' is the ``upper triangle''),
and holds the motion on this edge $e \in \cE$ for the remaining length $\eta^e_t > 0$ of this excursion; 
during this time \IM's Brownian motion $\eta_t + \abs{W_t}$ on the local coordinate behaves like a standard Brownian motion.
On the other hand, if all jump excursion times are zero, then $E(\eta^e_t, e \in \cE)$ just uses the original edge of the Walsh process $(W_t, t \geq 0)$ and $\eta_t = 0$ holds true,
so both coordinates of $X_t$ coincide with both original coordinates of~$W_t$. 
This means that, as long as there is no jump excursion, $X_t$ is just $W_t$. 
We will make these explanations rigorous in~Appendix~\ref{app:G_IM:remarks}, and only list the main results here:

\begin{remark} \label{rem:G_IM:behavior process X}
The function $P$ is strictly increasing, as $p_2 > 0$ or $p^e_4 \big( (0, \infty) \big) = +\infty$ for at least one $e \in \cE$.
%(cf.~theorem~\ref{theo:A_LP:sample path behavior}).
Thus, $P P^{-1}$ has a level of constancy at some time $P(t-)$ of length $h$,
if and only if $P^{-1}$ has, so by \ref{itm:G_IM:props pseudoinverse:jumps-const} of Lemma~\ref{lem:G_IM:props pseudoinverse},
if and only if there exists a jump of $P$ at time $t \in J$ of height $h$.
Therefore, we can decompose $\R_+$ into 
 \begin{align*}
  \sT := \{ t \geq 0: P P^{-1}(t) = t \} \quad \text{and} \quad \sT^\comp = \bigcup_{n \in \N} [l_n^-, l_n^+).
 \end{align*}
Here, the interval $(l_n^-, l_n^+)$ corresponds to a jump of $P$ at time $t_n$ of height~$l_n$ via
 \begin{align*}
  l_n^- = P(t_n-), \quad l_n^+ - l_n^- = l_n, \quad n \in \N.
 \end{align*}
Then, by definition of $P$, it is
 \begin{align*}
  l_n^+ = l_n^- + l_n =  P(t_n-) + \big( P(t_n) - P(t_n-) \big) = P(t_n), \quad n \in \N,
 \end{align*}
and by~\ref{itm:G_IM:props pseudoinverse:P-1 P} of Lemma~\ref{lem:G_IM:props pseudoinverse}, we also have
 \begin{align*}
  P^{-1}(l_n^-) = P^{-1} \big( P(t_n-) \big) = t_n, \quad n \in \N.
 \end{align*}
This gives for each $n \in \N$,
 \begin{align} \label{eq:G_IM:behavior process P P-1}
  \forall t \in \big[l_n^-, l_n^+\big) = \big[P(t_n-), P(t_n)\big): \quad P P^{-1}(t) = P P^{-1}(l_n^-) = P(t_n).
 \end{align}
For every $P_e P^{-1}$, $e \in \cE$, the same decomposition holds true, that is, we have $P_e P^{-1}(t) = P_e P^{-1}_e (t) = t$ for all $t \in \sT$.
However, observe that by Lemma~\ref{lem:G_IM:behavior process Pe},
 \begin{align} \label{eq:G_IM:behavior process Pe P-1}
  \forall t \in \big[l_n^-, l_n^+\big): %= \big[P(t_n-), P(t_n)\big):
  \quad P_e P^{-1}(t) = P_e(t_n) =
    \begin{cases}
      P(t_n),  & t_n \in J_e, \\
      P(t_n-), & t_n \notin J_e.
    \end{cases}
 \end{align}
\end{remark}

In total, we get the complete path behavior of $X$:
 \begin{equation} \label{eq:G_IM:behavior process X}
 \begin{aligned}
  X_t =
   \begin{cases}
    W_t, & L_t \in \sT, \\
    \big( e, l_n^+ - L_t + \abs{W_t} \big) , & L_t \in  [l_n^-, l_n^+) \text{ with $t_n \in J_e$}.
   \end{cases}
 \end{aligned}  
 \end{equation}
 
\begin{theorem} \label{theo:G_IM:continuity process X}
  The process $(X_t, t \geq 0)$ is right continuous, 
  and it is continuous on any excursion away from $0$, that is:
  For any $t \geq 0$ with $X_t \neq 0$, 
  $(X_t, t \geq 0)$ is continuous on~$[t, t_0]$, with $t_0 := \inf \{ s \geq t: X_s = 0 \}$.
\end{theorem}

\subsection{Shift and Translation Operators for \texorpdfstring{$X$}{X}} \label{subsec:G_IM:shift operators}

% Let $(\hg^{Q, i}_q)_{q \in \R}$, $\hG^{Q,i}$ be translation and centering operators for $\hQ^i$, $i \in \cE$,
% then with
%  \begin{align*}
%   \g^Q_{q^1, \ldots, q^n} & := \p^W \times \big( \hg^{Q, 1}_{q^1} \times \cdots \times \hg^{Q, n}_{q^n} \big) \circ \p^Q, \quad q^1, \ldots, q^n \in \R, \\
%   \G^Q & := \p^W \times \big( \hG^{Q, 1} \times \cdots \times \hG^{Q, n} \big) \circ \p^Q,
%  \end{align*}
% $(\g^Q_{q^1, \ldots, q^n})_{(q^1, \ldots, q^n) \in \R^n}$, $\G^Q$ are translation and centering operators for $Q$.
Define the operators %$(\g^P_x, x \in \R)$ on $\O$ by
 \begin{align*}
  \g^P_x & :=  \g^Q_{x/n, \ldots, x/n}, \quad x \in \R.
 \end{align*}
Then $(\g^P_x, x \in \R)$ and $\G^Q$ are translation and centering operators for all processes~$P_e$, $e \in \cE \cup \{0\}$, because
for $e \in \cE$ (for $e = 0$, namely $P_e = P$, the calculation is completely analogous), we obtain by shifting the underlying processes $Q^e$, $e \in \cE$,
 \begin{equation} \label{eq:G_IM:shifts on Pe 0}
 \begin{aligned}
  P_e(t) \circ \g^P_x
  & = p_2 \, t + Q^e(t) \circ \g^Q_{x/n, \ldots, x/n} + \sum_{f \neq e} Q^e(t-) \circ \g^Q_{x/n, \ldots, x/n} \\
  &  = p_2 \, t + Q^e(t) + \frac{x}{n} + \sum_{f \neq e} \big( Q^f(t-) + \frac{x}{n} \big) \\
  &  = P_e(t) + x, \\
 \end{aligned}  
 \end{equation} 
and analogously, by using the definition $Q^f(0-) = Q^f(0)$,
 \begin{equation} \label{eq:G_IM:shifts on Pe}
 \begin{aligned}  
  P_e(t) \circ \G^Q 
 % & = p_2 \, t + Q^e(t) \circ \G^Q + \sum_{f \neq e} Q^f(t-) \circ \G^Q \\
 % & = p_2 \, t + Q^e(t) + \sum_{f \neq e} Q^j(t-) - \big( p_2 \, 0 + Q^e(0) + \sum_{f \neq e} Q^j(0-) \big) \\
  & = P_e(t) - P_e(0).
 \end{aligned}  
 \end{equation}

Define the operators $\T^X_t$, $t \geq 0$, by
 \begin{align*}
   \T^X_t := \T^W_t \otimes (\g^P_{-L_t} \circ \T^Q_{\varrho_t}),
 \end{align*}
that is, for all $\o = (\o^W, \o^Q) \in \O$, 
 \begin{align*}
   \T^X_t(\o) = \big(  \hT^W_t(\o^W), \g^P_{-L_t(\o)} \big( \hT^Q_{\varrho_t(\o)}(\o^Q) \big)  \big).
 \end{align*}
In order not to complicate the notation even more, we will also write $\T^W$, $\T^Q$ for the lifts of these shift operators from $\O^W$, $\O^Q$ to $\O$:
For all $\o = (\o^W, \o^Q) \in \O$, the formulas $\T^W_t(\o) = \T^W_t(\o^W)$, $\T^Q_t(\o) = \T^Q_t(\o^Q)$ will be used implicitly in this section.
 
\begin{lemma} \label{lem:G_IM:shift operators for X}
 $(\T^X_t, t \geq 0)$ is a family of shift operators for $X$.
\end{lemma}
The proof is rather technical and can be found in Appendix~\ref{app:G_IM:shift operators}.

\subsection{Suitable Filtration for \texorpdfstring{$X$}{X}}\label{subsec:G_IM:filtration}
For the following results, we define for any collection $\sA_1$ of sets and any set $A_2$
the usual ``Cartesian product'' of families of sets
 \begin{align*}
  \sA_1 \times A_2 := \{ A_1 \times A_2 : A_1 \in \sA_1 \},
 \end{align*}
and analogously the set $A_2 \times \sA_1 = \{ A_2 \times A_1 : A_1 \in \sA_1 \}$.

In order to describe for any $t \geq 0$ the mapping
  \begin{align*} 
    X_t = \big( E(P_e(\varrho_t) - L_t, e \in \cE) \circ W_t , P(\varrho_t) - L_t + \abs{W_t} \big),
  \end{align*}
the ``information'' of $\sF^W_t$ and ``$\sF^Q_{\varrho_t}$'' is needed.
First of all, we must clarify what we mean by the latter $\sigma$-algebra, because $\varrho_t = P^{-1}(L_t)$ is certainly not an $(\sF^Q_t, t \geq 0)$-stopping time.
Following the general definition of a stopped $\sigma$-algebra $\sG_\t$, namely 
 \begin{align*}
    \sG_\t = \big\{ A \in \sG_\infty \, \big| \, \forall s \geq 0: ~ A \cap \{ \t \leq s\} \in \sG_s \big\},
 \end{align*}
we set for each $t \geq 0$
 \begin{align*}
   \sF_t := \big\{ A \in \sF^W_t \otimes \sF^Q_\infty \, \big| \, \forall s \geq 0: ~ A \cap \{ \varrho_t \leq s\} \in \sF^W_t \otimes \sF^Q_s \big\}.
 \end{align*}
It turns out that $\sF_t$ is just the stopped $\sigma$-algebra $\tsF^t_{\varrho_t}$
for the random time $\varrho_t$ and the filtration $(\tsF^t_s, s \geq 0)$ given by
 \begin{align*}
  \tsF^t_s := \sF^W_t \otimes \sF^Q_s, \quad s \geq 0.
 \end{align*}

For this definition to fit in the general theory of stopping times and in order to employ the basic results on usual stopped $\sigma$-algebras,
we show:

\begin{lemma} \label{lem:G_IM:rho is tF stopping time}
  For every $t \geq 0$, $\varrho_t$ is an $(\tsF^t_s, s \geq 0)$-stopping time, that is,
   \begin{align*}
    \forall s \geq 0: \quad \{ \varrho_t \leq s \} \in \sF^W_t \otimes \sF^Q_s.
   \end{align*}
\end{lemma}
\begin{proof}
 By using \ref{itm:G_IM:props pseudoinverse:inversion} of Lemma~\ref{lem:G_IM:props pseudoinverse},
 we see that for any $s, t \geq 0$,
  \begin{align*}
   \{ \varrho_t > s \} 
   & = \{ P^{-1}(L_t) > s \} 
     = \{ L_t > P(s) \}  \\
   & =  \bigcup_{q \in \Q_+} \big( \{ L_t > q \} \cap \{ q > P(s) \} \big)  ~ \in \sF^W_t \otimes \sF^Q_s. \qedhere
   \end{align*}
 %as $(L_t, t \geq 0)$ is adapted to $(\sF^W_t, t \geq 0)$ and $(P_s, s \geq 0)$ is adapted to $(\sF^Q_s, s \geq 0)$.
\end{proof}

% \begin{lemma} \label{lem:G_IM:rho is increasing}
%  The sequence $(\varrho_t, t \geq 0)$ is increasing.
% \end{lemma}
% \begin{proof}
%  As $(\hQ_t, t \geq 0)$ is a subordinator, the process $(P_t, t \geq 0)$ and thus $(P^{-1}_t, t \geq 0)$ are increasing.
%  $(L_t, t \geq 0)$ is increasing as well, so is $\big( \varrho_t = P^{-1}(L_t), t \geq 0 \big)$.
% \end{proof}

Thus, 
$\sF_{t_1} = \tsF^t_{\varrho_{t_1}} \subseteq \tsF^t_{\varrho_{t_2}}  = \sF_{t_2}$ for all $0 \leq t_1 \leq t_2$, 
by a well-known result on stopped $\sigma$-algebras (see, e.g., \cite[Theorem 1.3.5]{ChungWalsh05}), that is, $(\sF_t, t \geq 0)$ is a filtration.

\begin{theorem} \label{lem:G_IM:X is adapted to filtration F}
 $(X_t, t \geq 0)$ is adapted to the filtration $(\sF_t, t \geq 0)$.
\end{theorem}
\begin{proof}
 It is immediate from Lemma~\ref{lem:G_IM:rho is tF stopping time} that $(\varrho_t, t \geq 0)$ is adapted to $(\sF_t, t \geq 0)$, thus especially yielding
  \begin{align} \label{eq:G_IM:F_W subset of F}
   \forall t \geq 0: \quad \sF^W_t \times \O^Q \subseteq \sF_t.
  \end{align}
 An easy calculation then shows that for every $e \in \cE$, $u \geq 0$, the process
   \begin{align} \label{eq:G_IM:Q rho- is adapted to F}
    \text{$\big(Q^e \big( (\varrho_t - u) \vee 0 \big), t \geq 0 \big)$ is adapted to $(\sF_t, t \geq 0)$.}
   \end{align}
 
 Thus, for all $e \in \cE \cup \{0\}$, the process $\big( P_e(\varrho_t), t \geq 0 \big)$ is $(\sF_t, t \geq 0)$-adapted:
   This is evident for $P = P_0$ by setting $u = 0$ in~\eqref{eq:G_IM:Q rho- is adapted to F}, as for $t \geq 0$,
    \begin{align*}
      P(\varrho_t) = p_2 \, \varrho_t + \sum_{e \in \cE} Q^e(\varrho_t).
    \end{align*}
  For $e \in \sE$, $t \geq 0$,
    \begin{align*}
      Q^e(\varrho_t -) = \lim_{s \upuparrows \varrho_t} Q^e(s) = \lim_{n \rightarrow \infty} Q^e \big( (\varrho_t - \frac{1}{n}) \vee 0 \big)
    \end{align*}
  is $\sF_t$-measurable, and so
    \begin{align}\label{eq:Pe is adapted to F}
      P_e(\varrho_t) & = p_2 \, \varrho_t + Q^e(\varrho_t) + \sum_{f \neq e} Q^f(\varrho_t -) \quad \text{is $\sF_t$-measurable}.
    \end{align}
    
  This yields the result, as $W_t$ is $\sF_t$-measurable by \eqref{eq:G_IM:F_W subset of F}.\sqed
\end{proof}

\subsection{Strong Markov Property of \texorpdfstring{$(W,Q)$}{(W,Q)}} \label{subsec:G_IM:strong Markov of (W,Q)} 

In the construction of $X$, the process $P$ and, thus, the process $Q$ appear shifted by the random time $\varrho_t$.
In order to use Markov arguments when analyzing the process $X$ in the next subsections, 
we will need to transfer the strong Markov property of $\hQ$ to the part $Q$ of the combined process $(W, Q)$ and
then understand how the shifts $\T^W_t$ of $W$ and $\T^Q_{\varrho_t}$ of $Q$ act on this combined process.

The main idea is that, as we only shift the process $Q$ by $\varrho_t$, we only need to consider this part of the combined process.
Therefore, we introduce the filtration $(\bsF^Q_s, s \geq 0)$ by
 \begin{align*} 
   \bsF^Q_s := \sF^W_\infty \otimes \sF^Q_s, \quad s \geq 0.
  \end{align*}
It is immediate that
 \begin{align*} 
   \bsF^Q_\infty = \sF^W_\infty \otimes \sF^Q_\infty.
 \end{align*}
 
Surely, this filtration is large enough for the time shift $\varrho_t$, Lemma~\ref{lem:G_IM:rho is tF stopping time} gives:
 
\begin{lemma} \label{lem:G_IM:rho is barFQ stopping time}
 For all $t \geq 0$, $\varrho_t$ is an $(\bsF^Q_s, s \geq 0)$-stopping time.
\end{lemma}
% \begin{proof}
%  This is clear, as for all $s \geq 0$, we have by Lemma~\ref{lem:G_IM:rho is tF stopping time}:
%   \begin{align*} 
%     \{ \varrho_t \leq s \} & \in \sF^W_t \otimes \sF^Q_s \subseteq \sF^W_\infty \otimes \sF^Q_s = \bsF^Q_s.  \qedhere
%   \end{align*}
% \end{proof}

% As every translated stopping time is again a stopping time, the next result follows:
% 
% \begin{corollary} \label{cor:G_IM:rho+u is barFQ stopping time}
%  For all $t \geq 0$, $u \geq 0$, $\varrho_t + u$ is an $(\bsF^Q_s, s \geq 0)$-stopping time.
% \end{corollary}

The next two results show that this new filtration is, of course, larger than the actual filtrations needed,
which will be helpful for proving Markov properties later. They follow directly from the definitions of the involved $\sigma$-algebras:

\begin{lemma} \label{lem:G_IM:FW subset barFQ_tau}
 %For all $t \geq 0$ it is $\sF^W_\infty \otimes \O^Q \subseteq \bsF^Q_{\varrho_t}$.
 For any $(\bsF^Q_s, s \geq 0)$-stopping time $\t$, $\sF^W_\infty \times \O^Q \subseteq \bsF^Q_{\t}$.
\end{lemma}
% \begin{proof}
%  Let $A \in \sF^W_\infty$. Then, for any $s \geq 0$,
%    \begin{align*} 
%      A \times \O^Q \in \sF^W_\infty \otimes \sF^Q_s \subseteq \sF^W_\infty \otimes \sF^Q_\infty,
%    \end{align*}
%  and, as $\t$ is an $(\bsF^Q_s, s \geq 0)$-stopping time,
%    \begin{align*}
%      %(A \times \O^Q) \cap \{ \varrho_t \leq s \} \in \sF^W_\infty \otimes \sF^Q_s.
%      (A \times \O^Q) \cap \{ \t \leq s \} & \in \sF^W_\infty \otimes \sF^Q_s. \qedhere
%    \end{align*}    
% \end{proof}

\begin{lemma} \label{lem:G_IM:F subset barFQ_rho}
 For all $t \geq 0$, $\sF_t \subseteq \bsF^Q_{\varrho_t}$.
\end{lemma}
% \begin{proof}
%  Using the definitions of both $\sigma$-algebras, we get for $t \geq 0$,
%   \begin{align*}
%     \sF_t
%     & = \big\{ A \in \sF^W_t      \otimes \sF^Q_\infty \, | \, \forall s \geq 0: ~ A \cap \{ \varrho_t \leq s \} \in \sF^W_t \otimes \sF^Q_s \big\} \\     
%     & \subseteq \big\{ A \in \sF^W_\infty \otimes \sF^Q_\infty \, | \, \forall s \geq 0: ~ A \cap \{ \varrho_t \leq s \} \in \tsF_s \big\} \\
%     & = \bsF^Q_{\varrho_t}. \qedhere
%   \end{align*}
% \end{proof}

We are now able to transfer the Markov property and the strong Markov property from $\hQ$ to $Q$.
Then, due to their independence, are able to shift the processes~$W$ and~$Q$ by different time scales.
Observe that the following results are not really representing the canonical Markov properties which were defined and discussed in~subsection~\ref{subsec:Definitions Markov},
so we need to reiterate some standard proofs. Details can be found in Appendix~\ref{app:G_IM:strong Markov of (W,Q)}.

\begin{theorem} \label{theo:G_IM:Markov of (W,Q)}
 For all $g \in \cG$, $q \in \R^n$, $Y \in b \sF^0_\infty = b \sigma(W_r, Q_s, r,s \geq 0)$, $t \geq 0$,
  \begin{align*} 
    \EV_{g,q} \big( Y \circ \T^W_t \otimes \T^Q_{\varrho_t} \, \big| \, \sF_t \big) = \EV_{W_t, Q_{\varrho_t}} (Y).
  \end{align*}
\end{theorem}

For later use, we also need to introduce another coarser filtration:
 \begin{align*}
  \bbsF^Q_s := \O^W \times \sF^Q_s, \quad  s \geq 0.
 \end{align*}
Then we have the natural equivalent of %lemma~\ref{lem:G_IM:Q strongly Markovian barFQ} and 
Theorem~\ref{theo:G_IM:Markov of (W,Q)}:

\begin{theorem} \label{theo:G_IM:Q strongly Markovian barbarFQ}
 For all $g \in \cG$, $q \in \R^n$, $Y \in b \sF^0_\infty = b \sigma(W_r, Q_s, r,s \geq 0)$ and every stopping time $\t$ over $(\bbsF^Q_s, s \geq 0)$,
  \begin{align*} 
    \EV_{g,q} \big( Y \circ \id^W \otimes \T^Q_\t \, \big| \, \bbsF^Q_t \big) = \EV_{q, Q_\t} (Y).
  \end{align*}
\end{theorem}
% \begin{proof}
%  As usual, in regard to the MCT, it suffices to show this formula for 
%   \begin{align*}
%     Y = f_1(W_{r_1}) \, \cdots \, f_k(W_{r_k}) \, h_1(Q_{s_1}) \, \cdots \, h_l(Q_{s_l})
%   \end{align*}
%  with $k,l \in \N$, $f_1, \ldots, f_k \in b\sB(\R)$, $h_1, \ldots, h_l \in b\sB(\R)^{\otimes n}$, as well as $0 \leq r_1 < \cdots < r_k$, $0 \leq s_1 < \cdots < s_l$.
%  It is standard to extend lemma \ref{lem:G_IM:Q strongly Markovian barbarFQ} to 
%    \begin{align*}
%     \EV_{g,q} \big( h_1(Q_{s_1}) \, \cdots \, h_l(Q_{s_l}) \circ \T^Q_\t \,\big|\, \bbsF^Q_\t \big) = \EV_{g, Q_\t} \big( h_1(Q_{s_1}) \, \cdots \, h_l(Q_{s_l}) \big).
%   \end{align*}
%  Therefore, by setting $f^W := f_1(W_{r_1}) \, \cdots \, f_k(W_{r_k})$, $h^Q := h_1(Q_{s_1}) \, \cdots \, h_l(Q_{s_l})$
%  and using the same techniques as in the proof of lemma \ref{lem:G_IM:Q strongly Markovian barbarFQ},
%  it remains to compute:
%   \begin{align*} 
%     \EV_{g,q} \big( f^W \, h^Q \circ \id^W \otimes \T^Q_\t \, \big| \, \bbsF^Q_t \big) 
%     & = \EV_{g,q} \big( f^W \, h^Q \circ \T^Q_\t \, \big| \, \bbsF^Q_t \big)  \\
%     & = \EV_{g,q} \big( f^W \, \EV_{g,q} \big( h^Q \circ \T^Q_\t \, \big| \, \bsF^Q_t \big) \, \big| \, \bbsF^Q_t \big) \\
%     & = \EV_{g,q} \big( f^W \, \big| \, \bbsF^Q_t \big) \, \EV_{g, Q_\t} \big( h^Q \big) \\
%     & = \EV_{g,q} \big( f^W \big) \, \EV_{g, Q_\t} \big( h^Q \big) \\
%     & = \EV_{g, Q_\t} \big( f^W \, h^Q \big). \qedhere
%   \end{align*} 
% \end{proof}

Of course, the roles of $W$ and $Q$ can be interchanged in %lemmas 
%\ref{lem:G_IM:Q Markovian barFQ}, \ref{lem:G_IM:Q strongly Markovian barFQ}, \ref{lem:G_IM:Q strongly Markovian barbarFQ}
Theorem~\ref{theo:G_IM:Q strongly Markovian barbarFQ}, giving us for the filtration
 \begin{align*}
  \bbsF^W_t := \sF^W_t \times \O^Q, \quad  t \geq 0:
 \end{align*}

\begin{theorem} \label{theo:G_IM:W strongly Markovian barbarFW}
 For all $g \in \cG$, $q \in \R^n$, $Y \in b \sF^0_\infty = b \sigma(W_r, Q_s, r,s \geq 0)$ and every stopping time $\t$ over $(\bbsF^W_t, t \geq 0)$, 
  \begin{align*} 
    \EV_{g,q} \big( Y \circ \T^W_\t \otimes \id^Q  \, \big| \, \bbsF^W_t \big) = \EV_{W_\t, q} (Y).
  \end{align*}
\end{theorem}

\subsection{Strong Markov Property at \texorpdfstring{$\h_0$}{H0}} \label{subsec:G_IM:strong Markov at H0}

Let $\h_0 := \inf \{t \geq 0: X_t = 0 \}$ be the first entry time of $X$ in the vertex $0$.

We are going to show a strong Markov property at $\h_0$ next, which will be essential for the proof of both the Markov property and the strong Markov property of~$X$.
The Markovian behavior of $X$ at $\h_0$ should appear quite natural, because $X$ is just the underlying Walsh process $W$ until $\h_0 = \h^W_0$
(with $\h^W_0$ being the first entry time of~$W$ in~$0$), $W$ is strongly Markovian,
and the additional, independent parts of the subordinators only come into play after $\h_0$.

\begin{lemma} \label{lem:G_IM:X = W for t <= H0}
 It holds $X_t = W_t$ for all $t \leq \h_0$ and $\h_0 = \h^W_0$, $\PV_g$-a.s.\ for all $g \in \cG$.
\end{lemma}
\begin{proof}
  Let $g \in \cG$. Every identity in this proof will be meant $\PV_g = \PV_{(g,0)}$-a.s.\,.

  Because $(L_t, t \geq 0)$ only grows at $\{t \geq 0: W_t = 0\}$ and is continuous, we have $L_t = 0$ for all $t \leq \h^W_0$.
  The fact that $P$ starts at $0$ and is strictly increasing implies that 
  $P^{-1}(0) = 0$, so we get %$P_e P^{-1}(L_t) = P_i(0)$ for all $t \leq \h^W_0$, so we get
       \begin{align*} 
         \forall e \in \cE \cup \{0\}, t \leq \h^W_0: ~ P_e P^{-1}(L_t) = P_e(0) = 0.
       \end{align*}
  By checking the definition of $X$, it is immediate that 
       \begin{align*}
         \forall t \leq \h^W_0: \quad  X_t = W_t.
       \end{align*}
  As $X_t = W_t \neq 0$ for all $t < \h^W_0$ and $X_{\h^W_0} = W_{\h^W_0} = 0$, this proves $\h_0 = \h^W_0$.\sqed
\end{proof}

\begin{corollary} \label{cor:G_IM:dist of X = W for t <= H0}
 The processes $(X_{t \wedge \h_0}, t \geq 0)$ and $(W_{t \wedge \h^W_0}, t \geq 0)$ 
 have the same finite-dimensional distributions with respect to $\PV_{g}$ for all $g \in \cG$.
\end{corollary}

Before we continue with our developments towards the strong Markov property, we remark the following relation for later use:
\begin{lemma} \label{lem:G_IM:laplace of H0 for X}
 For all $\a > 0$, $t \geq 0$, $\o \in \O$,
  \begin{align*}
   \EV_{X_t(\o),0} \big( e^{-\a \h_0} \big)
   = \EV_{W_t(\o),0} \big( e^{-\a L^{-1}(\eta_t(\o))} \big).
  \end{align*}
\end{lemma}
\begin{proof}
 If $\eta_t(\o) \neq 0$, then there is exactly one $e \in \cE$ with $\eta^e_t(\o) > 0$ by Lemma~\ref{lem:G_IM:properties eta},
 so the definition of $X$ and Lemma~\ref{lem:G_IM:X = W for t <= H0} yield
   \begin{align*}
    \EV_{X_t(\o),0} \big( e^{-\a \h_0} \big)
    = \EV_{(e, \eta_t(\o) + \abs{W_t(\o)}),0} \big( e^{-\a \h^W_0} \big).
  \end{align*} 
 The first hitting time and the local time of the Walsh process~$W$ at the vertex correspond to
 the respective entities of the underlying (reflecting) Brownian motion~$B$ at the origin (see Theorem~\ref{theo:G_WP:WP path properties}),
 so \eqref{eq:BM passage time} and Lemma~\ref{lem:B_LT:laplace of local time inverse} give
   \begin{align*}
    \EV_{(e, \eta_t(\o) + \abs{W_t(\o)}),0} \big( e^{-\a \h^W_0} \big)
    & = \EV^B_{\eta_t(\o) + \abs{W_t(\o)}} \big( e^{-\a \h^B_0} \big) \\
    & = e^{-\sqrt{2\a} \, (\eta_t(\o) + \abs{W_t(\o)})} \\
    & = \EV^B_{\abs{W_t(\o)}} \big( e^{-\a L^{-1}(\eta_t(\o))} \big) \\
    & = \EV_{W_t(\o),0} \big( e^{-\a L^{-1}(\eta_t(\o))} \big).
  \end{align*}  
  
 If $\eta_t(\o) = 0$, then
  \begin{align*}
    \EV_{X_t(\o),0} \big( e^{-\a \h_0} \big)
    = \EV_{W_t(\o),0} \big( e^{-\a \h^W_0} \big),
  \end{align*}
 which completes the proof, as $L^{-1}(0) = \h^W_0$.\sqed
\end{proof}

We prepare the strong Markov property of~$X$ at~$\h_0$ with the following result:

\begin{lemma} \label{lem:G_IM:shift of X by H0}
  For all $g \in \cG$, $t \geq 0$, $f \in b\sB(\cG)$, $k \in \N$, $f_1, \ldots, f_k \in b\sB(\cG)$ and $0 \leq t_1 < \cdots < t_k$, the following holds true with
  $J := f_1(X_{t_1 \wedge \h_0}) \, \cdots \, f_k(X_{t_k \wedge \h_0})$:
  \begin{align*}
    \EV_{g} \big( f(X_{t + \h_0}) \cdot J \big) = \EV_{g} \big( \EV_{X_{\h_0}} \big( f(X_t) \big) \cdot J \big).
  \end{align*} 
\end{lemma}
\begin{proof}
 Consider the process $X$ shifted by $\h_0$, that is
  \begin{align*} 
    X_{t + \h_0} = \big( e(L_{t + \h_0}) \circ W_{t + \h_0} , P P^{-1}(L_{t + \h_0}) - L_{t + \h_0} + \abs{W_{t + \h_0}} \big).
  \end{align*}
 As $\h_0 = \h^W_0$ and $L_{\h_0} = L_{\h^W_0} = 0$, we have
  \begin{align*}
    L_{t + \h_0} = L_{t + \h_0} - L_{\h_0} = L_t \circ \T^W_{\h^W_0},
  \end{align*}
 and therefore
  \begin{align*}
    X_{t + \h_0} = X_t \circ \big( \T^W_{\h^W_0} \times \id^Q \big),
  \end{align*}
 so shifting by $\h_0$ does not shift $Q$. %(this is also clear by the definition of $X$ or by looking at its shift operators).
 Lemma~\ref{lem:G_IM:X = W for t <= H0} then gives
  \begin{align*}
   & \EV_{g} \big( f(X_{t + \h_0}) \, f_1(X_{t_1 \wedge \h_0}) \, \cdots \, f_k(X_{t_k \wedge \h_0}) \big) \\
   & = \EV_{g} \big( f(X_{t + \h_0}) \, f_1(W_{t_1 \wedge \h^W_0}) \, \cdots \, f_k(W_{t_k \wedge \h^W_0}) \big) \\
   & = \EV_{g} \big( f_1(W_{t_1 \wedge \h^W_0}) \, \cdots \, f_k(W_{t_k \wedge \h^W_0}) \, \EV_{g} \big( f(X_t) \circ (\T^W_{\h^W_0} \otimes \id^Q) \, \big| \, \bbsF^W_{\h^W_0} \big) \big),
  \end{align*}
 with $\bbsF^W_t = \sF^W_t \times \O^Q$, $t \geq 0$, as defined at the end of subsection \ref{subsec:G_IM:strong Markov of (W,Q)}.
 Using the strong Markov property of Theorem~\ref{theo:G_IM:W strongly Markovian barbarFW} of $W$ with respect to $(\bbsF^W_t, t \geq 0)$ 
 for the stopping time $\h^W_0$, the inner conditional expectation becomes
  \begin{align*}
   \EV_{g,0} \big( f(X_t) \circ (\T^W_{\h^W_0} \otimes \id^Q) \, \big| \, \bbsF^W_{\h^W_0} \big) 
   = \EV_{W_{\h^W_0}, 0} \big( f(X_t) \big)
   = \EV_{0,0} \big( f(X_t) \big),
  \end{align*}
 which completes the proof by using once again Lemma~\ref{lem:G_IM:X = W for t <= H0}, yielding
  \begin{align*}
   & \EV_{g} \big( f(X_{t + \h_0}) \, f_1(X_{t_1 \wedge \h_0}) \, \cdots \, f_k(X_{t_k \wedge \h_0}) \big) \\
   & = \EV_{g} \big( \EV_{X_{\h_0}} \big( f(X_t) \big) \, f_1(X_{t_1 \wedge \h_0}) \, \cdots \, f_k(X_{t_k \wedge \h_0}) \big). \qedhere
  \end{align*} 
\end{proof}

We are ready for the first main result, namely the strong Markov property at~$H_0$,
which we would like to prove with the help of Galmarino's theorem~\eqref{eq:galmarino}. 
However, there are no stopping operators for~$X$ available on the constructed space~$\O$, as stopping the process at the vertex~$0$
would cause the local time to explode. Therefore, we need to switch to the path space realization
$(Y_t, t \geq 0)$ of $X$. As the process~$X$ is right continuous 
%by Theorem~\ref{theo:G_IM:continuity process X}, we define the process $Y_t(o) := \o(t)$, $\o \in \O$, $t \geq 0$, 
and continuous inside the edges by Theorem~\ref{theo:G_IM:continuity process X}, we are able to construct 
the canonical process $Y_t(\o) := \o(t)$, $\o \in \O^Y$, $t \geq 0$, on the path space 
 \begin{align*}
  \O^Y := \big\{ \o \colon \R_+ \rightarrow \cG \,\big|\, & \text{$\o$ right continuous}
    \wedge \text{$\forall t \geq 0$ with $\o(t) \neq 0$: $\o$ is} \\
    & \text{continuous on $[t, t_0]$, with $t_0 := \inf \{ s \geq t: \o(s) = 0 \}$} 
   \big\}, 
 \end{align*}
equipped with canonical filtration $\sF^Y_t = \sigma(Y_s, s \leq t)$, $t \geq 0$,
and mapping operator 
 \begin{align*}
  \Phi \colon \O \rightarrow \O^Y, \quad \o \mapsto \Phi(\o) \quad \text{with} \quad \Phi(\o)(t) := X_t(\o), ~ t \geq 0.
 \end{align*}
As $X_t = Y_t \circ \Phi$, we have for the first entry time $\h^Y_0$ of $Y$ in $0$: 
 \begin{align*}
   \h^Y_0 \circ \Phi = \inf \{t \geq 0: Y_t \circ \Phi = 0 \} = \h_0.
 \end{align*}
 
%The space $\O^Y$ admits natural stopping operators: The natural stopping operators of the path space as in example
%\ref{ex:A_SM:stopping operators on path space}), can be restricted to $\O^Y$, as the stopped paths $\o( \,\cdot\, \wedge t)$
%fulfill the given conditions of $\O^Y$. 
The space $\O^Y$ admits the natural shift and stopping operators
 \begin{align*}
  \T_t(\o) := \o( \,\cdot\, + t), \quad \a_t(\o) :=\o( \,\cdot\, \wedge t), \quad t \geq 0, ~ \o \in \O^Y,
 \end{align*}
as both shifted and stopped paths admit the conditions on $\O^Y$.
Furthermore, we have the following:

\begin{lemma} \label{lem:G_IM:H_0Y is stopping time}
  $\h^Y_0$ is a stopping time over $(\sF^Y_t, t \geq 0)$.
\end{lemma}
\begin{proof}
  By the definition of $\O^Y$, the canonical coordinate process $Y$ is right continuous on $\R_+$ 
  and continuous on $[0, t_0]$, with $t_0 = \inf \{ s \geq 0: Y_s = 0 \} = \h^Y_0$. 
  As~$\{0\}$ is a closed subset of the Polish space $\cG$, a close examination of the proof of~\cite[Theorem~49.5]{Bauer96}
  yields that $\h^Y_0$ is a stopping time over the natural ``raw'' filtration $(\sF^Y_t, t \geq 0)$ (the continuity of the process is only needed up to the first entry time).\sqed
%  By definition of $\O^Y$, $Y_t(\o) = \o(t)$ is continuous on $[0, t_0]$, with $t_0 = \inf \{ s \geq 0: Y_s = v \} = \h^Y_0$. 
%  Therefore, the stopped process $Y_{\,\cdot\, \wedge \h^Y_0}$ (FOR THIS TO MAKE SENSE, $\h^Y_0$ must already be measurable!!!) is continuous.
%  But as $Y$ is right continuous and $\{0\}$ is closed, $Y_{\h^Y_0} = v$, so
%   \begin{align*}
%    \h^Y_0
%    & = \inf \{ s \geq 0: Y_s = v \}
%    & = \inf \{ s \in [0, \h^Y_0]: Y_s = v \}
%    & = \inf \{ s \geq 0: Y_{s \wedge \h^Y_0} = v \}
%   \end{align*}
%  is also the first entry time of the continuous process $Y_{\,\cdot\, \wedge \h^Y_0}$ into the closed set $\{0\}$,
%  and therefore by \cite[Theorem 49.5]{Bauer96} a stopping time for the filtration generated by $Y_{\,\cdot\, \wedge \h^Y_0}$.
% 
%  We are applying the first part of Galmarino's theorem \eqref{eq:galmarino}: 
%  Let $\o_1, \o_2 \in \O^Y$, $t \geq 0$, with $\a_t(\o_1) = \a_t(\o_2)$ and $t_1 := \h^Y_0(\o_1) \leq t$.
%  We need to show $\h^Y_0(\o_1) = \h^Y_0(\o_2)$.
%  
%  As $\o_1$ is right continuous and $\{0\}$ is closed, $\o_1(t_1) = v$. 
%  Because both stopped paths $\o_1( \,\cdot\, \wedge t) = \o_2( \,\cdot\, \wedge t)$ are identitical,
%   this implies $\o_2(t_1) = v$, i.e.\ $\h^Y_0(\o_2) \leq t_1$.
%  Assume $t_2 := \h^Y_0(\o_2) < t_1$. 
%  But then by the same argument as above, $\o_1(t_2) = \o_2(t_2) = v$, therefore 
%  $t_1 = \h^Y_0(\o_1) \leq t_2$, which contradicts the assumption. 
%  Therefore, $\h^Y_0(\o_1) = \h^Y_0(\o_2)$.
\end{proof}

Therefore, we are able to apply Galmarino's theorem in the context of $Y$:

\begin{theorem} \label{theo:G_IM:strong Markov of Y}
 $(Y_t, t \geq 0)$ is strongly Markovian with respect to $\big( (\sF^Y_t)_{t \geq 0}, \h^Y_0 \big)$.
\end{theorem}
\begin{proof}
 %We can (and will) employ Galmarino's theorem (see theorem \ref{x} and the discussion following it) in this context.
  Galmarino's theorem \eqref{eq:galmarino} 
 asserts that $\sF^Y_{\h^Y_0} = \sigma(Y_{t \wedge \h^Y_0}, t \geq 0)$. It is therefore sufficient to show 
 the equality
  \begin{align*}
    \EV_{g,0} \big( f(Y_{t + \h^Y_0}) \cdot J \big) = \EV_{g,0} \big( \EV_{Y_{\h^Y_0}} \big( f(Y_t) \big) \cdot J \big)
  \end{align*}
 for all $g \in \cG$, $t \geq 0$, $f \in b\sB(\cG)$, $k \in \N$, $f_1, \ldots, f_k \in b\sB(\cG)$, $0 \leq t_1 < \cdots < t_k$, with
  \begin{align*}
    J := f_1(Y_{t_1 \wedge \h^Y_0}) \, \cdots \, f_k(Y_{t_k \wedge \h^Y_0}).
  \end{align*}
 But this is immediately proved by Lemma~\ref{lem:G_IM:shift of X by H0}, as
   \begin{align*}
    & \EV^Y_{g,0} \big( f(Y_{t + \h^Y_0}) \, f_1(Y_{t_1 \wedge \h^Y_0}) \, \cdots \, f_k(Y_{t_k \wedge \h^Y_0} \big) \\
    & = \EV_{g,0} \big( f(X_{t + \h_0}) \, f_1(X_{t_1 \wedge \h_0}) \, \cdots \, f_k(X_{t_k \wedge \h_0}) \big) \\
    & = \EV_{g,0} \big( \EV_{X_{\h_0}} \big( f(X_t) \big) \, f_1(X_{t_1 \wedge \h_0}) \, \cdots \, f_k(X_{t_k \wedge \h_0}) \big) \\
    & = \EV^Y_{g,0} \big( \EV_{Y_{\h^Y_0}} \big( f(Y_t) \big) \, f_1(Y_{t_1 \wedge \h^Y_0}) \, \cdots \, f_k(Y_{t_k \wedge \h^Y_0}) \big). \qedhere
  \end{align*}
\end{proof}

\subsection{Markov Property of \texorpdfstring{$X$}{X}}\label{subsec:G_IM:Markov property of X}

Next, we need to prepare the proof of the Markov property of $X$ with respect to ${(\sF_t, t \geq 0)}$ by analyzing the action of the time shift $(\T^X_t, t \geq 0)$,
as defined in subsection~\ref{subsec:G_IM:shift operators}, on all of the underlying components of~ $X$
(recall the definitions of subsection~\ref{subsec:G_IM:definitions}). Let $t \geq 0$ be fixed in this subsection.

For each $e \in \sE \cup \{0\}$, we define the increments of the processes $P_e$ and $Q^e$ shifted by $\varrho_t = P^{-1}(L_t)$ for all times $s \geq 0$ by
 \begin{align*} 
  \pP_e(s) & := P_e(s + \varrho_t) - P(\varrho_t) = P_e(s + P^{-1}(L_t)) - P P^{-1}(L_t), \\
  \pQ^e(s) & := Q^e(s + \varrho_t) - Q^e(\varrho_t), 
 \end{align*}
as well as the centered processes by 
 \begin{align*}
  \nQ^e(s) := Q^e(s) - Q^e(0) = Q \circ \G (s), \\
  \nP(s) := P(s) - P(0) = P \circ \G(s).
 \end{align*}
We notice (recall equation~\eqref{eq:G_IM:shifts on Pe}) that
\begin{align}
 \pP = \nP \circ \T^Q_{\varrho_t} = P \circ \G \circ \T^Q_{\varrho_t},
 \quad
 \pQ^e = \nQ^e \circ \T^Q_{\varrho_t} = Q^e \circ \G \circ \T^Q_{\varrho_t},
\end{align}
and that the processes $\pP_e$, $e \in \cE \cup \{0\}$, and $\nP$ are strictly increasing as the underlying processes are (see Remark~\ref{rem:G_IM:behavior process X}).

A detailed examination of the excursion times $(\eta^e_t, t \geq 0)$, which is done in Appendix~\ref{app:G_IM:shift of excursion times},
yields the following:

\begin{theorem} \label{theo:G_IM:shift for eta}
 For all $\o \in \O$, $s \geq 0$, $e \in \cE \cup \{0\}$,
 \begin{align*}
  \eta^e_{t+s}(\o)
   & = \big( \nP_e \nP^{-1} \big( L_s - \eta_t(\o) \big) - \big( L_s - \eta_t(\o) \big) \big) \circ \T^W_t \otimes \T^Q_{\varrho_t} (\o) \\
   & = \big( P_e P^{-1} \big( L_s - \eta_t(\o) \big) - \big(L_s - \eta_t(\o) \big) \big) \circ \big( \id^W \otimes \G^Q \big) \circ \big( \T^W_t \otimes \T^Q_{\varrho_t} \big) (\o) \\
   & = \eta^e_s \circ \big( \id^W \otimes \big( \g^P_{\eta_t(\o)} \circ \G^Q \big) \big) \circ \big( \T^W_t \otimes \T^Q_{\varrho_t} \big) (\o).
 \end{align*}
\end{theorem}

Insertion into the definition of~$X$ gives:

\begin{corollary} \label{cor:G_IM:shift on X by Theta}
 For all $\o \in \O$, $s \geq 0$,
 \begin{align*}
   X_{t+s}(\o) = X_s \circ \big( \id^W \otimes \big( \g^P_{\eta_t(\o)} \circ \G^Q \big) \big) \circ \big( \T^W_t \otimes \T^Q_{\varrho_t} \big) (\o).
 \end{align*}
\end{corollary}

\begin{corollary} \label{cor:G_IM:shift on eta}
 For all $s \geq 0$ with $L_s \circ \T^W_t < \eta_t$, $e \in \cE \cup \{0\}$,
 \begin{align*}
  \eta^e_{t+s} 
  & = \eta^e_t - L_s \circ \T^W_t
 \end{align*}
 holds $\PV_g$-a.s.\ for every $g \in \cG$.
\end{corollary}

We are now able to prove the Laplace-transformed version~\eqref{eq:Markov property (resolvent)} of the Markov property for~$X$.
We start with the decomposition of its resolvent at $\h^X_0$:

\begin{lemma} \label{lem:G_IM:Markov decomposition of X}
 For all $g \in \cG$, $f \in b\sB(\cG)$, $t \geq 0$, 
 \begin{align*}
  & \EV_{g,0} \Big( \int_0^\infty e^{-\alpha s} \, f(X_{t+s}) \, ds \, \big| \, \sF_t \Big) &  \\
  & = \EV_{X_t,0} \Big( \int_0^{\h^X_0} e^{-\alpha s} \, f( X_s )  \, ds \Big)
    + \EV_{X_t,0} \Big( e^{-\a \h^X_0} \, \EV_{X_{\h^X_0},0} \Big( \int_{0}^\infty e^{-\alpha s} \, f( X_s )  \, ds \Big) \Big).
 \end{align*}
\end{lemma}
\begin{proof}
 We decompose the integral inside the conditional expectation at the end of the first excursion, that is
  \begin{align*}
  \EV_{g,0} \Big( \int_0^\infty e^{-\alpha s} \, f(X_{t+s}) \, ds \, \big| \, \sF_t \Big) = I_1 + I_2,
  \end{align*}
 with
  \begin{align*}
   I_1 := \EV_{g,0} \Big( \int_0^\infty \1_{\{L_s \circ \T^W_t < \eta_t \}} \, e^{-\alpha s} \, f(X_{t+s}) \, ds \, \big| \, \sF_t \Big) \\
   I_2 := \EV_{g,0} \Big( \int_0^\infty \1_{\{L_s \circ \T^W_t \geq \eta_t \}} \, e^{-\alpha s} \, f(X_{t+s}) \, ds \, \big| \, \sF_t \Big).
 \end{align*}
 
 For the part of the current excursion (if there is one),
 we compute % on $\{ \eta_t > 0 \}$ (otherwise, this part vanishes):
   \begin{align*} 
   %& \EV_{g,0} \Big( \int_0^\infty \1_{\{L_s \circ \T^W_t < \eta_t \}} \, e^{-\alpha s} \, f(X_{t+s}) \, ds \, \big| \, \sF_t \Big) (\o) & \\
   I_1 & = \EV_{g,0} \Big( \EV_{g,0} \Big(  \int_0^\infty \1_{\{L_s \circ \T^W_t < \eta_t \}} \, e^{-\alpha s} \, 
                        f \big( E(\eta^e_t - L_s \circ \T^W_t , e \in \cE) \circ  (W_s \circ \T^W_t), & \\
   & \hspace{12em}         \eta_t - L_s \circ \T^W_t  + \abs{W_s} \circ \T^W_t \big) \, ds \,
                        \big| \, \sF^W_t \otimes \sF^Q_\infty \Big) \, \big| \, \sF_t \Big), & 
   \end{align*}
   where we used $\sF_t \subseteq \sF^W_t \otimes \sF^Q_\infty$ (by definition of $\sF_t$) and Corollary~\ref{cor:G_IM:shift on eta}
   for the reduction of the shifted excursion times $\eta^e_{t+s}$ to $\eta^e_t$.
   The Markov property of $W$ with respect to $(\sF^W_t, t \geq 0)$ now gives
   \begin{align*}
%    & \EV_{g,0} \Big( \EV_{g,0} \Big(  \int_0^\infty \1_{\{L_s \circ \T^W_t < \eta_t \}} \, e^{-\alpha s} \, 
%                         f \big( E(\eta^e_t - L_s \circ \T^W_t , e \in \cE) \circ  (W_s \circ \T^W_t), & \\
%    & \hspace{12.5em}         \eta_t - L_s \circ \T^W_t  + \abs{W_s} \circ \T^W_t \big) \, ds \,
%                         \big| \, \sF^W_t \otimes \sF^Q_\infty \Big) \, \big| \, \sF_t \Big) (\o) & \\
  % & \EV_{g,0} \Big( \int_0^\infty \1_{\{L_s \circ \T^W_t < \eta_t \}} \, e^{-\alpha s} \, f(X_{t+s}) \, ds \, \big| \, \sF_t \Big) (\o) & \\
   I_1 & = \EV_{g,0} \Big( \EV_{W_t(\, \cdot \,), 0} \Big( \int_0^\infty \1_{\{L_s < \eta_t(\, \cdot \,) \}} \, e^{-\alpha s} \, 
                        f \big( E(\eta^e_t(\, \cdot \,) - L_s, e \in \cE) \circ W_s, \\
   & \hspace{19.65em}      \eta_t(\, \cdot \,) - L_s  + \abs{W_s} \big) \, ds \Big) \, \big| \, \sF_t \Big), & 
   \end{align*}
 where the auxiliary arguments ``$(\, \cdot \,)$'' are meant to be variables of the function inside $\EV_{g,0} \big( \cdots \,|\, \sF_t \big)$,\footnote{That is,
    $I_1  = \EV_{g,0} \big( Y \, | \, \sF_t \big)$ with
  \begin{align*}
   Y(\o) := \EV_{W_t(\o), 0} \Big( \int_0^\infty \1_{\{L_s < \eta_t(\o) \}} \, e^{-\alpha s} \, 
                        f \big( E(\eta^e_t(\o) - L_s, e \in \cE) \circ W_s, 
           \eta_t(\o) - L_s  + \abs{W_s} \big) \, ds \Big).
  \end{align*}
  }
 due to the measurability of $\eta^e_t$ (see equation~\eqref{eq:Pe is adapted to F}) with respect to the $\sigma$-algebra $\sF^W_t \otimes \sF^Q_\infty$  
 (cf.~\cite[Exercise~6.12]{Sharpe88}, which is analogously provable for Markov processes and deterministic shifts).
 Adaption of $W$ to $(\sF_t, t \geq 0)$ now trivializes the conditional expectation,
 and the decomposition $\{ \eta_t > 0 \} = {\biguplus_{e \in \cE} \{ \eta^e_t > 0 \}}$ by Lemma~\ref{lem:G_IM:properties eta}
 (as the whole integral vanishes for $\eta_t = 0$)
 together with the relation $\{ L_s < \eta_t(\o) \} = \{ s < L^{-1}_- (\eta_t(\o)) \}$ 
 for the left-continuous pseudo-inverse $L^{-1}_-$ of $L$ (see~Lemma~\ref{lem:G_IM:props pseudoinverse (cont. case)})
 yields
   \begin{align*}
%    & \EV_{g,0} \Big( \EV_{W_t(\, \cdot \,), 0} \Big( \int_0^\infty \1_{\{L_s < \eta_t(\, \cdot \,) \}} \, e^{-\alpha s} \, 
%                         f \big( E(\eta^e_t(\, \cdot \,) - L_s, e \in \cE) \circ W_s, \\
%    & \hspace{17.9em}      \eta_t(\, \cdot \,) - L_s  + \abs{W_s} \big) \, ds \Big) \, \big| \, \sF_t \Big) (\o) & \\
 %  & \EV_{g,0} \Big( \int_0^\infty \1_{\{L_s \circ \T^W_t < \eta_t \}} \, e^{-\alpha s} \, f(X_{t+s}) \, ds \, \big| \, \sF_t \Big) (\o) & \\
  I_1(\o) & = \sum_{e \in \cE} \1_{\{\eta^e_t(\o) > 0 \}} \,
                        \EV_{W_t(\o), 0} \Big( \int_0^{L_-^{-1}(\eta_t(\o))} e^{-\alpha s} \, 
                        f \big( e, \eta_t(\o) + \abs{W_{s}} - L_s \big)  \, ds \Big). & 
   \end{align*}
 By employing L\'evy's characterization of the local time and the distribution of its inverse, as examined in 
 Lemmas~\ref{lem:B_LT:lc rc inverse} and~\ref{lem:B_LT:extension to levys char},
 applied to the radial part $\abs{W}$ of Walsh Brownian motion (see Theorem~\ref{theo:G_WP:WP path properties}),
 and then using Lemma~\ref{lem:G_IM:X = W for t <= H0} as well as the definition of $X$, we conclude that 
   \begin{equation} \label{eq:G_IM:Markov decomposition of X, part 1}
   \begin{aligned}
%    & \sum_{e \in \cE} \1_{\{\eta^e_t(\o) > 0 \}} \,
%                         \EV_{W_t(\o), 0} \Big( \int_0^{L_-^{-1}(\eta_t(\o))} e^{-\alpha s} \, 
%                         f \big( e, \eta_t(\o) + \abs{W_{s}} - L_s \big)  \, ds \Big) & \\
  % & \EV_{g,0} \Big( \int_0^\infty \1_{\{L_s \circ \T^W_t < \eta_t \}} \, e^{-\alpha s} \, f(X_{t+s}) \, ds \, \big| \, \sF_t \Big) (\o) & \\
   I_1 & = \sum_{e \in \cE} \1_{\{\eta^e_t > 0 \}} \,
                        \EV_{(e, \eta_t + \abs{W_t}), 0} \Big( \int_0^{\h^W_0} e^{-\alpha s} \, f( W_s )  \, ds \Big) & \\
   & = \1_{\{\eta_t > 0 \}} \, \EV_{X_t, 0} \Big( \int_0^{\h^X_0} e^{-\alpha s} \, f( X_s )  \, ds \Big). & 
  \end{aligned}  
  \end{equation}
 In summary, we have shown that---with the knowledge of the process' history---the part of the shifted first excursion (if there is one currently running) equals the first non-shifted excursion, in case the process is 
 restarted at current state of the process.
%  \begin{equation} \label{eq:G_IM:Markov decomposition of X, part 1}
%  \begin{aligned}  
%   & \EV_{g,0} \Big( \int_0^\infty \1_{\{L_s \circ \T^W_t < \eta_t \}} \, e^{-\alpha s} \, f(X_{t+s}) \, ds \, \big| \, \sF_t \Big) \\
%   & = \1_{\{\eta_t > 0 \}} \, \EV_{X_t, 0} \Big( \int_0^{\h^X_0} e^{-\alpha s} \, f( X_s )  \, ds \Big). &
%  \end{aligned}  
%  \end{equation}
 
 Turning to the part after the first excursion, we get by the definition of $X_{t+s}$:
  \begin{align*} 
   %& \EV_{g,0} \Big( \int_0^\infty \1_{\{L_s \circ \T^W_t \geq \eta_t \}} \, e^{-\alpha s} \, f(X_{t+s}) \, ds \, \big| \, \sF_t \Big) (\o) & \\
    I_2 
    = \EV_{g,0} \Big( \int_0^\infty \1_{\{L_s \circ \T^W_t \geq \eta_t \}} \, e^{-\alpha s} \, 
                        f \big( E(\eta^e_{t+s}, e \in \cE) \circ W_{t+s}, \eta_{t+s} + \abs{W_{t+s}} \big) \, ds \, \big| \, \sF_t \Big). 
  \end{align*}
  Theorem~\ref{theo:G_IM:shift for eta} reduces the shifted excursion times $\eta^e_{t+s}$ to $\eta^e_t$ with the help of shifts and centerings of 
  the underlying processes, thus yielding
%   %%\disableHboxWarning
%   \begin{align*}
%   % & \EV_{g,0} \Big( \int_0^\infty \1_{\{L_s \circ \T^W_t \geq \eta_t \}} \, e^{-\alpha s} \, f(X_{t+s}) \, ds \, \big| \, \sF_t \Big) (\o) & \\
% %    & \EV_{g,0} \Big( \int_0^\infty \1_{\{L_s \circ \T^W_t \geq \eta_t \}} \, e^{-\alpha s} \, 
% %                         f \big( E(\eta^e_{t+s}, e \in \cE) \circ W_{t+s}, \eta_{t+s} + \abs{W_{t+s}} \big) \, ds \, \big| \, \sF_t \Big) (\o) & \\
%    I_2 & = \EV_{g,0} \Big( \Big( \int_0^\infty \1_{\{L_s \geq \eta_t(\, \cdot \,) \}} \, e^{-\alpha s} \, 
%                         f \big( E \big( P_e P^{-1} (L_s - \eta_t(\, \cdot \,)) - (L_s - \eta_t(\, \cdot \,)), e \in \cE \big) \circ W_{s}, & \\
%    & \hspace{14.5em}            P P^{-1} (L_s - \eta_t(\, \cdot \,)) - (L_s - \eta_t(\, \cdot \,))  + \abs{W_{s}} \big)  \, ds \Big) & \\
%    & \hspace{20em} \circ \big( \id^W \otimes \G^Q \big) \circ \big( \T^W_t \otimes \T^Q_{\varrho_t} \big) (\, \cdot \,) \, \big| \, \sF_t \Big), &
%    \end{align*}
%    %%\enableHboxWarning
     \begin{align*}
  % & \EV_{g,0} \Big( \int_0^\infty \1_{\{L_s \circ \T^W_t \geq \eta_t \}} \, e^{-\alpha s} \, f(X_{t+s}) \, ds \, \big| \, \sF_t \Big) (\o) & \\
%    & \EV_{g,0} \Big( \int_0^\infty \1_{\{L_s \circ \T^W_t \geq \eta_t \}} \, e^{-\alpha s} \, 
%                         f \big( E(\eta^e_{t+s}, e \in \cE) \circ W_{t+s}, \eta_{t+s} + \abs{W_{t+s}} \big) \, ds \, \big| \, \sF_t \Big) (\o) & \\
   I_2 = \EV_{g,0} \Big( \Big( \int_0^\infty & \1_{\{L_s \geq \eta_t(\, \cdot \,) \}} \, e^{-\alpha s} \\
                     &   f \big( E \big( P_e P^{-1} (L_s - \eta_t(\, \cdot \,)) - (L_s - \eta_t(\, \cdot \,)), e \in \cE \big) \circ W_{s}, \\
              &  \hspace*{2.4em} P P^{-1} (L_s - \eta_t(\, \cdot \,)) - (L_s - \eta_t(\, \cdot \,))  + \abs{W_{s}} \big)  \, ds \Big)  \\
   & \hspace*{-2.7em} \circ \big( \id^W \otimes \G^Q \big) \circ \big( \T^W_t \otimes \T^Q_{\varrho_t} \big) (\, \cdot \,) \, \big| \, \sF_t \Big), 
   \end{align*}
    where the auxiliary arguments ``$(\, \cdot \,)$'' again represent the variable of the function inside $\EV_{g,0} \big( \cdots \,|\, \sF_t \big)$.
    Employing the Markov property of $(W,Q)$ with respect to~$(\sF_t, t \geq 0)$, as shown in Theorem~\ref{theo:G_IM:Markov of (W,Q)}, gives 
%    %%\disableHboxWarning 
%    \begin{align*}  
%  %  & \EV_{g,0} \Big( \int_0^\infty \1_{\{L_s \circ \T^W_t \geq \eta_t \}} \, e^{-\alpha s} \, f(X_{t+s}) \, ds \, \big| \, \sF_t \Big) (\o) & \\
% %    & \EV_{g,0} \Big( \Big( \int_0^\infty \1_{\{L_s \geq \eta_t(\, \cdot \,) \}} \, e^{-\alpha s} \, 
% %                         f \big( E \big( P_e P^{-1} (L_s - \eta_t(\, \cdot \,)) - (L_s - \eta_t(\, \cdot \,)), e \in \cE \big) \circ W_{s}, & \\
% %    & \hspace{14.5em}            P P^{-1} (L_s - \eta_t(\, \cdot \,)) - (L_s - \eta_t(\, \cdot \,))  + \abs{W_{s}} \big)  \, ds \Big) & \\
% %    & \hspace{20em} \circ \big( \id^W \otimes \G^Q \big) \circ \big( \T^W_t \otimes \T^Q_{\varrho_t} \big) (\, \cdot \,) \, \big| \, \sF_t \Big) (\o)  \\
%   I_2(\o) & = \EV_{(W_t, Q_{\varrho_t})(\o)} \Big( \Big( \int_0^\infty \1_{\{L_s \geq \eta_t(\o) \}} \, e^{-\alpha s} \, 
%                         f \big( E \big( P_e P^{-1} (L_s - \eta_t(\o)) - (L_s - \eta_t(\o)), e \in \cE \big) \circ W_{s}, & \\
%    & \hspace{19em}            P P^{-1} (L_s - \eta_t(\o)) - (L_s - \eta_t(\o)) + \abs{W_{s}} \big)  \, ds \Big) & \\
%    & \hspace{20em} \circ \big(\id^W \otimes \G^Q \big) \Big).  & 
%   \end{align*}
%   %%\enableHboxWarning
     \begin{align*}  
  I_2(\o) = \EV_{(W_t, Q_{\varrho_t})(\o)} \Big( \Big( \int_0^\infty & \1_{\{L_s \geq \eta_t(\o) \}} \, e^{-\alpha s} \\ 
                      &  f \big( E \big( P_e P^{-1} (L_s - \eta_t(\o)) - (L_s - \eta_t(\o)), e \in \cE \big) \circ W_{s},  \\
         &  \hspace*{2.4em}     P P^{-1} (L_s - \eta_t(\o)) - (L_s - \eta_t(\o)) + \abs{W_{s}} \big)  \, ds \Big)  \\
   & \hspace*{-2.7em} \circ \big(\id^W \otimes \G^Q \big) \Big).  
  \end{align*}
  The centering operator can be processed with the help of \eqref{eq:centering translation for Levy} by translating the starting point $Q_{\varrho_t}(\o)$ to $0$.
  Furthermore, the set $\{L_s \geq \eta_t(\o) \}$ is decomposed into $\{L_s = \eta_t(\o) \}$ and $\{s > L^{-1}( \eta_t(\o) ) \}$ (see Lemma~\ref{lem:G_IM:props pseudoinverse (cont. case)}),
  transforming $I_2(\o)$ into
%   %%\disableHboxWarning
%   \begin{align*}
% %   & \EV_{g,0} \Big( \int_0^\infty \1_{\{L_s \circ \T^W_t \geq \eta_t \}} \, e^{-\alpha s} \, f(X_{t+s}) \, ds \, \big| \, \sF_t \Big) (\o) & \\
% %    & \EV_{(W_t, Q_{\varrho_t})(\o)} \Big( \Big( \int_0^\infty \1_{\{L_s \geq \eta_t(\o) \}} \, e^{-\alpha s} \, 
% %                         f \big( E \big( P_e P^{-1} (L_s - \eta_t(\o)) - (L_s - \eta_t(\o)), e \in \cE \big) \circ W_{s}, & \\
% %    & \hspace{17.7em}            P P^{-1} (L_s - \eta_t(\o)) - (L_s - \eta_t(\o)) + \abs{W_{s}} \big)  \, ds \Big) & \\
% %    & \hspace{20em} \circ \big(\id^W \otimes \G^Q \big) \Big)  &  \\
%  I_2(\o) & = \EV_{W_t(\o), 0} \Big( \int_0^\infty \1_{\{L_s = \eta_t(\o) \}} \, f(W_s) \, ds & \\
%    & \hspace{5.5em}        + \int_{L^{-1}(\eta_t(\o))}^\infty e^{-\alpha s} \, 
%                                 f \big( E \big( P_e P^{-1} (L_s - \eta_t(\o)) - (L_s - \eta_t(\o)), e \in \cE \big) \circ W_{s}, & \\
%    & \hspace{15.8em}            P P^{-1} (L_s - \eta_t(\o)) - (L_s - \eta_t(\o))  + \abs{W_{s}} \big)  \, ds \Big). & 
%    \end{align*}
%    %%\enableHboxWarning
  \begin{align*}
   &  \EV_{W_t(\o), 0} \Big( \int_0^\infty \1_{\{L_s = \eta_t(\o) \}} \, f(W_s) \, ds \\
   &  \hspace*{3em} + \int_{L^{-1}(\eta_t(\o))}^\infty e^{-\alpha s} \, 
                                f \big( E \big( P_e P^{-1} (L_s - \eta_t(\o)) - (L_s - \eta_t(\o)), e \in \cE \big) \circ W_{s}, & \\
   & \hspace{13.5em}            P P^{-1} (L_s - \eta_t(\o)) - (L_s - \eta_t(\o))  + \abs{W_{s}} \big)  \, ds \Big). & 
   \end{align*}   
   Now, $\{L_s = \eta_t(\o) \}$ is a null set for every $\eta_t(\o) \neq 0$,
   and as $L$ is an additive functional, $L_s \circ \T^W_{L^{-1}(u)} = L_{s + L^{-1}(u)} - L_{L^{-1}(u)} = L_{s + L^{-1}(u)} - u$ 
   holds true, so
   \begin{align*}
%   & \EV_{g,0} \Big( \int_0^\infty \1_{\{L_s \circ \T^W_t \geq \eta_t \}} \, e^{-\alpha s} \, f(X_{t+s}) \, ds \, \big| \, \sF_t \Big) (\o) & \\
%    & \EV_{W_t(\o), 0} \Big( \int_0^\infty \1_{\{L_s = \eta_t(\o) \}} \, f(W_s) \, ds & \\
%    & \hspace{4em}        + \int_{L^{-1}(\eta_t(\o))}^\infty e^{-\alpha s} \, 
%                                 f \big( E \big( P_e P^{-1} (L_s - \eta_t(\o)) - (L_s - \eta_t(\o)), e \in \cE \big) \circ W_{s}, & \\
%    & \hspace{14.3em}            P P^{-1} (L_s - \eta_t(\o)) - (L_s - \eta_t(\o))  + \abs{W_{s}} \big)  \, ds \Big), &  \\                            
  I_2(\o) & = \1_{\{\eta_t(\o) = 0\}} \, \EV_{W_t(\o), 0} \Big( \int_0^\infty \1_{\{L_s = 0 \}} \, f(W_s) \, ds \Big) & \\
   &  \quad  + \EV_{W_t(\o), 0} \Big( e^{-\a L^{-1}(\eta_t(\o))} \, \Big( \int_{0}^\infty e^{-\alpha s} \, 
                        f \big( E \big( P_e P^{-1} (L_s) - L_s , e \in \cE \big) \circ W_{s}, & \\
   & \hspace{13.7em}              P P^{-1} (L_s) - L_s  + \abs{W_{s}} \big)  \, ds \Big) \circ \T^W_{L^{-1}(\eta_t(\o))} \Big). & 
  \end{align*}
  As the local time vanishes until the first hit of the vertex, $\{L_s = 0\} = \{s \leq \h^W_0\}$ holds true,
  and applying the strong Markov property of $W$ with respect to its augmented, right continuous filtration for
  the stopping time $L^{-1} \big( \eta_t(\o) \big)$ (while treating the part of the subordinator $Q$ to be constant, which is possible due to Fubini's theorem),
  with stopping point $W_{L^{-1}(\eta_t(\o))} = 0$, yields
  \begin{align*}
%   & \EV_{g,0} \Big( \int_0^\infty \1_{\{L_s \circ \T^W_t \geq \eta_t \}} \, e^{-\alpha s} \, f(X_{t+s}) \, ds \, \big| \, \sF_t \Big) (\o) & \\
%    & \1_{\{\eta_t(\o) = 0\}} \, \EV_{W_t(\o), 0} \Big( \int_0^\infty \1_{\{L_s = 0 \}} \, f(W_s) \, ds \Big) & \\
%    &  \quad  + \EV_{W_t(\o), 0} \Big( e^{-\a L^{-1}(\eta_t(\o))} \, \Big( \int_{0}^\infty e^{-\alpha s} \, 
%                         f \big( E \big( P_e P^{-1} (L_s) - L_s , e \in \cE \big) \circ W_{s}, & \\
%    & \hspace{18em}              P P^{-1} (L_s) - L_s  + \abs{W_{s}} \big)  \, ds \Big) \circ \T^W_{L^{-1}(\eta_t(\o))} \Big)  \\ 
 I_2(\o)  & = \1_{\{\eta_t(\o) = 0\}} \, \EV_{W_t(\o), 0} \Big( \int_0^{\h^W_0}  f(W_s) \, ds \Big)  \\
   &  \quad + \EV_{W_t(\o), 0} \Big( e^{-\a L^{-1}(\eta_t(\o))} \, \EV_{0,0} \Big( \int_{0}^\infty e^{-\alpha s} \, 
                        f( X_s )  \, ds \Big) \Big).
  \end{align*} 
  Lemma~\ref{lem:G_IM:laplace of H0 for X}, the relation $X_{\h^X_0} = 0$ on $\{ \h^X_0 < \infty \}$ (by right continuity of~$X$) and the definition of~$X$ imply
%   \begin{align*} 
%   & \EV_{g,0} \Big( \int_0^\infty \1_{\{L_s \circ \T^W_t \geq \eta_t \}} \, e^{-\alpha s} \, f(X_{t+s}) \, ds \, \big| \, \sF_t \Big) (\o) & \\
% %    & \1_{\{\eta_t(\o) = 0\}} \, \EV_{W_t(\o), 0} \Big( \int_0^{\h^W_0}  f(W_s) \, ds \Big)  \\
% %    &  \quad + \EV_{W_t(\o), 0} \Big( e^{-\a L^{-1}(\eta_t(\o))} \, \EV_{0,0} \Big( \int_{0}^\infty e^{-\alpha s} \, 
% %                         f( X_s )  \, ds \Big) \Big) & \\            
%    & = \1_{\{\eta_t(\o) = 0\}} \, \EV_{X_t(\o), 0} \Big( \int_0^{\h^X_0}  f(X_s) \, ds \Big)  \\
%    & \quad + \EV_{X_t(\o), 0} \Big( e^{-\a \h^X_0} \, \EV_{X_{\h^X_0},0} \Big( \int_{0}^\infty e^{-\alpha s} \, f( X_s )  \, ds \Big) \Big).  
%   \end{align*} 
%   
%   In total, we obtained
  \begin{align*}
%    & \EV_{g,0} \Big( \int_0^\infty \1_{\{L_s \circ \T^W_t \geq \eta_t \}} \, e^{-\alpha s} \, f(X_{t+s}) \, ds \, \big| \, \sF_t \Big)  & \\
   I_2
   & = \1_{\{\eta_t = 0\}} \, \EV_{X_t, 0} \Big( \int_0^{\h^X_0}  f(X_s) \, ds \Big)  \\
   & \quad  + \EV_{X_t, 0} \Big( e^{-\a \h^X_0} \, \EV_{X_{\h^X_0},0} \Big( \int_{0}^\infty e^{-\alpha s} \, f( X_s )  \, ds \Big) \Big),
  \end{align*}
  which together with the result \eqref{eq:G_IM:Markov decomposition of X, part 1} for the first excursion concludes the proof.\sqed
\end{proof}

By combining this lemma with the strong Markov property at $\h_0$, we are now able to deduce the Markov property of $X$.
As we only have access to the strong Markov property at $\h_0$ with respect to canonical filtration $(\sF^Y_t, t \geq 0)$ of the path space realization $Y$ of~$X$
(see the preceding subsection~\ref{subsec:G_IM:strong Markov at H0}), we need to restrict our attention to the canonical filtration of $X$ as well:
 \begin{align*}
  \sF^X_t := \sigma(X_s, s \leq t), \quad t \geq 0.
 \end{align*}
As $X$ is adapted to $(\sF_t, t \geq 0)$ by Lemma~\ref{lem:G_IM:X is adapted to filtration F},
we have $\sF^X_t \subseteq \sF_t$ for all $t \geq 0$. Thus, Lemma~\ref{lem:G_IM:Markov decomposition of X} yields:
\begin{corollary}\label{cor:G_IM:Markov decomposition of X, natural}
  For all $g \in \cG$, $f \in b\sB(\cG)$, $t \geq 0$, 
 \begin{align*}
  & \EV_{g,0} \Big( \int_0^\infty e^{-\alpha s} \, f(X_{t+s}) \, ds \, \big| \, \sF^X_t \Big)  \\
  & = \EV_{X_t,0} \Big( \int_0^{\h^X_0} e^{-\alpha s} \, f( X_s )  \, ds \Big)
    + \EV_{X_t,0} \Big( e^{-\a \h^X_0} \, \EV_{X_{\h^X_0},0} \Big( \int_{0}^\infty e^{-\alpha s} \, f( X_s )  \, ds \Big) \Big).
 \end{align*}
\end{corollary}
% \begin{proof}
%  As $X$ is adapted to $(\sF_t, t \geq 0)$ by lemma~\ref{lem:G_IM:X is adapted to filtration F},
%  we have $\sF^X_t \subseteq \sF_t$ for all $t \geq 0$. Thus, lemma~\ref{lem:G_IM:Markov decomposition of X} yields
%   \begin{align*}
%   & \EV_{g,0} \Big( \int_0^\infty e^{-\alpha s} \, f(X_{t+s}) \, ds \, \big| \, \sF^X_t \Big) &  \\
%   & = \EV_{g,0} \Big( \EV_{g,0} \Big( \int_0^\infty e^{-\alpha s} \, f(X_{t+s}) \, ds \, \big| \, \sF_t \Big) \big| \, \sF^X_t \Big) &  \\
%   & = \EV_{g,0} \Big( \EV_{X_t,0} \Big( \int_0^{\h^X_0} e^{-\alpha s} \, f( X_s )  \, ds \Big) \\
%   & \hspace*{3.5em} + \EV_{X_t,0} \Big( e^{-\a \h^X_0} \, \EV_{X_{\h^X_0},0} \Big( \int_{0}^\infty e^{-\alpha s} \, f( X_s )  \, ds \Big) \Big) \big| \, \sF^X_t \Big).
%  \end{align*}
%  The remaining conditional expectation trivializes as $X$ is adapted to $(\sF^X_t, t \geq 0)$.
% \end{proof}

\begin{lemma}  \label{lem:G_IM:Markov property of X}
 For all $g \in \cG$, $f \in b\sB(\cG)$, $t \geq 0$, 
  \begin{align*} 
   \EV_{g} \Big( \int_0^\infty e^{-\alpha s} \, f(X_{t+s}) \, ds \, \big| \, \sF^X_t \Big) 
   & = \EV_{X_t} \Big( \int_0^\infty e^{-\alpha s} \, f(X_s) \, ds \Big). 
  \end{align*}
\end{lemma}
\begin{proof}
  Let $g \in \cG$, $f \in b\sB(\cG)$ and $t \geq 0$.
  Switching to the path-space realization $Y$ of $X$, we set
  for all  $n \in \N$, $A_1, \ldots, A_n \in \sB(\cG)$, $0 \leq t_1 \leq \cdots \leq t_n \leq t$
  \begin{align*}
   A^X & := \{ X_{t_1} \in A_1, \ldots,  X_{t_n} \in A_n \} \in \sF^X_t, \\
   A^Y & :=  \{ Y_{t_1} \in A_1, \ldots,  Y_{t_n} \in A_n \} \in \sF^Y_t.
  \end{align*}%
   Corollary~\ref{cor:G_IM:Markov decomposition of X, natural} yields
   \begin{align*} 
    &  \EV_{g} \Big( \int_0^\infty e^{-\alpha s} \, f(X_{t+s}) \, ds \, ; \, A^X \Big)  \\            
  % & = \EV_{g} \Big( \EV_{Y_t \circ \Phi} \Big( \int_0^{\h^Y_0 \circ \Phi} e^{-\alpha s} \, f( Y_s \circ \Phi )  \, ds \Big) \\
  % & \hspace{2.8em}    + \EV_{Y_t \circ \Phi} \Big( e^{-\a \h^Y_0 \circ \Phi} \, \EV_{Y_{\h^Y_0}  \circ \Phi} \Big( \int_{0}^\infty e^{-\alpha s} \, f( Y_s \circ \Phi )  \, ds \Big) \Big)
  %                     \, ; \, \Phi^{-1}(A^Y) \Big) \\   
   & = \EV^Y_{g} \Big( \EV^Y_{Y_t} \Big( \int_0^{\h^Y_0} e^{-\alpha s} \, f( Y_s )  \, ds \Big) \\
   & \hspace*{3em}      + \EV^Y_{Y_t} \Big( e^{-\a \h^Y_0} \, \EV^Y_{Y_{\h^Y_0}} \Big( \int_{0}^\infty e^{-\alpha s} \, f( Y_s )  \, ds \Big) \Big)
                       \, ; \, A^Y \Big).
   \end{align*}
    The path-space realization $Y$ satisfies the strong Markov property of  at $\h^Y_0$, as shown in Theorem~\ref{theo:G_IM:strong Markov of Y},
    so Dynkin's formula~\eqref{eq:Dynkins formula (resolvent)} gives
   \begin{align*}
   \EV_{g} \Big( \int_0^\infty e^{-\alpha s} \, f(X_{t+s}) \, ds \, ; \, A^X \Big) 
   & = \EV^Y_{g} \Big( \EV^Y_{Y_t} \Big( \int_0^\infty e^{-\alpha s} \, f( Y_s )  \, ds \Big)
                       \, ; \, A^Y \Big) \\                        
   & = \EV_{g} \Big( \EV_{X_t} \Big( \int_0^\infty e^{-\alpha s} \, f(X_s) \, ds \Big) \,;\, A^X \Big). \qedhere
  \end{align*} 
\end{proof}

\begin{theorem}
 $X = (\O, \sF, (\sF^X_t)_{t \geq 0}, (X_t)_{t \geq 0}, (\T^X_t)_{t \geq 0}, (\PV_g)_{g \in \cG})$ is a Markov process.
\end{theorem}
\begin{proof}
 The right continuity of $(X_t, t \geq 0)$ has been shown in Theorem~\ref{theo:G_IM:continuity process X}. 
 In view of~\eqref{eq:Markov property (resolvent)}, the Markov property has been proved in Lemma~\ref{lem:G_IM:Markov property of X}.\sqed
\end{proof}

\subsection{Strong Markov Property of \texorpdfstring{$X$}{X}}  \label{subsec:G_IM:strong Markov of X}

With the Markov property of $X$ and its strong Markov property at the first hitting time of $0$,
we are now able to deduce the Feller property (and thus, the strong Markov property) of $X$:

\begin{theorem} \label{lem:G_IM:Feller}
 $X$ is a Feller process.
\end{theorem}
\begin{proof}
 We already know that $X$ is a Markov process. %, which implies the family $(T_t, t \geq 0)$ defined by 
%   \begin{align*}
%     T_t f(x) = \EV_{g} \big( f(X_t) \big), \quad t \geq 0, f \in b\sB(\cG), g \in \cG,
%   \end{align*}
%is indeed a Markov semigroup. It is therefore sufficient to show that this semigroup is Feller.
 We will check property \eqref{eq:Feller (resolvent)}.
% By the right continuity and normality of $X$, Lebesgue's dominated convergence theorem yields
%   \begin{align*}
%     \forall g \in \cG, f \in \cC_0(\cG): \quad \lim_{t \downarrow 0} \, \EV_{g} \big( f(X_t) \big) = \EV_{g} \big( f(X_0) \big) =  f(g),
%   \end{align*}
% so $(T_t, t \geq 0)$ is continuous at $0$. 
% It remains to prove that the resolvent $(U_\a, \a > 0)$ of $X$ preserves $\cC_0(\cG)$, that is, we need to show that
%  \begin{align*} 
%    \forall \a > 0, f \in \cC_0(\cG) : \quad U_\a f \in \cC_0(\cG).
%  \end{align*}
    %The boundedness is clear, as $\norm{U_\a f} \leq \frac{1}{\a} \norm{f}$ as usual (c.f. \ref{x})
    %The continuity follows from 
  To this end, we decompose once again the resolvent of $X$ at $\h_0$ with Lemma~\ref{lem:G_IM:Markov decomposition of X} for $t = 0$:
    Using $X_{\h_0} = 0$ (by the right continuity of $X$) and Lemma~\ref{lem:G_IM:X = W for t <= H0}, we get
    for $g = (e,x) \in \cG$:
    \begin{align*}
     U_\a f(g) 
     & = \EV_{g} \Big( \int_0^\infty e^{-\a s} \, f(X_s) \, ds \Big) \\
     & = \EV_{g} \Big( \int_0^{\h_0} e^{-\a s} \, f(X_s) \, ds \Big) + \EV_{g} \Big( e^{-\a \h_0} \, \EV_{X_{\h_0}} \Big( \int_0^\infty e^{-\a s} \, f(X_s) \, ds \Big) \Big) \\
     & = \EV^W_g \Big( \int_0^{\h^W_0} e^{-\a s} \, f(W_s) \, ds \Big) + \EV^W_g \big( e^{-\a \h^W_0} \big) \, U_\a f(0) \\
     & = U^{W,D}_\a f(g) + e^{-\sqrt{2\a} x} \, U_\a f(0),
     %& = \sum_{e \in \cE} \1_{\{ \pi^1(g) = e \}} \, \EV^B_x \Big( \int_0^{\h^B_0} e^{-\a s} \, f_e(B_s) \, ds \Big) + \EV^B_x \big( e^{-\a \h^B_0} \big) \, U_\a f(0) \\
     %& = \sum_{e \in \cE} \1_{\{ \pi^1(g) = e \}} \, U^K_\a f_e(x) + \psi_\a(x) \, U_\alpha f(0),
    \end{align*}
    with $\big( U^{W,D}_\a, \a > 0 \big)$ being the resolvent of the Walsh process on $\cG$ killed at $0$. 
    It is now immediate that $(U_\a, \a > 0)$ preserves~$\cC_0(\cG)$, 
    because $\big( U^{W,D}_\a, \a > 0 \big)$ preserves~$\cC_0(\cG)$ by Example~\ref{ex:G_WP:dirichlet Walsh BB}, 
    $(e,x) \mapsto \exp \big( -\sqrt{2\a} x \big)$ is continuous and vanishes at infinity,
    and $\lim_{g \rightarrow 0} U^{W,D}_\a f(g) = 0$ holds true.\sqed
    %the strong Markov property of  and the Feller property of the subprocess $W$, because
    %by setting $g = (e,x)$, $f_e := f(e, \cdot)$, we have with lemma (x) for $t = 0$ %with Dynkin's formula
    %with $(U^K_\a, \a > 0)$ being the resolvent of the killed Brownian motion. As $f_e = \pi_\cG^e f \in b\cC(\R_+)$, it follows with lemma \ref{x} that $U^K_\a f_e \in b\cC(\R_+)$,
    %and $\psi_\a  \in b\cC(\R_+)$, as seen in lemma \ref{x},
    %so with knowing that the vertex $e \in \cE$ can only change continuously at the length $x = 0$, and $U^K_\a f_e(0) = 0$ there,
    %we get $U_\a f \in b\cC(\cG)$.
\end{proof}

\subsection{Local Time of \texorpdfstring{$X$}{X} at the Vertex} \label{subsec:G_IM:local time of X}

As $P$ is strictly increasing, the process $P^{-1} L_t$ grows if and only if $P P^{-1} L_t$ grows, that is, if $L_t \in \sT$  (cf.\ the results of subsection \ref{subsec:G_IM:remarks}).
But then $X_t = W_t$ must be at $0$. Furthermore, we showed in equation~\eqref{eq:G_IM:shifts:shift on P-1(Ls)} that $t \mapsto P^{-1} L_t$ is an additive functional for $X$.
Therefore, the following result is to be expected (see also \cite[Section V.3]{BlumenthalGetoor69}):

\begin{theorem}
 The local time $(L^X_t, t \geq 0)$ of $X$ at $0$ is 
   \begin{align*}
     L^X_t = P^{-1} L_t, \quad t \geq 0.
   \end{align*}
\end{theorem}

In general, the local time of $X$ at $0$ only depends on the behavior of $X$ at $0$, and therefore only on the behavior of the local coordinate
$(\eta_t + \abs{W_t}, t \geq 0)$ at the origin. This is exactly the Brownian motion on the half line which was constructed by It\^{o} and McKean,
and it was proved in \cite[Section 14]{ItoMcKean63} that $(P^{-1} L_t, t \geq 0)$ is its local time at the origin.
So the above theorem is achieved by carrying over their result to our generalization.

\subsection{General Brownian Motion \texorpdfstring{$X^\bullet$}{X\textbullet} on a Star Graph} \label{subsec:G_IM:general BB on star graph}

Up to this point, we only took care of the reflection parameters $(p^e_2, e \in \cE)$
and the jump distributions $(p^e_4, e \in \cE)$.
We will now implement the stickiness parameter $p_3 \geq 0$ and the killing parameter $p_1 \geq 0$ by
using the standard procedures of time change and killing. 
To this end, we will now consider the Feller process $X$ as right process in the context of the usual hypotheses.

In order to implement stickiness, we define the additive functional $(\t_t, t \geq 0)$ by %with respect to $X$ by
 \begin{align*}
  \t_t := t + p_3 L^X_t, \quad t \geq 0,
 \end{align*}
and consider the time-changed process $\big( X_{\t^{-1}(t)}, t \geq 0 \big)$.
By \cite[Theorem~A.3.11]{ChenFukushima11},
$\big( X_{\t^{-1}(t)}, t \geq 0 \big)$ is a right process with shift operators $\big( \T^X_{\t^{-1}(t)}, t \geq 0 \big)$.
Its local time turns out to be $\big( L^X_{\t^{-1}(t)}, t \geq 0 \big)$, which we will only need (and thus, show) partially:

\begin{lemma} \label{lem:G_IM:additive functional for time subst process}
 $\big( L^X_{\t^{-1}(t)}, t \geq 0 \big)$ is an additive functional for $\big( X_{\t^{-1}(t)}, t \geq 0 \big)$.
\end{lemma}
\begin{proof}
 For any $s, t \geq 0$, we compute
  \begin{align*}
   L^X_{\t^{-1}(t)} \circ \T^X_{\t^{-1}(s)}
   & = L^X_{\t^{-1}(t) \circ \T^X_{\t^{-1}(s)} + \t^{-1}(s)} - L^X_{\t^{-1}(s)} \\
   & = L^X_{\t^{-1}(t+s)} - L^X_{\t^{-1}(s)},
  \end{align*}
 where we used that $(L^X_t, t \geq 0)$ is an additive functional 
 (cf.~equation~\eqref{eq:G_IM:shifts:shift on P-1(Ls)})
 for the first identity, and for the second identity employed the relation 
  \begin{align*}
   \t^{-1}(t) \circ \T^X_{\t^{-1}(s)} + \t^{-1}(s) = \t^{-1}(t+s), \quad s,t \geq 0,
  \end{align*}
 which is a general result for the inverse of any additive functional $(\t_t, t \geq 0)$ (see, e.g., \cite[Proposition 65.8]{Sharpe88}
 or the computations in the proof of \cite[Theorem 6.4]{Knight81}).\sqed
\end{proof}

Now kill this new process $\big( X_{\t^{-1}(t)}, t \geq 0 \big)$ once its local time $\big( L^X_{\t^{-1}(t)}, t \geq 0 \big)$ reaches a certain level:
To this end, introduce an exponentially distributed random variable $S$ with mean $1$, independent of~$\sF$, by extending the probability space 
(for a standard construction, see~\cite[Appendix~A]{KPS10} or~\cite[Appendix~A.3]{ChenFukushima11}), and set
  \begin{align*}
    \z := \inf \{ t \geq 0: p_1 \, L^X_{\t^{-1}(t)} > S \}.
  \end{align*}
Establish the definitive process $X^\bullet$ resulting from killing $\big( X_{\t^{-1}(t)}, t \geq 0 \big)$ at $\z$ by
  \begin{align*}
    \forall t \geq 0: \quad
    X^\bullet_t :=
      \begin{cases}
        X_{\t^{-1}(t)}, & t < \z, \\
        \D,             & t \geq \z.
      \end{cases}
  \end{align*} 
In view of Lemma~\ref{lem:G_IM:additive functional for time subst process}, \cite[Theorem~A.3.13]{ChenFukushima11} yields the following:

\begin{theorem} \label{theo:G_IM:X_bullet right process}
 $X^\bullet$ is a right process.
\end{theorem}
  
\subsection{Resolvent and Generator of \texorpdfstring{$X^\bullet$}{X\textbullet}} \label{subsec:G_IM:resolvent and generator}

We will conclude our construction by showing that $X^\bullet$ is indeed the process which implements the correct boundary conditions
into the generator.
Let $(U^\bullet_\a, \a > 0)$ be the resolvent and $A^\bullet$ be the generator of $X^\bullet$.

We first trace the resolvent $U^\bullet$ of $X^\bullet$ back to the components of $X$:
\begin{lemma} \label{lem:G_IM:resolvent of X preparations}
  For $\a > 0$, $f \in b\cC(\cG)$, $g \in \cG$,
  \begin{align*}
     U^\bullet_\a f (g) 
     & = \EV_g \Big( \int_0^\infty e^{-\a t} \, e^{-(p_1 + \a p_3) L^X_t} \, f(X_t) \, d \t(t) \Big).
  \end{align*}
\end{lemma}
\begin{proof}
  The definition of $X^\bullet$ and the independence of $S$ from everything else yield
   \begin{align*}
     U^\bullet_\a f (g) 
     & = \EV_g \Big( \int_0^\infty e^{-\a t} \, f(X^\bullet_t) \, dt \Big) \\
     & = \EV_g \Big( \int_0^\infty e^{-\a t} \, \1_{ \{p_1 \, L^X_{\t^{-1}(t)} < S\} } \, f(X_{\t^{-1}(t)}) \, dt \Big) \\
     & = \EV_g \Big( \int_0^\infty e^{-\a t} \, \Big( \int_{p_1 \, L^X_{\t^{-1}(t)}}^\infty e^{-s} \, ds \Big) \, f(X_{\t^{-1}(t)}) \, dt \Big) \\
     & = \EV_g \Big( \int_0^\infty e^{-\a t} \, e^{-p_1 L^X_{\t^{-1}(t)}} \, f(X_{\t^{-1}(t)}) \, dt \Big).
  \end{align*}     
  As $(\t_t, t \geq 0)$ is increasing and bijective, the substitution rule for Stieltjes integrals (see, e.g., \cite{FalknerTeschl11}) gives
  \begin{align*}
     U^\bullet_\a f (g)
     & = \EV_g \Big( \int_0^\infty e^{-\a \t(\t^{-1}(t))} \, e^{-p_1 L^X_{\t^{-1}(t)}} \, f(X_{\t^{-1}(t)}) \, dt \Big) \\
     & = \EV_g \Big( \int_0^\infty e^{-\a \t(t)} \, e^{-p_1 L^X_t} \, f(X_t) \, d \t(t) \Big),
  \end{align*}
  and inserting the definition $\t_t = t + p_3 L^X_t$ completes the proof.\sqed
\end{proof}

We are now ready to completely calculate the resolvent of $X^\bullet$.
The form of the resolvent is well-known for the case of the half line, see 
\cite[Section 15]{ItoMcKean63}, or \cite[Theorem 3]{Rogers83} for a different approach via excursion theory.
As we constructed $X^\bullet$ pathwise, we will follow the computational techniques of \cite{ItoMcKean63} 
in order to prove the following theorem:

\begin{theorem} \label{theo:G_IM:resolvent of X}
  For $\a > 0$, $f \in b\cC(\cG)$, $g = (e,x) \in \cG$, 
    \begin{align*}
      U^\bullet_\a f (g)
      & = U^{W,D}_\a f(g) + e^{-\sqrt{2\a} x} \, U^\bullet_\a f(0)
    \end{align*}
  holds, with $\big( U^{W,D}_\a, \a > 0 \big)$ being the resolvent of the Walsh process on $\cG$ killed at~$0$ (as given in Example~\ref{ex:G_WP:dirichlet Walsh BB}), and
 %   \begin{align*}
 %      U^{W,D}_\a f(g) = U^{B,K}_\a \big( f(e,\cdot) \big) (x)
 %   \end{align*}
 % and
    \begin{align*}
      U^\bullet_\a f (0)
      & = \frac { \sum_{e \in \cE} p^e_2 \, 2 \int_0^\infty e^{-\sqrt{2 \a} x} f(e,x) \, dx + p_3 \, f(0) + \int U^{W,D}_\a f(g) \, p_4(dg) }
                { p_1 + \sqrt{2 \a} \, p_2 + \a \, p_3 + \int_0^\infty (1 - e^{-\sqrt{2 \a} l}) \, p^\Sigma_4 (dl) }
    \end{align*}
  holds with $p^\Sigma_4 = \sum_{e \in \cE} p^e_4$.
\end{theorem}
\begin{proof}
 Let $g = (e,x) \in \cG$. Consider the first hitting time of the vertex $0$ for $X^\bullet$, that is,
   \begin{align*}
    H^{\bullet}_0 := \inf \{ t \geq 0: X^\bullet_t = 0 \}.
  \end{align*}
 We observe that the transformation effects from $X$ to $X^\bullet$ only take effect after the first hitting of $0$, 
 so $X^\bullet_t = X_t = W_t$ for all $t \leq H^{\bullet}_0 = H^X_0 = H^W_0$ (see also Lemma~\ref{lem:G_IM:X = W for t <= H0}).
 In addition, $X^\bullet_{H^{\bullet}_0} = 0$ holds by right continuity of $X^\bullet$.
 The application of Dynkin's formula~\eqref{eq:Dynkins formula (resolvent)} for the decomposition of the right (thus strongly Markovian) 
 process $X^\bullet$ at the stopping time $H^{\bullet}_0$ therefore yields
    \begin{align*}
      U^\bullet_\a f (g)
      & = \EV_g \Big( \int_0^{H^{\bullet}_0} e^{-\a t} \, f(X_t) \, dt \Big) + \EV_g \big( e^{-\a H^{\bullet}_0} \, U^\bullet_\a f(X^\bullet_{H^{\bullet}_0}) \big) \\
      & = \EV^W_g \Big( \int_0^{H^W_0} e^{-\a t} \, f(W_t) \, dt \Big) + \EV^W_g \big( e^{-\a H^W_0} \big) \, U^\bullet_\a f(0).
    \end{align*}
 The Laplace transform of the first hitting time of the vertex reads $\EV^W_g \big( e^{-\a H^W_0} \big) = e^{-\sqrt{2\a} x}$ 
 by \eqref{eq:BM passage time}, as the Walsh process $W$ behaves on any edge like a reflecting Brownian motion (see Theorem~\ref{theo:G_WP:WP path properties}).
 
 It remains to analyze the resolvent at the vertex $0$: Continuing the computations of Lemma~\ref{lem:G_IM:resolvent of X preparations}, we obtain 
 by inserting the definition of $\t_t$ and using that $L^X_t$ only grows at $X_t = 0$, that 
  \begin{align*}
     U^\bullet_\a f (0) 
     & = \EV_0 \Big( \int_0^\infty e^{-\a t} \, e^{-(p_1 + \a p_3) L^X_t} \, f(X_t) \, d \t(t) \Big) \\
     & = \EV_0 \Big( \int_0^\infty e^{-\a t} \, e^{-(p_1 + \a p_3) L^X_t} \, f(X_t) \, d t \Big) \\
     & \quad + p_3 \, f(0) \, \EV_0 \Big( \int_0^\infty e^{-\a t} \, e^{-(p_1 + \a p_3) L^X_t} \, d L^X(t) \Big).
  \end{align*} 
  Decomposing $\R_+$ into $\bigcup_{n \in \N} \big[ L_-^{-1}(l^-_n), L_-^{-1}(l^+_n) \big)$ and its complement, and using 
  that $X_t = \big( e_n, l^+_n + \abs{W_t} - L_t \big)$ holds for $t \in \big[ L_-^{-1}(l^-_n), L_-^{-1}(l^+_n) \big)$, $n \in \N$, 
  and $X_t = W_t$ otherwise (see subsection~\ref{subsec:G_IM:remarks}, especially equation~\eqref{eq:G_IM:behavior process X}) results in
  \begin{align*}
     U^\bullet_\a f (0) 
     & = \sum_{n \in \N} \EV_0 \Big( \int_{L_-^{-1}(l^-_n)}^{L_-^{-1}(l^+_n)} e^{-\a t} \, e^{-(p_1 + \a p_3) P^{-1} L_t} \, f \big( e_n, l^+_n + \abs{W_t} - L_t \big) \, d t \Big) \\
     & \quad + \EV_0 \Big( \int_0^\infty e^{-\a t} \, e^{-(p_1 + \a p_3) P^{-1} L_t} \, f(W_t) \, d t \Big) \\
     & \quad - \sum_{n \in \N} \EV_0 \Big( \int_{L_-^{-1}(l^-_n)}^{L_-^{-1}(l^+_n)} e^{-\a t} \, e^{-(p_1 + \a p_3) P^{-1} L_t} \, f(W_t) \, d t \Big) \\
     & \quad + p_3 \, f(0) \, \EV_0 \Big( \int_0^\infty e^{-\a t} \, e^{-(p_1 + \a p_3) P^{-1} L_t} \, d P^{-1} L_t \Big) \\
     & =: u_1 + u_2 - u_3 + u_4.
  \end{align*}
  We are going to compute these four expressions one after the other:
  
  We start with $u_1$: The functions $l^-_n$, $l^+_n$, $n \in \N$, only depend on $\O^Q$. We begin by computing the (conditional) expectation with respect to the
  space $\O^W$. Fubini's theorem asserts that while integrating on $\O^W$, we can treat $l^-_n$, $l^+_n$, $n \in \N$, as constants 
  (this will not be annotated in the formulas below to keep them reasonably readable), therefore
  \begin{align*}
   u_1 
   & = \sum_{n \in \N} \EV^Q_0 \Big( \EV^W_0 \Big( \int_{0}^{L_-^{-1}(l^+_n) - L_-^{-1}(l^-_n)} \, e^{-\a (t + L_-^{-1}(l^-_n))} \,
                                          e^{-(p_1 + \a p_3) P^{-1}(l^-_n)} \\
   & \hspace{14.7em} f \big( e_n, l^+_n + \bigabs{W_{t + L_-^{-1}(l^-_n)}} - L_{t + L_-^{-1}(l^-_n)} \big) \, dt \Big) \Big).
  \end{align*}
  Using
  $L_-^{-1}(l^+_n) - L_-^{-1}(l^-_n) = L_-^{-1}(l_n) \circ \T^W_{L_-^{-1}(l^-_n)}$ by Lemma~\ref{lem:B_LT:shift of inverse local time},
  the additive functional property $L_{t + L_-^{-1}(l^-_n)} = L_t \circ \T^W_{L_-^{-1}(l^-_n)} - L_{L_-^{-1}(l^-_n)}$,
  with $L_{L_-^{-1}(l^-_n)} = l^-_n$ by continuity of $L$, 
  as well as $l^+_n - l^-_n = l_n$ by Remark~\ref{rem:G_IM:behavior process X},
  then yields for $u_1$
  \begin{align*}
   & \sum_{n \in \N} \EV^Q_0 \Big( \EV^W_0 \Big( e^{-\a L_-^{-1}(l^-_n)} \, e^{-(p_1 + \a p_3) P^{-1}(l^-_n)} \\ 
   & \hspace{4em}                                 \EV^W_0 \Big( \Big( \int_{0}^{L_-^{-1}(l_n)} e^{-\a t} \, f \big(e_n, l_n + \abs{W_t} - L_t \big) \, dt \Big)
                                                                \circ \T^W_{L_-^{-1}(l^-_n)} \, \big| \, \sF^W_{\T^W_{L_-^{-1}(l^-_n)}} \Big) \Big) \Big).
  \end{align*}
  $W$ is strongly Markovian with respect to the stopping time $L_-^{-1}(l^-_n)$, with the stopping point being given by $W_{L_-^{-1}(l^-_n)} = 0$ (as $L$ only grows at $0$),
  so by also using that $L_-^{-1} (l^-_n) = L^{-1}(l^-_n)$ holds a.s.\ by Lemma~\ref{lem:B_LT:lc rc inverse}, it follows that 
  \begin{align*}
   u_1 
    = \sum_{n \in \N} \EV^Q_0 \Big( e^{-(p_1 + \a p_3) P^{-1}(l^-_n)} \, & \EV^W_0 \Big( e^{-\a L^{-1}(l^-_n)} \\
                                 & \quad   \EV^W_0 \Big( \int_{0}^{L^{-1}(l_n)} e^{-\a t} 
                                                     \, f \big(e_n, l_n + \abs{W_t} - L_t \big) \, d t \Big) \Big) \Big).
  \end{align*}
  Now the process $\big( l_n + \abs{W_t} - L_t, t \leq L^{-1}(l_n) \big)$ started at $0$ behaves just like 
  the standard Brownian motion $\big(B_t, t \leq H^B_0 \big)$ started at $l_n$ (cf.~Lemma~\ref{lem:B_LT:extension to levys char}).
  By using Lemma~\ref{lem:B_LT:laplace of local time inverse} for the characteristic function of $L^{-1}$, we thus get
  \begin{align*}
   u_1 
   & = \sum_{n \in \N} \EV^Q_0 \Big( e^{-(p_1 + \a p_3) P^{-1}(l^-_n)} \,  
                                 \EV^W_0 \Big( e^{-\a L^{-1}(l^-_n)} \,
                                          \EV^B_{l_n} \Big( \int_{0}^{H^B_0} e^{-\a t} \, f(e_n, B_t) \, d t \Big) \Big) \Big) \\
   & = \sum_{n \in \N} \EV^Q_0 \Big( e^{-(p_1 + \a p_3) P^{-1}(l^-_n)} \,
                                \EV^W_0 \Big( e^{-\a L^{-1}(l^-_n)} \Big) \,
                                U^{[0,\infty)}_\a \big( f(e_n, \cdot) \big) (l_n) \Big) \\
   & = \sum_{n \in \N} \EV^Q_0 \Big( e^{-\a l^-_n} \, e^{-(p_1 + \a p_3) P^{-1}(l^-_n)} \, U^{W,D}_\a f(e_n, l_n) \Big).
  \end{align*}
  Representing $P$ by its random measure $N$, with jump times $(t_n, n \in \N)$ and jump marks $\big( (e_n, l_n) , n \in \N \big)$ as discussed in 
   Remark~\ref{rem:G_IM:behavior process X}, results in
  \begin{align*}
   u_1
   & = \sum_{n \in \N} \EV^Q_0 \Big( e^{-\a P(t_n-)} \, e^{-(p_1 + \a p_3) t_n} \, U^{W,D}_\a f(e_n, l_n) \Big) \\
   & = \EV^Q_0 \Big( \int e^{-\a P(t-)} \, e^{-(p_1 + \a p_3) t} \, U^{W,D}_\a f(g) \, N(dt \times dg) \Big).
  \end{align*}  
  Computing the expectation of the above stochastic integral with respect to the Poisson random measure $N$
  through its compensator (cf.~\cite[p.~62]{IkedaWatanabe89}), we obtain
  \begin{align*}
   u_1 
   & = \int_0^\infty e^{-t ( \sqrt{2 \a} p_2 + \int_0^\infty (1 - e^{-\sqrt{2 \a} l}) \, p^\Sigma_4 (dl) )} \, e^{-(p_1 + \a p_3) t} dt
       \cdot \int  U^{W,D}_\a f(g) \, p_4(dg) \\
   & = \frac { \int U^{W,D}_\a f(g) \, p_4(dg) }
             { p_1 + \sqrt{2 \a} \, p_2 + \a \, p_3 + \int_0^\infty \big(1 - e^{-\sqrt{2 \a} l} \big) \, p^\Sigma_4 (dl) }. 
  \end{align*}   
  
 The computations for $u_3$ follow the same path, but are easier. By using the same techniques as for $u_1$, we get
  \begin{align*}
   u_3 
   & = \sum_{n \in \N} \EV^Q_0 \Big( e^{-(p_1 + \a p_3) P^{-1}(l^-_n)} \\               
   & \hspace*{5em}            \EV^W_0 \Big( \int_{0}^{L_-^{-1}(l_n) \circ \T^W_{L_-^{-1}(l^-_n)}} e^{-\a (t + L_-^{-1}(l^-_n))}
                                           \, f \big( W_t \circ \T^W_{L_-^{-1}(l^-_n)} \big) \, dt \Big) \Big) \\
   & = \sum_{n \in \N} \EV^Q_0 \Big( e^{-(p_1 + \a p_3) P^{-1}(l^-_n)} \,   
                                 e^{-\a L^{-1}(l^-_n)} \,
                                 \EV^W_0 \Big( \int_{0}^{L^{-1}(l_n)} f(W_t) \, dt \Big) \Big).
  \end{align*}
 Applying Dynkin's formula~\eqref{eq:Dynkins formula (resolvent)} for the decomposition at the stopping time $L^{-1}(l_n)$ 
 (see Lemma~\ref{lem:B_LT:inverse local time is stopping time}) yields
  \begin{align*}
   U^W_\a f(0)
   & = \EV^W_0 \Big( \int_0^{L^{-1}(l_n)} e^{-\a t} \, f(W_t) \, dt \Big) + \EV^W_0 \Big( e^{-\a L^{-1}(l_n)} \, U^W_\a f \big(W_{L^{-1}(l_n)} \big) \Big) \\
   & = \EV^W_0 \Big( \int_0^{L^{-1}(l_n)} e^{-\a t} \, f(W_t) \, dt \Big) + e^{-\sqrt{2 \a} l_n} \, U^W_\a f(0),
  \end{align*} 
 thus resulting in 
  \begin{align*}
   u_3
   & = \sum_{n \in \N} \EV^Q_0 \Big( e^{-(p_1 + \a p_3) P^{-1}(l^-_n)} \,
                                 e^{-\sqrt{2 \a} l^-_n} \,
                                 \big( 1 -  e^{-\sqrt{2 \a} l_n} \big) \, U^W_\a f(0) \Big) \\
   & = \sum_{n \in \N} \EV^Q_0 \Big( e^{-(p_1 + \a p_3) t_n} \,
                                 e^{-\sqrt{2 \a} P(t_n-)} \, % + L^{-1}(l^-_n))} \, ??
                                 \big( 1 -  e^{-\sqrt{2 \a} l_n} \big) \Big) \cdot U^W_\a f(0) \\
   & = \frac { \int_0^\infty \big( 1 - e^{-\sqrt{2 \a} l} \big) \, p^\Sigma_4 (dl) }
             { p_1 + \sqrt{2 \a} \, p_2 + \a \, p_3 + \int_0^\infty \big( 1 - e^{-\sqrt{2 \a} l} \big) \, p^\Sigma_4 (dl) }
       \cdot U^W_\a f(0),
  \end{align*} 
 where the last identity follows again from \cite[p.~62]{IkedaWatanabe89} together with
  \begin{align*}
   \int \big( 1 - e^{-\sqrt{2 \a} \pi^2(e,x)} \big) \, p_4 \big( d(e,x) \big)
   & = \int_0^\infty  \big( 1 - e^{-\sqrt{2 \a} x} \big) \, p^\Sigma_4(dx).
  \end{align*}
  
 We are turning to $u_2$ next. %we first recall the joint distribution of $(W_t, L_t)$ :
%  \begin{align*}
%    \EV^W_0 \big( f(W_t, L_t) \big)
%    & = \sum_{e \in \cE} q^e_2 \int_0^\infty \int_0^\infty f \big( (e,x),y \big) \, \frac{2(x+y)}{\sqrt{2\pi t^3}} \, e^{-\frac{(x+y)^2}{2t}} \, dx \, dy.
%  \end{align*}
 Using the independence of $W$ and $Q$, as well as the distribution for $(W_t, L_t)$ (see Lemma~\ref{lem:G_WP:joint dist W LT}), gives
 for $u_2$
  \begin{align*}
   & \EV^Q_0 \Big(  \sum_{e \in \cE} q^e_2 \, \int_0^\infty \! \int_0^\infty \! \int_0^\infty e^{-\a t} \,
                       e^{-(p_1 + \a p_3) P^{-1}(y)} \, f \big( (e,x) \big) \, \frac{2(x+y)}{\sqrt{2\pi t^3}} \, e^{-\frac{(x+y)^2}{2t}} \, dt \, dx \, dy \Big) \\
   & = \sqrt{2\a} \, U^W_\a f(0) \, \EV^Q_0 \Big( \int_0^\infty e^{-\sqrt{2\a} y} \, e^{-(p_1 + \a p_3) P^{-1}(y)} \, dy \Big),
  \end{align*}
 where we used, with $z = x+y > 0$, that
  \begin{align*}
   \int_0^\infty e^{-\a t} \, \frac{z}{\sqrt{2\pi t^3}} \, e^{-\frac{z^2}{2t}} \, dt
   & = \int_0^\infty e^{-\a t} \, \frac{\partial}{\partial z} \big( \frac{1}{\sqrt{2\pi t}} \, e^{-\frac{z^2}{2t}} \big) \, dt \\
   & = \frac{\partial}{\partial z} \frac{1}{\sqrt{2\a}} \, e^{-\sqrt{2\a} z} \\
   & = e^{-\sqrt{2\a} z},
  \end{align*}
 and, by the closed form \eqref{eq:G_WP:resolvent of Walsh BB} of the resolvent of $W$, that
  \begin{align*}
  \sum_{e \in \cE} q^e_2 \, 2 \, \int_0^\infty e^{-\sqrt{2\a} x} \, f \big( (e,x) \big) \, dx 
   & = \sqrt{2\a} \, U^W_\a f(0). 
  \end{align*}
 We compute the remaining expectation separately, for $\l := \sqrt{2 \a}$, $\b := p_1 + \a p_3$: 
  \begin{align*}
   \l \int_0^\infty e^{-\l t} \, e^{-\b P^{-1}(t)} \, dt
   & = - \int_0^\infty e^{-\b P^{-1}(t)} \, d e^{-\l t} \\
   %& = \lim_{t \rightarrow \infty} e^{-\b P^{-1}(t)} \, e^{-\l t} - e^{-\b P^{-1}(0)} \, e^{-\l 0} - \int_0^\infty e^{-\l t} \, d e^{-\b P^{-1}(t)} \\
   & = 1 - \b \int_0^\infty e^{-\l t} \, e^{-\b P^{-1}(t)} \, d P^{-1}(t) \\
   & = 1 - \b \int_0^\infty e^{-\l P(t)} \, e^{-\b t} \, d t.
  \end{align*}
 As $P(t-) = P(t)$ a.s., we conclude by using the well-known Laplace transform of the subordinator $P$ (see \cite[Remark 21.6]{Sato13}, \cite[Section II.37]{RogersWilliams1})
 that
  \begin{align*}
   u_2
   & = U^W_\a f(0) \, \Big( 1 - (p_1 + \a p_3) \int_0^\infty \EV^Q_0 \big( e^{-\sqrt{2\a} P(y)} \big) \,  e^{-(p_1 + \a p_3) y} \, dy \Big) \\
   & = U^W_\a f(0) \, \Big( 1 - (p_1 + \a p_3) \int_0^\infty e^{-y ( \sqrt{2\a} p_2 + \int_0^\infty (1 - e^{-\sqrt{2\a} l}) \, p^\Sigma_4(dl) )} \, e^{-(p_1 + \a p_3) y} \, dy \Big) \\
   & = U^W_\a f(0) \, \Big( 1 - \frac{p_1 + \a p_3}{p_1 + \sqrt{2\a} p_2 + \a p_3 + \int_0^\infty  \big( 1 - e^{-\sqrt{2\a} l} \big) \, p^\Sigma_4(dl)} \Big) \\
   & = \frac{\sqrt{2\a} \, p_2 + \int_0^\infty \big( 1 - e^{-\sqrt{2\a} l} \big) \, p^\Sigma_4(dl)}{p_1 + \sqrt{2\a} \, p_2 + \a \, p_3 + \int_0^\infty \big( 1 - e^{-\sqrt{2\a} l} \big) \, p^\Sigma_4(dl)} \cdot U^W_\a f(0).
  \end{align*}
  
 It remains to compute $u_4$: 
 If $p_1 + \a p_3 \neq 0$, then
  \begin{align*}
   u_4
   & = - \frac{p_3 \, f(0)}{p_1 + \a p_3} \EV_0 \Big( \int_0^\infty e^{-\a t} \, d e^{-(p_1 + \a p_3) P^{-1} L_t} \Big) \\
   & = - \frac{p_3 \, f(0)}{p_1 + \a p_3} \EV_0 \Big( \lim_{t \rightarrow \infty} e^{-\a t} \, e^{-(p_1 + \a p_3) P^{-1} L_t}
                                               - e^{-\a 0} \, e^{-(p_1 + \a p_3) P^{-1} L_0}  \\
   & \hspace{8em}                              - \int_0^\infty e^{-(p_1 + \a p_3) P^{-1} L_t} \, d e^{-\a t} \Big) \\
   & = \frac{p_3 \, f(0)}{p_1 + \a p_3} \Big( 1 - \a \, \EV_0 \Big( \int_0^\infty e^{-\a t} \, e^{-(p_1 + \a p_3) P^{-1} L_t} \, dt \Big) \Big),
  \end{align*}
 and observing that the last expectation is just $u_2$ with $f \equiv 1$, we get with $U^W_\a 1 = \frac{1}{\a}$:
  \begin{align*}
   u_4
   & = \frac{p_3 \, f(0)}{p_1 + \a p_3} \Big(1 -  \frac{\sqrt{2\a} \, p_2 + \int_0^\infty \big( 1 - e^{-\sqrt{2\a} l} \big) \, p^\Sigma_4(dl)}
                                                    {p_1 + \sqrt{2\a} \, p_2 + \a \, p_3 + \int_0^\infty \big( 1 - e^{-\sqrt{2\a} l} \big) \, p^\Sigma_4(dl)} \Big) \\
   & = \frac{p_3 \, f(0)}{p_1 + \sqrt{2\a} \, p_2 + \a \, p_3 + \int_0^\infty \big( 1 - e^{-\sqrt{2\a} l} \big) \, p^\Sigma_4(dl)}.
  \end{align*}
 If $p_1 + \a p_3 = 0$, that is, if $p_1 = p_3 = 0$, then $u_4 = 0$ holds by its definition, which is in accord with the above formula for $u_4$.
 
 Adding everything up, we get
  \begin{align*}
     U^\bullet_\a f (0) 
     & = \frac{\int U^{W,D}_\a f(g) \, p_4(dg) + \sqrt{2\a} \, p_2 \, U^W_\a f(0) + p_3 \, f(0)}{p_1 + \sqrt{2\a} \, p_2 + \a \, p_3 + \int_0^\infty \big( 1 - e^{-\sqrt{2\a} l} \big) \, p^\Sigma_4(dl)},
  \end{align*} 
 and insertion of the closed form for $U^W_\a f(0)$ (see equation \eqref{eq:G_WP:resolvent of Walsh BB}) completes the proof.\sqed
\end{proof}

It was already shown in Theorem~\ref{theo:G_IM:X_bullet right process} that $X^\bullet$ is a right process.
By checking the resolvent condition~\eqref{eq:Feller (resolvent)}
with the help of the decomposition given in the above Theorem~\ref{theo:G_IM:resolvent of X}
(the resolvent $(U^{W,D}_\a, \a > 0)$ of the killed Walsh process preserves~$\cC_0(\cG)$ by Example~\ref{ex:G_WP:dirichlet Walsh BB}), we obtain the next result:

\begin{corollary}
 $X^\bullet$ is a Feller process.
\end{corollary}

We finish the construction on the star graph by showing that the process $X^\bullet$ implements the desired boundary conditions:

\begin{theorem} \label{theo:G_IM:generator of X}
 $X^\bullet$ is a Brownian motion on $\cG$. Its generator reads $A^\bullet = \frac{1}{2}\D$ with
 \begin{align*}
   \sD(A^\bullet) = 
    \Big\{ & f \in \cC^2_0(\cG) : \\
    & p_1 \, f(0) - \sum_{e \in \cE} p^e_2 \, f_e'(0) + \frac{p_3}{2} f''(0) - \int \big( f(g) - f(0) \big) \, p_4(dg) = 0 \Big\}.
 \end{align*}
\end{theorem}
%\FINALNEWPAGE
\begin{proof}
 Let $\h_0^\bullet$ be the first entry time of $X^\bullet$ in $0$. As the transformation effects of subsection~\ref{subsec:G_IM:general BB on star graph}
 only take effect after the first hitting of $0$, we have by Lemma~\ref{lem:G_IM:X = W for t <= H0}
  \begin{align*}
   \forall t \leq \h_0^\bullet = \h_0 = \h_0^W: \quad X^\bullet_t = X_t = W_t.
  \end{align*}
 Thus the stopped process $(X^\bullet_{t \wedge \h_0^\bullet}, t \geq 0)$ behaves identically to a stopped Walsh process $(W_{t \wedge \h_0^W}, t \geq 0)$,
 which by Theorem~\ref{theo:G_WP:WP path properties} fulfills the defining conditions~\ref{def:G_BM:BM} of a Brownian motion on the metric graph $\cG$. 
 In addition, $X^\bullet$ is right continuous and strongly Markovian by Theorem~\ref{theo:G_IM:X_bullet right process}, therefore it is a Brownian motion on the star graph $\cG$.

 In view of Lemma~\ref{lem:G_BM:totality on star graph}, we only need to show that the domain of the generator lies inside the right-hand set.
 As $X^\bullet$ is Feller, $\sD(A^\bullet) = U^\bullet_\a \big( \cC_0(\cG) \big)$ holds true for any~$\a > 0$, 
 so it is enough to prove that every potential $U^\bullet_\a f$, $f \in \cC_0(\cG)$, satisfies the above-stated boundary condition:
 The derivatives of $U^{W,D}_\a f$ were already computed in Example~\ref{ex:G_WP:dirichlet Walsh BB} (it is $f_e'(0) = f'(e,0+)$ there),
 so the first formula of Theorem~\ref{theo:G_IM:resolvent of X}
 gives for $g = (e,x) \in \cG$, by setting $\psi_\a(g) := e^{-\sqrt{2\a} x}$:
  \begin{align*}
   U^\bullet_\a f_e'(0) 
    & = U^{W,D}_\a f_e'(0) + \psi_\a'(0) \, U^\bullet_\a f(0) \\
    & = 2 \int_0^\infty e^{-\sqrt{2\a} x} \, f(e,x) \, dx - \sqrt{2\a} \, U^\bullet_\a f(0), \\
   U^\bullet_\a f''(0) 
    & = U^{W,D}_\a f''(0) + \psi_\a''(0) \, U^\bullet_\a f(0) \\
    & = - 2 f(0) + 2\a \, U^\bullet_\a f(0), \\    
   U^\bullet_\a f(g) - U^\bullet_\a f(0) 
    & = U^{W,D}_\a f(g) - \big( 1 - e^{-\sqrt{2\a} x} \big) \, U^\bullet_\a f(0). 
  \end{align*}
 By using these relations and then inserting the closed form of $U^\bullet_\a f(0)$ as given in Theorem~\ref{theo:G_IM:resolvent of X}, we obtain
  \begin{align*}
   & p_1 \, U^\bullet_\a f(0) - \sum_{e \in \cE} p^e_2 \, U^\bullet_\a f_e'(0) + \frac{p_3}{2} \, U^\bullet_\a f''(0) - \int \big( U^\bullet_\a f(g) - U^\bullet_\a f(0) \big) \, p_4(dg)  \\
  % & \quad = p_1 U^\bullet_\a f(0) + \sqrt{2 \a} \sum_{e \in \cE} p^e_2 U^\bullet_\a f(0) + 2\a \frac{p_3}{2} U^\bullet_\a f(0) + \int (1 - e^{-\sqrt{2\a} x}) \, p^\Sigma_4(dx) U^\bullet_\a f(0) & \\
  % & \quad ~~ - \sum_{e \in \cE} p^e_2 U^K_\a f'(e,0+) + \frac{p_3}{2} U^K_\a f''(0) - \int U^K_\a f(g) \, p_4(dg) & \\
   & \quad = \Big( p_1 + \sqrt{2 \a} \, p_2 + \a \, p_3 + \int \big( 1 - e^{-\sqrt{2\a} x} \big) \, p^\Sigma_4(dx) \Big) \cdot U^\bullet_\a f(0)  \\
   & \quad ~~~ - \Big( 2 \, \sum_{e \in \cE} p^e_2 \int_0^\infty e^{-\sqrt{2\a} x} \, f(e,x) \, dx + \frac{p_3}{2} \, 2 \, f(0) +  \int U^{W,D}_\a f(g) \, p_4(dg) \Big)  \\
   & \quad = 0. \qedhere
  \end{align*}
\end{proof}

\begin{appendix}

\section{Preliminaries on Brownian Motions}\label{app:Preliminaries}
  
As the fundamental process of this article, we define the Brownian motion on~$\R$ in the setting of Markov processes:
  
\begin{definition} \label{def:B_BM:def}
 A continuous, strong Markov process
  \begin{align*}
   B = \big( \O, \sG, (\sG_t)_{t \geq 0}, (B_t)_{t \geq 0}, (\T_t)_{t \geq 0}, (\PV_x)_{x \in \R} \big)
  \end{align*}
 on $\R$ with transition semigroup 
  \begin{align*}
   T^B_t f(x) = \int_{\R} f(y) \, \frac{1}{\sqrt{2\pi t}} \, e^{-\frac{(y-x)^2}{2t}} \, dt, \quad x \in \R, ~ t \geq 0, ~ f \in b\sB(\R),
  \end{align*}
 is called \textdef{(standard) Brownian motion on $\R$}.
\end{definition}

The resolvent of the Brownian motion is well known 
(see, e.g., \cite[Section 2.16]{Dynkin65} and \cite[Exercise III.3.13, Example III.6.9]{RogersWilliams1}),
it is given by
  \begin{equation} \label{eq:B_BM:generator}
  \begin{aligned}
    \forall \a > 0, f \in b\sB(\R), x \in \R: \quad U^B_\a f(x)
    & = \int_{\R} \frac{1}{\sqrt{2 \a}} \, e^{-\sqrt{2 \a} \, \abs{y-x}} \, f(y) \, dy \\
   % & = \frac{1}{\sqrt{2 \a}} \, e^{-\sqrt{2 \a} \, x} \int_{-\infty}^x  \, e^{\sqrt{2 \a} \, y} \, f(y) \, dy 
   %   + \frac{1}{\sqrt{2 \a}} \, e^{\sqrt{2 \a} \, x} \int_x^\infty  \, e^{- \sqrt{2 \a} \, y} \, f(y) \, dy
  \end{aligned}  
  \end{equation}
An easy analysis shows the following:
\begin{lemma} \label{lem:B_BM:resolvent preserves C}
  The resolvent of $B$ admits $U^B b\cC(\R) \subseteq b\cC(\R)$ and $U^B \cC_0(\R) \subseteq \cC^2_0(\R)$.
\end{lemma}

\subsection{Brownian Motions on the Half Line}
In the introduction, we already listed some of the Brownian motions on $\R_+$. 
Two of them are going to be useful auxiliary processes for us, so we take a closer look at them:

\begin{example} \label{ex:B_HL:reflecting BB on half line}
 Mapping the Brownian motion $B$ on $\R$ to $\R_+$ by the absolute-value norm~$\abs{\,\cdot\,}$ results in the \textdef{reflecting Brownian motion} $(\abs{B_t}, t \geq 0)$ on $\R_+$,
 which satisfies for all
 %which is a Brownian motion on $\R_+$ in the sense of definition~\ref{def:B_HL:BM}.
 %This is done rigorously, for example, with the help of Theorem \ref{theo:C_MS:mapping, standard} (see \cite[Example~10.26]{Dynkin65} or \cite[Section~I.14]{RogersWilliams1}),
 %and we then get for 
 $t \geq 0$, $x \in \R_+$, $A \in \sB(\R_+)$:
  \begin{align*}
    \PV_x ( \abs{B_t} \in A ) 
    & = \PV_x ( B_t \in A ) + \PV_x ( B_t \in -A ) \\
    & = \PV_x ( B_t \in A ) + \PV_{-x} ( B_t \in A ).
  \end{align*}
 %With this, the resolvent of $\abs{B}$ can be derived from the resolvent of $B$. %with the help of Dynkin's formula \ref{theo:A_SM:Dynkins formula (resolvent)}.
 In particular, the resolvent of $|B|$ reads at the origin:
  \begin{align*}
   \forall \a > 0, f \in b\sB(\R_+): \quad
   U^{\abs{B}}_\a f(0)
   %& = 2 \, U^{B}_\a f^+(0) \\
   & = 2 \int_0^\infty \frac{1}{\sqrt{2\a}} \, e^{-\sqrt{2\a}y} \, f(y) \, dy.
  \end{align*}
 %where we used an auxiliary function $f^+ \colon \R \rightarrow \R$, defined by $f^+(y) = f(y)$ for $y \geq 0$ and $f^+(y) = 0$ otherwise,
 %as well as the closed formula \eqref{eq:B_BM:generator} for the resolvent $U^B$ of $B$.
\end{example}

\begin{example} \label{ex:B_HL:dirichlet BB on half line}
 Let ${(\abs{B_t}, t \geq 0)}$ be the reflecting Brownian motion on $\R_+$ with its first hitting time $\h_0$ of the origin,
 and consider the process resulting from killing~$\abs{B}$ at~$\h_0$, resulting in the \textdef{killed Brownian motion} on $\R_+$:
 %Instead of reflecting the ``Brownian particle'' at the origin, we can just let it ``disappear'' there, which results in the .
 %To this end, 
  \begin{align*}
    B^{[0,\infty)}_t :=
      \begin{cases}
        \abs{B_t}, & t < \h_0, \\
        \D, & t \geq \h_0.
      \end{cases}
  \end{align*}
  
 %$B^{[0,\infty)}$ is not a Brownian motion on $\R_+$ in the sense of our definition,
 %because it is not normal at $0 \in \R_+$:
  %\begin{align*}
  % \PV_0(B^{[0,\infty)}_t = \D) = 1.   
  %\end{align*}
 This process is not normal at $0$ (as $\PV_0(B^{[0,\infty)}_t = \D) = 1$).
 However, it is a right process on $\R_{> 0} = (0, \infty)$ by \cite[Corollary~12.24]{Sharpe88}.
% Thus, it is not really in the scope of our work and will not be treated extensively (for more results on killed 
% Brownian motions, the reader may consult, e.g., \cite[Chapter 2]{ChungZhao95}).
% Nonetheless, it is going to be a supporting process in some of our computations,
% so we will examine this process a bit further: 

Using the convention $f(\D) = 0$ for all functions $f$, the resolvent of $B^{[0,\infty)}$ can be computed
 with the help of Dynkin's formula \eqref{eq:Dynkins formula (resolvent)}. The decomposition of the one-dimensional Brownian motion $B$
 at the stopping time $\h_0$ gives for $x \geq 0$
% \begin{align*}
%  U^B_\a f(x)
%  & = \EV_x \Big( \int_0^{\h_0} e^{-\a t} f( \abs{B_t} ) \, dt \Big) + \EV_x \big( e^{-\a \h_0} \big) \, U^B_\a f(0),
% \end{align*}
% which is equivalent to
 \begin{equation} \label{eq:B_HL:dirichlet resolvent}
 \begin{aligned} 
  U^{[0,\infty)}_\a f(x)
  & = \EV_x \Big( \int_0^{\h_0} e^{-\a t} f(\abs{B_t}) \, dt \Big) \\
  & = U^B_\a f(x) - \EV_x \big( e^{-\a \h_0} \big) \, U^B_\a f(0),
 \end{aligned}  
 \end{equation}
 where we interpret the function $f$ inserted into $U^B_\a f$ as an arbitrary continuation of~$f \in b\sB( [0,\infty) )$ to $b\sB(\R)$. 
 With the passage time formula (cf.~\cite[Section~1.7]{ItoMcKean74})
  \begin{align}\label{eq:BM passage time}
   \EV_x \big( e^{-\a \h_0} \big) = e^{-\sqrt{2 \a} x}
  \end{align}
 and Lemma~\ref{lem:B_BM:resolvent preserves C}, we get 
  \begin{align*}
   U^{[0,\infty)} b\cC(\R_+) \subseteq b\cC(\R_+) \quad \text{and} \quad U^{[0,\infty)} \cC_0(\R_+) \subseteq \cC^2_0(\R_+).
  \end{align*}
 %where we set
 % \begin{align*}
 %  \cC^D_0(\R_+) := \big\{ f \in \cC_0(\R_+): \lim_{x \downarrow 0} f(x) = 0 \big\}.
 % \end{align*}
% Differentiating \eqref{eq:B_HL:dirichlet resolvent} twice yields
% \begin{align*}
%  U^{[0,\infty)}_\a f'(x)  & = U^B_\a f'(x)  + \sqrt{2 \a} \, e^{-\sqrt{2 \a} \, x} \, U^B_\a f(0), \\
%  U^{[0,\infty)}_\a f''(x) & = U^B_\a f''(x) - 2 \a        \, e^{-\sqrt{2 \a} \, x} \, U^B_\a f(0) \\
%                & = 2 \big( \a \, U^{[0,\infty)}_\a f(x) - f(x) \big),
% \end{align*}
% and using the derivatives of $U^B$, calculated in subsection \ref{subsec:B_BM:generator resolvent}, results in
% \begin{align*}
%  U^{[0,\infty)}_\a f(0)   & = 0, \\
%  U^{[0,\infty)}_\a f'(0)  & = 2 \int_0^\infty e^{-\sqrt{2 \a} \, y} \, f(y) \, dy, \\
%  U^{[0,\infty)}_\a f''(0) & = -2 f(0).
% \end{align*} 
% Another review of \eqref{eq:B_HL:dirichlet resolvent} gives, by using the closed form \eqref{eq:B_BM:generator} for $U^B$,
%  \begin{align*}
%   U^{[0,\infty)}_\a f(x) = \frac{1}{\sqrt{2\a}} \, \int_0^\infty \Big( e^{-\sqrt{2\a} \, \abs{x - y}} - e^{-\sqrt{2\a} \, (x + y)} \Big) \, f(y) \, dy,
%  \end{align*}
% which is in accordance with Andr\'e's reflection principle: 
%  \begin{align*}
%   \PV_x \big( B^{[0,\infty)}_t \in A \big) & = \PV_x (B_t \in A) - \PV_x(B_t \in -A), \quad A \in \sB(\R_+). \qedhere
%  \end{align*}
  
 More results on killed Brownian motions can be found in~\cite[Chapter~2]{ChungZhao95}).
\end{example}

These examples depict the easiest boundary behaviors at the origin, a complete examination can be found in~\cite{ItoMcKean63}. 
Closed forms for the resolvents and semigroups are known for any possible local boundary condition in the half-line case, 
see \cite[Section~4]{KPS10} and \cite[Section~9.1]{Taira14}.

\subsection{Walsh Processes and Walsh Brownian Motions on the Star Graph}\label{app:Walsh}

The first non-trivial examples of Brownian motions on star graphs are processes  which only feature 
``skew'' effects at the origin in the following sense: 
Take the excursions from the origin of a reflecting Brownian motion, and for each excursion choose an edge independently with respect to
a distribution $\mu := \sum_{e \in \cE} p^e_2 \, \e_e$ on the edges $\cE$ (see figure~\ref{fig:G_WP:Walsh process}).
Such processes are a extension of ``skew Brownian motions'' on $\R$, constructed by It\^{o} and McKean in~\cite[Section~17]{ItoMcKean63}, to the graph setting,
and restrict Walsh Brownian motions~\cite{Walsh78} on~$\R^2$ from a general angular distribution~$\mu$ on~$[0, 2\pi)$ to the edge space~$\cE$.
As we will only consider processes on graphs, there will be no confusion when we use the term ``Walsh Brownian motions'' for the restriction
of the ``general Walsh processes'' on~$\R^2$ to the star graph case.

\begin{figure}[tb]
   \includegraphics[width=0.9\textwidth,keepaspectratio]{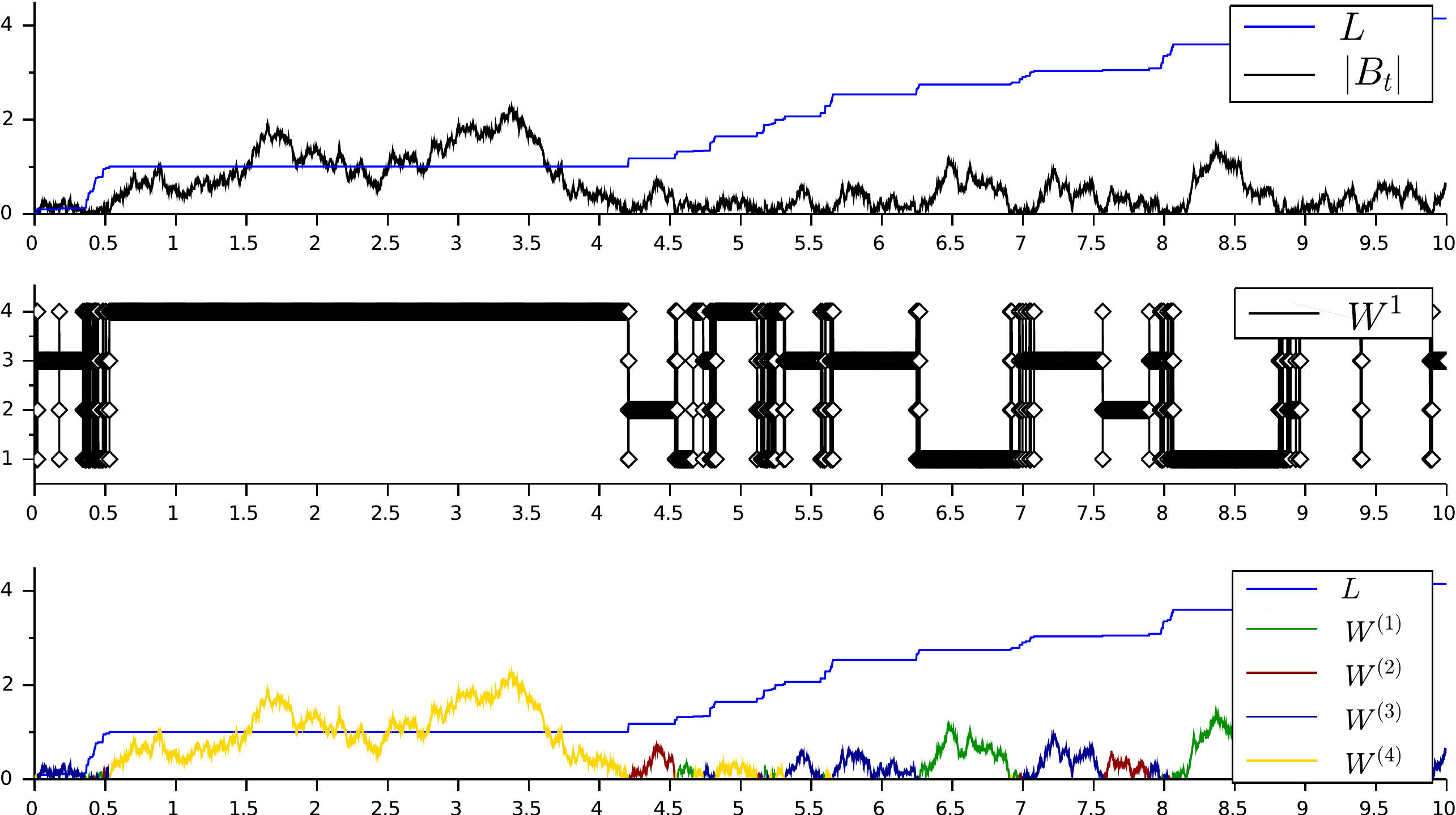}
   \caption[Construction of a Walsh Brownian motion on a star graph]
           {Construction of a Walsh Brownian motion on a star graph
            by choosing for each excursion of $\abs{B}$ an edge independently with respect to some distribution $\mu$, resulting in the edge process~$W^1$.
            Then $\big((W^1_t, \abs{B_t}), t \geq 0 \big)$ is a Walsh Brownian motion with local time $L$ at the star vertex.
            The parts $W^{(i)}$ in the above graph indicate on which of the edges $e_i \in \cE$, $i \in \{1,2,3,4\}$,
            the Walsh Brownian motion is currently running.} \label{fig:G_WP:Walsh process}
\end{figure}

Let $\cG = \{0\} \cup \bigcup_{e \in \cE} \big( \{e\} \times (0, \infty) \big)$ be a star graph with vertex $0 \equiv \{ (e,0), e \in \cE \}$.
As the Walsh Brownian motion is only defined illustratively or with the help of excursion theory in most of the older works, we follow
\cite[Definition~2.1]{FitzsimmonsKuterWalsh} for a rigorous context:

\begin{definition} \label{def:B_WP:Walsh BM}
 A strong Markov process $W = (W^1, W^2)$ on $\cG$ is a \textdef{Walsh Brownian motion} (or \textdef{Walsh process}) on $\cG$ with \textdef{weights} $(p^e_2, e \in \cE)$,
 if $p^e_2 \geq 0$ for all $e \in \cE$ and $\sum_{e \in \cE} p^e_2 = 1$, and with  $\mu := \sum_{e \in \cE} p^e_2 \, \e_e$, the process $W$ satisfies:
  \begin{enumerate}
   \item $W^2$ is a reflecting Brownian motion on $\R_+$;
   \item if $W_0 = 0$, then for $t > 0$, the distribution of $W^1_t$ is given by $\mu$;
   \item if $W_0 = (e,x)$ with $x > 0$, then $W^1_t = e$ holds for all $t < \h_0$, and on $t > \h_0$, the distribution of $W^1_t$ is equal to $\mu$
         and independent of $(W^2_t, t \geq 0)$.
  \end{enumerate}
\end{definition}

\cite{BarlowPitmanYor89} contains a list of various existence proofs, such as an approach via the implied semigroup.
Others may proceed via the application of It\^{o} excursion theory \cite{Ito72},
generalized from the skew Brownian motion \cite[Example 5.7]{Salisbury86} to the star graph case. 
\cite{Lejay06} gives a comprehensive survey on construction methods for skew Brownian motions. 
Details on the construction in the context of star graphs can also be found in \cite[Section 2]{FitzsimmonsKuterWalsh}.

For a Walsh process $W = (W^1, W^2)$, we will denote the ``radial process'' by
 \begin{align*}
  \abs{W_t} := W^2_t, \quad t \geq 0.
 \end{align*}

The semigroup of the Walsh Brownian motion can be obtained using its strong Markov property at the first hitting time of the vertex.
The process then decomposes into a one-dimensional Brownian motion on the starting edge killed on hitting the origin,
followed by a reflecting Brownian motion on the edges chosen by the weight distribution $\mu = \sum_{e \in \cE} p^e_2 \, \e_e$. 
The closed form of the semigroup is given in \cite[Equations (2.1)--(2.2)]{Walsh78} in a more general context. By inserting a discrete distribution~$\mu$ on~$\cE$,
we get:

\begin{lemma} \label{lem:G_WP:semigroup}
 The semigroup $(T^W_t, t \geq 0)$ of the Walsh process reads
 for all $f \in b\sB(\cG)$, $t \geq 0$, $(l,x) \in \cG$:
  \begin{align*}
    T^W_t f(l,x)
    & = \sum_{e \in \cE} \, p^e_2 \Big( T^{\abs{B}}_t f(e, \,\cdot\,) 
         +  T^{[0,\infty)}_t \big( f(l,\,\cdot\,) - f(e, \,\cdot\,) \big) \Big) (x),
  \end{align*}
 with $\big( T^{\abs{B}}_t, t \geq 0 \big)$, $\big( T^{[0,\infty)}_t, t \geq 0 \big)$ being the semigroups of the reflecting Brownian motion,
 the standard Brownian motion killed when hitting the origin respectively,  
 as introduced in Examples~\ref{ex:B_HL:reflecting BB on half line} and~\ref{ex:B_HL:dirichlet BB on half line}.
\end{lemma}

In particular, we have $T^W_t f(0) = \sum_{e \in \cE} \, p^e_2 \, T^{\abs{B}}_t f(e, \,\cdot\,) (0)$, so 
the resolvent of the Walsh process at the star vertex $0$ is obtained with the help of Example~\ref{ex:B_HL:reflecting BB on half line}:
\begin{equation} \label{eq:G_WP:resolvent of Walsh BB}
\begin{aligned} 
 U^W_\a f(0)
 %& = \sum_{e \in \cE} \, p^e_2 \, U_\a^{\abs{B}} f(e, \,\cdot\,) (0) \\
 & = \sum_{e \in \cE} \, p^e_2 \, \frac{2}{\sqrt{2\a}} \int_0^\infty e^{-\sqrt{2\a} x} f(e,x) \, dx.
\end{aligned}
\end{equation}

As the semigroups of reflected and killed Brownian motion are Feller semigroups, 
the Feller property of the Walsh Brownian motion is follows (cf.~\cite[Theorem~2.1]{BarlowPitmanYor89}). 
By the closed form \eqref{eq:G_WP:resolvent of Walsh BB} of the resolvent, the following is immediate:

\begin{theorem}
 Every Walsh process on a star graph is a Feller process.
 Its generator reads $A = \frac{1}{2} \D$, with domain
 \begin{align*}
  \sD(A) = \big\{ f \in \cC_0^2(\cG): \sum_{e \in \cE} p^e_2 f_e'(0) = 0 \big\}.
 \end{align*}
\end{theorem}

We will always work with a continuous version of the Walsh Brownian motion, which exists 
by~\cite[Lemmas~2.2, 2.3, Theorem~2.4]{BarlowPitmanYor89}:

\begin{theorem} \label{theo:G_WP:WP path properties}
 There exists a version $(W_t, t \geq 0)$ of the Walsh Brownian motion on the star graph~$\cG$ which
 is continuous, and for which $(\abs{W_t}, t \geq 0)$ is a reflecting Brownian motion on $\R_+$.
\end{theorem}

Therefore, properties which only depend on $\abs{W}$ or on the behavior of $W$ on one edge can be derived from the respective properties
of a Brownian motion on $\R$ or on $\R_+$. For instance, passage time formulas like~\eqref{eq:BM passage time} can be used in appropriate
cases for the Walsh Brownian motion $W$ as well.

%\cite[Lemme 2.4]{Najnudel07}

As, conditional on $\PV_0$, the edge process $(W^1_t, t \geq 0)$
is independent of the radial process $(W^2_t = \abs{W_t}, t \geq 0)$ (and thus of its local time), the 
following result is a direct consequence of 
the well-known joint distribution of a reflecting Brownian motion and its local time at zero
(cf.~\cite[Proposition~2.8.15]{KaratzasShreve91}):

\begin{lemma} \label{lem:G_WP:joint dist W LT}
 The joint distribution of $(W, L)$ with respect to $\PV_0$ given by
  \begin{align*}
    \EV^W_0 \big( f(W_t, L_t) \big)
    & = \sum_{e \in \cE} p^e_2 \int_0^\infty \int_0^\infty f \big( (e,x),y \big) \, \frac{2(x+y)}{\sqrt{2\pi t^3}} \, e^{-\frac{(x+y)^2}{2t}} \, dx \, dy,
  \end{align*}
 for $f \in b\big(\sB(\cG) \otimes \sB(\R_+)\big)$, $t > 0$.
\end{lemma}

\begin{example} \label{ex:G_WP:dirichlet Walsh BB}
 Consider the ``Dirichlet Walsh process'' $W^D$, that is the Walsh process~$W$ killed at the first hitting time $\h_0 := \inf \{ t \geq 0: W_t = 0 \}$: %of the star vertex $0$: 
  \begin{align*}
    W^D_t :=
      \begin{cases}
        W_t, & t < \h_0, \\
        \D,  & t \geq \h_0.
      \end{cases}
  \end{align*}
 As the Walsh process $W$ just behaves like a standard (reflecting) Brownian motion on the starting edge until hitting the star vertex,
 the Dirichlet Walsh process~$W^D$, with fixed starting edge, equals the Dirichlet process $B^{[0, \infty)}$ on the half line. %(see Example~\ref{ex:B_HL:dirichlet BB on half line}).
 Therefore, %when identifying $\{ (e,\D), e \in \cE \} \equiv \D$, 
 we get $\PV_{(e,x)}$-a.s.\ for any $(e,x) \in \cG$:
  \begin{align*}
   \forall t \geq 0: W^D_t = \big(e, B^{[0, \infty)}_t \big).
  \end{align*}
 Thus, the resolvent of $W^D$ reads, for $\a > 0$, $f \in b\sB(\cG)$, $(e,x) \in \cG$,
  \begin{align*}
   U^{W,D}_\a f(e,x) = U^{[0,\infty)}_\a \big(f(e, \,\cdot\,)\big) (x).
  \end{align*}
 Our findings of Example~\ref{ex:B_HL:dirichlet BB on half line} imply that $\big( U^{W,D}_\a, \a > 0 \big)$ preserves $\cC_0(\cG)$.
 Furthermore, they give 
 \begin{align*}
  U^{W,D}_\a f'(e, 0+) & = 2 \int_0^\infty e^{-\sqrt{2\a} x} \, f(e,x) \, dx, \\
  U^{W,D}_\a f''(0) & = - 2 f(0).
 \end{align*}
% The domain of the generator then reads  
% \begin{align*}
%  \sD(A^D) =  \big\{ f \in \cC_0^2(\cG) : f(0) = 0 \big\}.
% \end{align*}
% 
 For later use, we also remark that for all $(e,x) \in \cG$,
  \begin{align*}
   U^{W,D}_\a 1 (e,x) 
   & = \EV^B_x \Big( \int_0^{\h_0} e^{-\a t} \, dt \Big)
     = \frac{1}{\a} \, \EV^B_x \big( 1 - e^{-\a \h_0} \big)
     = \frac{1}{\a} \, \big( 1 - e^{-\sqrt{2\a} x} \big). \qedhere
  \end{align*}
\end{example}

\subsection{On the Local Time of Brownian Motion}
  
An essential tool in the study of the Brownian sample paths
is the Brownian local time, or ``\textit{mesure du voisinage}'', as it was coined when first introduced by L\'evy in~\cite{Levy65}.
Brownian local time is the source of many deep and outstanding results, such as the Ray\textendash{}Knight theorems.
However, we will ``only'' resort to one main result by L\'evy in our work, which we need to extend to initial measures other than $\PV_0$.

Let $(L_t, t \geq 0)$ be the local time of the standard Brownian motion $B$ on $\R$ (see~\cite{Trotter58} or~\cite[Section~3.6]{KaratzasShreve91}),
that is, a perfect continuous additive functional, adapted to the Brownian filtration, such that 
% \begin{definition}  \label{def:B_LT:local time}
%   The family of random variables $\big( L_t(x), t \geq 0, x \in \R \big)$ with values in $[0, \infty)$ is called
%   \textdef{local time} for the Brownian motion $B$, if for every $t \geq 0$, $x \in \R$, $L_t(x)$ is $\sF_t$-measurable, and a.s.,
%   $(t, x) \mapsto L_t(x)$ is continuous and satisfies
%    \begin{align*}
%     \forall t \geq 0, A \in \sB(\R): \quad  \int_A L_t(x) \, dx = \lambda \big( \{s \leq t: B_s \in A \} \big).
%    \end{align*}
%   The local time at the origin is denoted by $L = (L_t, t \geq 0)$ with $L_t := L_t(0)$.
% \end{definition}
   \begin{align*}
    \forall t \geq 0: \quad L_t = \lim_{\e \downarrow 0} \frac{1}{2\e} \, \lambda \big( \{s \leq t: \abs{B_s} \leq \e \} \big).
   \end{align*}
% Existence of the local time is not trivial, its proof is usually named Trotter's theorem \cite{Trotter58}, see also \cite[Theorem 3.6.11]{KaratzasShreve91}
% for a modern approach.
% 
% The following properties of the Brownian local time are immediate from its definition:
% 
% \begin{lemma} \label{lem:B_LT:basic properties}
%  For every $x \in \R$, $(L_t(x), t \geq 0)$ is a perfect continuous additive functional, and satisfies
%   \begin{align*}
%    \forall t \geq 0: \quad L_t(x) = \lim_{\e \downarrow 0} \frac{1}{2\e} \, \lambda \big( \{s \leq t: \abs{B_s - x} \leq \e \} \big).
%   \end{align*}
% \end{lemma}%

 We will mainly use a part of the celebrated characterization by L\'evy, as given in \cite[Theorem~3.6.17]{KaratzasShreve91}:
 
 \begin{theorem} \label{theo:B_LT:levys local time}
  Let $B$ be a Brownian motion with local time $L$ at the origin.
  Then, started at the origin, the process $\tB_t := -\int_0^t \sgn(B_s) \, dB_s$, $t \geq 0$, is a Brownian motion.
  Define its running maximum process $\tM_t := \max_{s \leq t} \tB_s$, $t \geq 0$.
  Then,
   \begin{align*}
     \PV_0 \big( \forall t \geq 0: \abs{B_t} = \tM_t - \tB_t,  L_t = \tM_t \big) = 1.
   \end{align*}
  In particular, for a Brownian motion $B$ with local time $L$ at the origin and running maximum process $M_t := \max_{s \leq t} B_t$, $t \geq 0$,
  the processes ${\big( (M_t - B_t, M_t), t \geq 0 \big)}$ and $\big( ( \abs{B_t}, L_t), t \geq 0 \big)$ have the same law under $\PV_0$.
 \end{theorem}
 
An immediate consequence is that,
because 
 \begin{align*}
  \abs{B_t} - L_t = \tM_t - \tB_t - \tM_t = -\tB_t, \quad t \geq 0, \quad \text{$\PV_0$-a.s.},
 \end{align*}
the process $(\abs{B_t} - L_t, t \geq 0)$ is a Brownian motion under $\PV_0$. We will extend this result
to initial laws other than $\PV_0$.

We start by examining the pseudo-inverses of the local time $(L_t, t \geq 0)$:
 \begin{align*}
  L^{-1}(a)   & = \inf \{ t \geq 0: L_t > a \}, \quad a \geq 0, \\
  L_-^{-1}(a) & = \inf \{ t \geq 0: L_t \geq a \}, \quad a \geq 0.
 \end{align*}
The following basic properties will be very helpful later:
\begin{lemma}  \label{lem:B_LT:inverse local time is stopping time}
 For every $a \in \R_+$, $L^{-1}(a)$ is an $(\sF_{t+}, t \geq 0)$-stopping time and $L_-^{-1}(a)$ is an $(\sF_{t}, t \geq 0)$-stopping time.
\end{lemma}
\begin{proof}
 By \ref{itm:G_IM:props pseudoinverse (cont. case):inversion} and \ref{itm:G_IM:props lc pseudoinverse (cont. case):inversion} of Lemma~\ref{lem:G_IM:props pseudoinverse (cont. case)},
 we have for all $t \geq 0$,
  \begin{align*}
    \{ L^{-1}(a) < t \} 
    & = \{ a < L_t \} \in \sF_t, \\
    \{ L_-^{-1}(a) \leq t \} 
    & = \{ a \leq L_t \} \in \sF_t. \qedhere
  \end{align*}
\end{proof}

\begin{lemma} \label{lem:B_LT:shift of inverse local time}
  For $a \in \R_+$
  and any random time $\t \leq L^{-1}(a)$ with $L(\t) = 0$ a.s.,
  \begin{align*}
   L^{-1}(a) \circ \T_\t & = L^{-1}(a) - \t, \\
   L_-^{-1}(a) \circ \T_\t & = L_-^{-1}(a) - \t
  \end{align*}  
  hold a.s.\ true.
\end{lemma}
\begin{proof}
  Let $\t$ be as above. Then, a.s.,
  \begin{align*}
    L^{-1}(a) - \t
    %& = \inf \{ u \geq 0: L(u) > a \} - \t \\
    & = \inf \{ u \geq \t: L(u) > a \} - \t \\
    & = \inf \{ u \geq 0: L(u + \t) - L(\t) > a \} \\
    & = L^{-1}(a) \circ \T_\t,
  \end{align*}
 where we used $L_t \leq a$ for all $t \leq L^{-1}(a)$ for the first identity.
 
 The computation for $L_-^{-1}(a)$ proceeds completely analogously.\sqed
\end{proof}

\begin{lemma} \label{lem:B_LT:lc rc inverse}
 For any $a > 0$, $L_-^{-1}(a) = L^{-1}(a)$ a.s.\ holds true.
\end{lemma}
\begin{proof}
 \cite[Lemma 3.6.18]{MarcusRosen06} shows that $L_-^{-1}(a) = L^{-1}(a)$ holds $\PV_0$-a.s.\ for any $a > 0$.
 For a general initial law $\PV$, we compute by using Lemma~\ref{lem:B_LT:shift of inverse local time} (as $\h_0 < L^{-1}(a)$ for any $a > 0$)
 and the strong Markov property of $B$:
  \begin{align*}
   \PV \big( L_-^{-1}(a) = L^{-1}(a) \big)
   & = \PV \big( L_-^{-1}(a) \circ \T_{\h_0} = L^{-1}(a) \circ \T_{\h_0} \big) \\
   & = \PV \big( \PV_{\h_0} \big( L_-^{-1}(a) = L^{-1}(a) \big) \big) \\
   & = \PV_0 \big( L_-^{-1}(a) = L^{-1}(a) \big) = 1. \qedhere
  \end{align*}
\end{proof}

The inverses of the local time have a close relation to the first hitting times $\h_a$ of points $a \geq 0$ 
(see, e.g., \cite[Theorem 5.9]{Cinlar11}), which appears natural in view of L\'evy's characterization.
We will only note the following formula for later use:

\begin{lemma} \label{lem:B_LT:laplace of local time inverse}
 For all $x, a \in \R_+$,
 \begin{align*}
  \EV_x \big( e^{-\a L^{-1}(a)} \big) = e^{-\sqrt{2 \a} (x+a)}.
 \end{align*}
\end{lemma}
\begin{proof}
 For $x = 0$, this is proved in \cite[Lemma B.1]{KPS10} or found in the collection of results of \cite[Theorem 6.2.1]{KaratzasShreve91}.
 For $x \neq 0$, by using $L^{-1}(a) = \h_0 + L^{-1}(a) \circ \T_{\h_0}$ of Lemma~\ref{lem:B_LT:shift of inverse local time}, we get
 \begin{align*}
  \EV_x \big( e^{-\a L^{-1}(a)} \big) 
  & = \EV_x \big( e^{-\a \h_0} \, \EV_x \big( e^{-\a L^{-1}(a)} \circ \T_{\h_0} \big| \sF_{\h_0} \big) \big)  \\
  & = \EV_x \big( e^{-\a \h_0} \big) \, \EV_0 \big( e^{-\a L^{-1}(a)} \big),
 \end{align*}
 and insertion of the values for both expectations completes the proof.\sqed
\end{proof}

\begin{lemma} \label{lem:B_LT:extension to levys char}
  For all $x, a \in \R_+$, $\a > 0$, $f \in b\sB(\R)$,
 \begin{align*}
   \EV_0 \Big( \int_0^{L^{-1}(a)} e^{-\a t} \, f \big( \abs{B_t} - L_t + a \big) \, dt \Big)
   = \EV_a \Big( \int_0^{\h_0} e^{-\a t} \, f \big( \abs{B_t} \big) \, dt \Big).
 \end{align*}
\end{lemma}
\begin{proof}
 We have $B_{L^{-1}(a)} = 0$, as $L$ only grows when $B$ is at the origin,
 and using the additive functional property and the continuity of $L$, we get
  \begin{align*}
   L_{t + L^{-1}(a)} = L_t \circ \T_{L^{-1}(a)} + L \big( L^{-1}(a) \big) \quad \text{with} \quad L \big( L^{-1}(a) \big) = a.
  \end{align*}
 Thus, Dynkin's formula \eqref{eq:Dynkins formula (resolvent)} applied for the stopping time $L^{-1}(a)$ yields
  \begin{align*}
   & \EV_0 \Big( \int_0^\infty e^{-\a t} \, f \big( \abs{B_t} - L_t + a \big) \, dt \Big) & \\
   & = \EV_0 \Big( \int_0^{L^{-1}(a)} e^{-\a t} \, f \big( \abs{B_t} - L_t + a \big) \, dt \Big) & \\
   & \quad + \EV_0 \Big( e^{-\a L^{-1}(a)} \, \EV_0 \Big( \int_0^\infty e^{-\a t} \, f \big( \abs{B_{t + L^{-1}(a)}} - L_{t + L^{-1}(a)} + a \big) \, dt \, \Big| \, \sF_{L^{-1}(a)} \Big) \Big) & \\
   & = \EV_0 \Big( \int_0^{L^{-1}(a)} e^{-\a t} \, f \big( \abs{B_t} - L_t + a \big) \, dt \Big) & \\
   & \quad + \EV_0 \Big( e^{-\a L^{-1}(a)} \, \EV_0 \Big( \int_0^\infty e^{-\a t} \, f \big( \abs{B_{t}} - L_{t} \big) \, dt \Big) \Big).
  \end{align*}
 With Theorem~\ref{theo:B_LT:levys local time}, the translation formula~\eqref{eq:centering translation for Levy} and Lemma~\ref{lem:B_LT:laplace of local time inverse},
 we obtain 
   \begin{align*}
   & \EV_0 \Big( \int_0^{L^{-1}(a)} e^{-\a t} \, f \big( \abs{B_t} - L_t + a \big) \, dt \Big) & \\
   & = \EV_a \Big( \int_0^\infty e^{-\a t} \, f \big( B_t \big) \, dt \Big) - \EV_a \Big( e^{-\a \h_0} \, \EV_0 \Big( \int_0^\infty e^{-\a t} \, f \big( B_t \big) \, dt \Big) \Big).
  \end{align*}
 Another application of Dynkin's formula~\eqref{eq:Dynkins formula (resolvent)} for $\h_0$ yields the result,
 as $B_{\h_0} = 0$ by the continuity of $B$, and $B_t = \abs{B_t}$ $\PV_a$-a.s.\ for all $t \leq H_0$.\sqed
\end{proof}

\begin{theorem} \label{theo:B_LT:extension to levys char}
 For all $x, a \in \R_+$, $\a > 0$, $f \in b\sB(\R^2)$,
 \begin{align*}
   & \EV_x \Big( \int_0^\infty e^{-\a t} \, f \big( \abs{B_t} - L_t + a, (L_t - a)^+ \big) \, dt \Big) \\
   & = \EV_{x+a} \Big( \int_0^\infty e^{-\a t} \, f \big( \abs{B_t} - L_t, L_t \big) \, dt \Big).
 \end{align*}
\end{theorem}
\begin{proof}
 We decompose both sides of the claimed identity separately via Dynkin's formula~\eqref{eq:Dynkins formula (resolvent)}
 with respect to the stopping times $L^{-1}(a)$ and $\h_0$,
 using the same techniques as in the proof of Lemma~\ref{lem:B_LT:extension to levys char}. 
 By splitting at $\h_0$ and $L^{-1}(a)$, using the additive functional relation
 $L(t + \h_0) = L_t \circ \T_{\h_0} + L(\h_0)$ with $L(\h_0) = 0$,
 and $L^{-1}(a) - \h_0 = L^{-1}(a) \circ \T_{\h_0}$ by Lemma~\ref{lem:B_LT:shift of inverse local time},
 the left-hand side of the above claim reads
 \begin{align*}
   & \EV_x \Big( \int_0^\infty e^{-\a t} \, f \big( \abs{B_t} - L_t + a, (L_t - a)^+ \big) \, dt \Big) & \\
  % & = \EV_x \Big( \int_0^{L^{-1}(a)} e^{-\a t} \, f \big( \abs{B_t} - L_t + a, 0 \big) \, dt \Big) & \\
  % & \quad + \EV_x \big( e^{-\a L^{-1}(a)} \, \EV_0 \Big( \int_0^\infty e^{-\a t} \, f \big( \abs{B_t} - L_t, L_t \big) \, dt \Big) \Big) & \\
   & = \EV_x \Big( \int_0^{\h_0} e^{-\a t} \, f \big( \abs{B_t} + a, 0 \big) \, dt \Big) & \\
   & \quad + \EV_x \Big( e^{-\a \h_0} \, \EV_0 \Big( \int_0^{L^{-1}(a)} e^{-\a t} \, f \big( \abs{B_t} - L_t + a, 0 \big) \, dt \Big) \Big) & \\
   & \quad + \EV_x \big( e^{-\a L^{-1}(a)} \, \EV_0 \Big( \int_0^\infty e^{-\a t} \, f \big( \abs{B_t} - L_t, L_t \big) \, dt \Big) \Big). &
 \end{align*} 
 Decomposition at $\h_a$ and $\h_0$, employing the terminal time property $\h_0 - \h_a = \h_0 \circ \T_{\h_a}$ 
 $\PV_{x+a}$-a.s.\ by the continuity of $B$, transforms the right-hand side to
% \begin{align*}
%   & \EV_{x+a} \Big( \int_0^\infty e^{-\a t} \, f \big( \abs{B_t} - L_t, L_t \big) \, dt \Big) & \\
%   & = \EV_{x+a} \Big( \int_0^{\h_0} e^{-\a t} \, f \big( \abs{B_t}, 0 \big) \, dt \Big) & \\
%   & \quad + \EV_{x+a} \big( e^{-\a \h_0} \, \EV_0 \Big( \int_0^\infty e^{-\a t} \, f \big( \abs{B_t} - L_t, L_t \big) \, dt \Big) \Big). & 
% \end{align*}  
% Another decomposition of the first integral at $\h_a$, employing the terminal time property $\h_0 - \h_a = \h_0 \circ \T_{\h_a}$ 
% $\PV_{x+a}$-a.s.\ by the continuity of $B$,  %implied by $\h_a \leq \h_0$
% gives 
 \begin{align*}
   & \EV_{x+a} \Big( \int_0^\infty e^{-\a t} \, f \big( \abs{B_t} - L_t, L_t \big) \, dt \Big) & \\
   & = \EV_{x+a} \Big( \int_0^{\h_a} e^{-\a t} \, f \big( \abs{B_t}, 0 \big) \, dt \Big) & \\
   & \quad + \EV_{x+a} \Big( e^{-\a \h_a} \, \EV_a \Big( \int_0^{\h_0} e^{-\a t} \, f \big( \abs{B_t}, 0 \big) \, dt \Big) \Big) & \\
   & \quad + \EV_{x+a} \big( e^{-\a \h_0} \, \EV_0 \Big( \int_0^\infty e^{-\a t} \, f \big( \abs{B_t} - L_t, L_t \big) \, dt \Big) \Big). &
 \end{align*}  
 A comparison of the particular summands with the help of \eqref{eq:BM passage time} and Lemmas~\ref{lem:B_LT:laplace of local time inverse}
 and~\ref{lem:B_LT:extension to levys char} yields the result.\sqed
\end{proof}

\section{Revival and Killing Effects on the Generator of a Brownian Motion}\label{app:killing and revival}

In section~\ref{sec:KPS extension}, we construct Brownian motions which admit a finite jump measure 
by applying the revival technique explained in subsection~\ref{subsec:C_CO:identical iterations} 
and subsequent killing by mapping an absorbing set to the cemetery point $\D$.
We examine the effects of these transformations on the generator of a Brownian motion:

\begin{proof}[Proof of Lemma~\ref{lem:G_IM:Feller resolvent of revived X}:]
 We decompose the resolvent at the first revival time $R^1$ with the help of Dynkin's formula~\eqref{eq:Dynkins formula (resolvent)}:
 As the process $X^q$ up to the time $R^1$ equals the original process $X^\bullet$ up to its lifetime $\z$, we have by Theorem~\ref{theo:concatenation countable},
 for any $f \in \cC_0(\cG)$:%, $g \in \cG$:
  \begin{align*}
   \forall g \in \cG: \quad 
   U^q_\a f(g)
   & = \EV_g \Big( \int_0^{R^1} e^{-\a t} \, f(X^q_t) \, dt \Big)
       + \EV_g \big( e^{-\a R^1} \, U_\a f(X^q_{R^1}) \big) \\
   & = \EV_g \Big( \int_0^{\z} e^{-\a t} \, f(X^\bullet_t) \, dt \Big)
       + \EV_g \big( e^{-\a R^1} \, K(U^q_\a f) \big) \\
   & = U^\bullet_\a f(g)
       + \varphi_\a(g) \, q(U^q_\a f),      
  \end{align*}
 with $(U^\bullet_\a, \a > 0)$ being the resolvent of $X^\bullet$, and $\varphi_\a := \EV_{\, \cdot} \big( e^{-\a \z} \big)$.
%   \begin{align*}
%     \varphi_\a(g)
%     & = \EV_g \big( e^{-\a \z} \big) \\
%     & = 1 - \a \, \EV_g \Big( \int_0^\infty e^{-\a t} \, \1_{\cG} (X^\bullet_t) \, dt \Big) \\
%     & = 1 - \a \, U^\bullet_\a \1_{\cG} (g).
%   \end{align*}

 As $X^\bullet$ is Feller and $\varphi_\a \in \cC_0(\cG)$ by assumption, %Theorem~\ref{theo:G_IM:resolvent of X} and Example~\ref{ex:G_WP:dirichlet Walsh BB}, 
 $(U^q_\a, \a > 0)$ preserves $\cC_0(\cG)$ as well.
 Furthermore, $X^q$ is right continuous and normal by definition, so $X^q$ is Feller by~\eqref{eq:Feller (resolvent)}.
 As $X^q_t = X^\bullet_t$ holds for all $t \leq \h_0$ and $X^\bullet$ is a Brownian motion on $\cG$, $X^q$ is also a Brownian motion on $\cG$.
 
 We are ready to compute the boundary conditions for~$X^q$:
 Let $h \in \sD(A^q)$. As~$X^q$ is Feller, there exists an $f \in \cC_0(\cG)$ with $h = U^q_\a f$.
% As $\D \notin \cG$ is isolated, we have $\1_{\cG} \in b\cC(\cG)$, so both
%  \begin{align*}
%   1 - \varphi_\a = \a \, U^\bullet_\a \1_{\cG} % \in \sD(Y^0),
%  \end{align*}
 %i.e. $1 - \varphi_\a$ 
 As $U^\bullet f$ and (by assumption) $1 - \varphi_\a$ fulfill the boundary conditions for $X^\bullet$,
 the above decomposition yields
 % as the proof of Theorem~\ref{theo:G_IM:generator of X} is also applicable to functions in $b\cC(\cG)$.
  \begin{align*}
   & \frac{p_3}{2} \, U^q_\a f''(0)  \\
   & = \frac{p_3}{2} \big( U^\bullet_\a f + \varphi_\a \, q(U^q_\a f) \big)''(0) \\
  % & = \frac{p_3}{2} \, U^\bullet_\a f''(0) - \frac{p_3}{2} \, (1 - \varphi_\a)''(0) \, q(U^q_\a f) \\
   & = - p_1 \, U^\bullet_\a f(0) + \sum_{e \in \cE} p^{e}_2 \, U^\bullet_\a f_e'(0) + \int \big( U^\bullet_\a f(g) - U^\bullet_\a f(0) \big) \, p_4(dg) \\
   & \quad - \Big( - p_1 \big( 1 - \varphi_\a(0) \big) - \sum_{e \in \cE} p^{e}_2 \, ({\varphi_\a})_e'(0) 
    - \int \big( \varphi_\a(g) - \varphi_\a(0) \big) \, p_4(dg) \Big) q(U^q_\a f).
  \end{align*}
 Applying the decomposition of $U^q$ again gives
 \begin{align*}
   & \frac{p_3}{2} \, U^q_\a f''(0)  \\   
   & = - p_1 \, U^q_\a f(0) + \sum_{e \in \cE} p^{e}_2 \, U^q_\a f_e'(0) + \int \big( U^q_\a f(g) - U^q_\a f(0) \big) \, p_4(dg) 
    +  p_1 \, q(U^q_\a f),
  \end{align*}
 and as $q$ is a probability measure, we have
  \begin{align*}
   q(U^q_\a f) = \int \big( U^q_\a f(g) - U^q_\a f(0) \big) \, q(dg) + U_\a f(0),
  \end{align*}
 so it follows that
  \begin{align*}
   \frac{p_3}{2} \, U^q_\a f''(0)
   & = \sum_{e \in \cE} p^{e}_2 \, U^q_\a f_e'(0) + \int \big( U^q_\a f(g) - U^q_\a f(0) \big) \, (p_4 + p_1 \, q)(dg).
  \end{align*}
 Lemma~\ref{lem:G_BM:totality on star graph} completes the proof.\sqed
\end{proof}

\begin{proof}[Proof of Lemma~\ref{lem:G_GC:killing on absorbing set, generator data}:]
 For all $f \in \sD(A^Y)$,
 we have for $g \in \cG \bs F$
  \begin{align*}
   A^X (f \circ \psi) (g)
   & = \lim_{t \downarrow 0} \frac{\EV_g \big( f \circ \psi(X_t) \big) - f \circ \psi(g) }{t} \\
   & = \lim_{t \downarrow 0} \frac{\EV_g \big( f(Y_t) \big) - f(g)}{t},
  \end{align*}
 which exists and is equal to $A^Y f(g)$. 
 On the other hand, if $g \in F$, then $X_t \in F$ holds for all $t \geq 0$, $\PV_g$-a.s., because $F$ is absorbing for $X$, and it follows that
  \begin{align*}
   A^X (f \circ \psi) (g)
   & = \lim_{t \downarrow 0} \frac{\EV_g \big( f \circ \psi(X_t) \big) - f \circ \psi(g) }{t}
     = \lim_{t \downarrow 0} \frac{\EV_g \big( f (\D) \big) - f (\D) }{t}
     = 0.
  \end{align*}
 Thus, we have $f \circ \psi \in \sD(A^X)$ for all $f \in \sD(A^Y)$, and $A^X (f \circ \psi) = A^Y f \, \1_{\comp F}$ in this case.
 %immediately yields:
 % \begin{align*}
 %   f \circ \psi \in \sD(A^X) \quad \Leftrightarrow \quad f \in \sD(A^Y).
 % \end{align*}
 
 So, if $f \in \sD(A^Y)$, then $f \circ \psi$ fulfills the boundary condition for $X$, that is
  \begin{align*}
    %0 
    & p_1 f \big( \psi (0) \big)- \sum_{e \in \cE} p^{e}_2 f_e' \big( \psi (0) \big) + \frac{p_3}{2} f'' \big( \psi (0) \big) - \int_{\cG \bs \{0\}} \big( f \big( \psi (g) \big) - f \big( \psi (0) \big) \big) \, p_4(dg) \\
    & = p_1 f(0) - \sum_{e \in \cE} p^{e}_2 f_e'(0) + \frac{p_3}{2} f''(0) - \int_{\cG \bs (F \cup \{0\})} \big( f(g) - f(0) \big) \, p_4(dg) + f(0) \, p_4(F),
  \end{align*}
 vanishes, where we used $f \big( \psi (g) \big) = f(\D) = 0$ for all $g \in F$.\sqed
\end{proof}

\section{Technical Proofs of Section~\ref{sec:G_IM}}
 
\subsection{On the Path-Behavior of \texorpdfstring{$X$}{X}} \label{app:G_IM:remarks}
In order to ensure that the process $(X_t, t \geq 0)$ of subsection~\ref{subsec:G_IM:definitions} is well-defined, it is necessary that
at any time $t \geq 0$, 
there is at most one $e \in \cE$ with $\eta^e_t > 0$. This will be shown below in Lemma~\ref{lem:G_IM:properties eta}.
To this end, we need to analyze the defining functions
$P_e P^{-1}$, $e \in \cE \cup \{0\}$. The difference between the functions $P_e$, $e \in \cE$, are rather subtle:
If we define the set of all jumps of the subordinator $Q^e$ by $J_e := \{ t > 0: \D Q^e(t) \neq 0 \}$, $e \in \cE$, then 
the set of all jumps reads $J := \biguplus_{e \in \cE} J_e$, as there are no simultaneous jumps.
By definition, $P_e(t) = P(t)$ holds true for all $e \in \cE$ if $t \in \comp J$, whereas for $t \in J$, we have
 \begin{align*}
  P_e(t) = 
   \begin{cases}
    P(t),  & t \in J_e, \\
    P(t-), & t \notin J_e, 
   \end{cases}
 \end{align*}
that is, the function $P_e$ is right continuous at the jumps of $Q^e$, and left continuous with a positive jump discontinuity at the jumps of all other subordinators $Q^f$, $f \neq e$.
We collect these first findings:
\begin{lemma} \label{lem:G_IM:behavior process Pe}
 For every $e \in \cE$, let $J_e = \{ t > 0: \D Q^e(t) \neq 0 \}$ be the set of all jumps of the subordinator $Q^e$, and set $J = \biguplus_{e \in \cE} J_e$.
 Then, for all $e \in \cE$, $t \geq 0$,
 \begin{align*} 
  P_e(t) = 
   \begin{cases}
    P(t),  & t \in J_e \cup \comp J, \\
    P(t-), & t \in J \cap \comp J_e. 
   \end{cases}
 \end{align*}
\end{lemma}

Before we proceed with the analysis of $P_e P^{-1}$, we collect some properties of \textdef{pseudo-inverses} (or \textdef{generalized inverses}).
These results can mostly be found scattered in the literature, for instance in \cite{Embrechts13} and \cite{Feng12}.
We will need results on both the left and right continuous pseudo-inverses for strictly increasing functions, or for increasing and continuous functions.

\begin{definition}
 An increasing function $f \colon [0, \infty) \rightarrow [0,\infty]$ has a \textdef{level of constancy at $t_0 \geq 0$ of length $h > 0$}, if 
 $f(t) = f(t_0+)$ for all $t \in (t_0, t_0 + h)$, $f(t) < f(t_0)$ for all $t < t_0$, and $f(t) > f(t_0+)$ for all $t > t_0 + h$.
\end{definition}

\begin{lemma} \label{lem:G_IM:props pseudoinverse}
 Let $P \colon [0, \infty) \rightarrow [0, \infty)$ be a right continuous, strictly increasing function. Then, the generalized inverse 
  \begin{align*}
   P^{-1} \colon [0, \infty) \rightarrow [0, \infty], \quad t \mapsto P^{-1}(t) := \inf \{ s \geq 0: P(s) > t \}
  \end{align*}
 admits:
 \begin{enumerate}
  \item $P^{-1}$ is right continuous and increasing;                 \label{itm:G_IM:props pseudoinverse:incr,right cont}
  \item for all $t \geq 0$: $P^{-1}P(t) = P^{-1} \big( P(t-) \big) = t$;        \label{itm:G_IM:props pseudoinverse:P-1 P}
  \item for all $t \geq 0$ with $P^{-1}(t) < +\infty$: $P P^{-1} (t) \geq t$; \label{itm:G_IM:props pseudoinverse: P P-1 >= id}
  \item for all $t \in \ran(P)$: $P P^{-1}(t) = t$;                  \label{itm:G_IM:props pseudoinverse:P P-1 = id}
  \item for all $t, u \geq 0$: $P^{-1}(t) \leq u$, if and only if $t \leq P(u)$; \label{itm:G_IM:props pseudoinverse:inversion}
  \item $P^{-1}$ is continuous,                                      \label{itm:G_IM:props pseudoinverse:P-1 cont}
  \item $P$ has a jump at $t > 0$ of height $h$, if and only if $P^{-1}$ has a level of constancy at $P(t-)$ of length $h$. \label{itm:G_IM:props pseudoinverse:jumps-const}
 \end{enumerate}
\end{lemma}

\begin{lemma} \label{lem:G_IM:props pseudoinverse (cont. case)}
 Let $L \colon [0, \infty) \rightarrow [0, \infty)$ be a continuous, increasing function. Then, the generalized inverses 
  \begin{align*}
   L^{-1} \colon [0, \infty) \rightarrow [0, \infty], \quad & t \mapsto L^{-1}(t) := \inf \{ s \geq 0: L(s) > t \}, \\
   L_-^{-1} \colon [0, \infty) \rightarrow [0, \infty], \quad & t \mapsto L_-^{-1}(t) := \inf \{ s \geq 0: L(s) \geq t \}
  \end{align*} 
 admit:
 \begin{enumerate}
  \item $L^{-1}$ is right continuous and increasing;                 \label{itm:G_IM:props pseudoinverse (cont. case):incr,right cont}
  \item for all $t, u \geq 0$: $L^{-1}(t) < u$, if and only if $t < L(u)$; \label{itm:G_IM:props pseudoinverse (cont. case):inversion}
  \item $L_-^{-1}$ is left continuous and increasing;                \label{itm:G_IM:props lc pseudoinverse (cont. case):incr,left cont}
  \item for all $t, u \geq 0$: $L_-^{-1}(t) \leq u$, if and only if $t \leq L(u)$. \label{itm:G_IM:props lc pseudoinverse (cont. case):inversion}
 \end{enumerate}
\end{lemma}

These results give us enough structural properties of generalized inverses to analyze the functions $t \mapsto P_e P^{-1}(t)$, $e \in \cE \cup \{0\}$
in Remark~\ref{rem:G_IM:behavior process X} (recall the definition of~$\sT$ given there).
We are now able to deduce the following:
 
\begin{lemma} \label{lem:G_IM:properties eta}
 For all $t \geq 0$, the following holds true:
 \begin{enumerate}
  \item There is at most one $e \in \cE$ with $\eta^e_t > 0$.
  \item $\eta_t > 0$, if and only if $\eta^e_t > 0$ for exactly one $e \in \cE$.
 \end{enumerate}
\end{lemma}
\begin{proof}
 There are no simultaneous jumps by construction, so all jump times $t_n$, $n \in \N$, are pairwise distinct.
 Thus, the intervals $\big( P(t_n-), P(t_n) \big)$, $n \in \N$,
 are pairwise disjoint and for each $n \in \N$, there is exactly one $e \in \cE$ with $t_n \in J_e$.
 In summary, we have for all $e \in \cE$, $n \in \N$, $t \in \big[l_n^-, l_n^+\big) = \big[ P(t_n-), P(t_n) \big)$,
  \begin{align*}
    P_e P^{-1}(t) - t =
    \begin{cases}
     P(t_n) - t  > 0,  & t_n \in J_e, \\
     P(t_n-) - t \leq 0,  & t_n \notin J_e,
    \end{cases}
 \end{align*}
 and $P_e P^{-1}(t) - t = 0$ for all $t \in \sT$. Therefore, for any $t \geq 0$, there is at most one~$e \in \cE$ with
 $P_e P^{-1}(t) - t > 0$, and in this case $t \in [l_n^-, l_n^+)$ for some $n \in \N$, which is equivalent to $P P^{-1}(t) - t > 0$ 
 by~\eqref{eq:G_IM:behavior process P P-1} and~\eqref{eq:G_IM:behavior process Pe P-1}.\sqed
\end{proof}

The path behavior is now clear: For $L_t \in \sT$, we have ${\eta_t = P P^{-1} (L_t) - L_t = 0}$, 
and $\eta^e_t = 0$ for all $e \in \cE$ by Lemma~\ref{lem:G_IM:properties eta}, so for these times, it is $X_t = W_t$ by definition.
Otherwise, if $L_t \in [l_n^-, l_n^+)$ with $[l_n^-, l_n^+)$ corresponding to a jump $\big(t_n, (e_n, l_n) \big)$, 
we have $\eta^{e_n}_t = \eta_t > 0$, which yields $E(\eta^e_t, e \in \cE) \circ W_t = e_n$ and
 \begin{align*}
  \abs{X_t} = \eta_t + \abs{W_t} = P(t_n) - L_t + \abs{W_t} = l_n^+ - L_t + \abs{W_t}.
 \end{align*}
Therefore, $X_t$ behaves for $t \in L_-^{-1} \big( [l_n^-, l_n^+) \big)$ like a Brownian motion started at ${l_n^+ - l_n^- = l_n}$.
In total, we get equation~\eqref{eq:G_IM:behavior process X}.
 
We complete the study of the paths of~$X$:
\begin{proof}[Proof of Theorem~\ref{theo:G_IM:continuity process X}:]
  As $(W_t, t \geq 0)$ and $(L_t, t \geq 0)$ are continuous, and $L_t$ only grows if $W_t$ is at $0$,
  the edge of $X_t$ only changes at some time $t \geq 0$, if either the edge of $W_t$ changes
  or $L_t$ grows over some $l_n^+$, in which case $l_n^+ - L_t + \abs{W_t} = 0$ holds true.
  Thus, as the second coordinate $(\eta_t + \abs{W_t}, t \geq 0)$ is right continuous,
  and the first coordinate only changes if the radial part is at the origin, the resulting process $(X_t, t \geq 0)$ is right continuous.
  
  $X$ is away from $0$ if either $W$ is or if $L \in  [l_n^-, l_n^+)$ for some $n \in \N$. 
  In both cases the process behaves continuously in the open interior of these times,
  which follows from the representation~\eqref{eq:G_IM:behavior process X} and the continuity of $W$ and $L$.
  For $t \in L_-^{-1} \big( [l_n^-, l_n^+) \big)$, equation~\eqref{eq:G_IM:behavior process X} gives $X_t =  \big( e, l_n^+ - L_t + \abs{W_t} \big)$, 
  thus we have 
   \begin{align*}
    t_0 := \inf \{ s \geq t: X_s = 0 \} =  \inf\{ s \geq 0: L_s \geq l_n^+ \},
   \end{align*}
  and for every sequence $(t_n, n \in \N)$ in $\R_+$ which strictly increases to $t_0$, 
  $(X_{t_n}, n \in \N)$ converges to $(e, 0) = 0$. But as $X$ is right continuous and $\{0\}$ is closed, we have
  $X_{t_0} = 0$, so $X$ is also continuous at $t_0$.\sqed
\end{proof}

\subsection{Shift Operators for \texorpdfstring{$X$}{X}}\label{app:G_IM:shift operators}
  
We are checking that the operators $(\T^X_t, t \geq 0)$, as defined in subsection~\ref{subsec:G_IM:shift operators}, indeed constitute a family of shift operators for $X$:
\begin{proof}[Proof of Lemma~\ref{lem:G_IM:shift operators for X}:]
 Fix $s, t \geq 0$. It is clear that $\T^X_t \colon \O \rightarrow \O$, as $\T^W_t \colon \O^W \rightarrow \O^W$, $\T^Q_t \colon \O^Q \rightarrow \O^Q$, 
 $\varrho_t(\o) \geq 0$ for all $\o \in \O$ and $\g^P_x \colon \O^Q \rightarrow \O^Q$ for all $x \in \R$.
 
 We begin by calculating the shift on the subordinator: For all $u \geq 0$, we have
 \begin{align*}
  P(\g^P_{-L_t} \circ \T^Q_{\varrho_t})^{-1} (u)
  & = \inf \{ s \geq 0: P(\g^P_{-L_t} \circ \T^Q_{\varrho_t})(s) > u \} \\
  & = \inf \{ s \geq 0: P(s + \varrho_t) - L_t > u \} \\
 % & = \big( \inf \{ s \geq 0: P(s) > u + L_t \} - \varrho_t \big) \vee 0 \\
  & = P^{-1}(u + L_t) - \varrho_t.
 \end{align*}
 $(L_t, t \geq 0)$ is an additive functional and $P^{-1}(L_{s+t}) \geq P^{-1}(L_t)$, so
 \begin{equation} \label{eq:G_IM:shifts:shift on P-1(Ls)}
 \begin{aligned}
  P^{-1}(L_s) \circ \T^X_t 
  & = P(\g_{-L_t} \circ \T^Q_{\varrho_t})^{-1} (L_{s+t} - L_t) \\
  %& = \big( P^{-1}(L_{s+t}) - \varrho_t \big)^+ \\
  & = P^{-1}(L_{s+t}) - \varrho_t.
 \end{aligned} 
 \end{equation}
 Let $e \in \cE \cup \{0\}$. Then, by applying the shift $\T^X_t$ and the above findings, we obtain
 \begin{align*}
  \big( P_e P^{-1}(L_s) - L_s \big) \circ \T^X_t 
  & = P_e(\g^P_{-L_t} \circ \T^Q_{\varrho_t}) \big( P^{-1}(L_s) \circ \T^X_t \big) - L_s \circ \T^W_t \\
  & = P_e \big( P^{-1}(L_s) \circ \T^X_t + \varrho_t \big) - L_t - (L_{s+t} - L_t) \\
  & = P_e P^{-1}(L_{s+t}) - L_{s+t}. 
 \end{align*}
 By inserting the last two formulas into the definition of $X$ and additionally using 
 \begin{align*}
  W_s \circ \T^X_t = W_s \circ \T^W_t = W_{s+t}, 
 \end{align*}
 we get $X_s \circ \T^X_t = X_{s+t}$.
 
 It remains to prove $\T^X_s \circ \T^X_t = \T^X_{s+t}$. We calculate for $\o = (\o^W, \o^Q)$
  \begin{align*}
   \T^X_s \big( \T^X_t(\o) \big)
   & = \big( \T^W_s \big( \T^W_t(\o^W) \big), \g^P_{-L_s ( \T_t(\o) )} \big( \T^Q_{\varrho_s(\T_t(\o))} \big( \g^P_{-L_t(\o)} \big( \T^Q_{\varrho_t(\o)}(\o^Q) \big) \big) \big) \\
   & = \big( \T^W_{s+t} (\o^W) , \g^P_{-L_{s+t}(\o) + L_t(\o)} \big( \T^Q_{\varrho_{s+t}(\o) - \varrho_t(\o)} \big( \g^P_{-L_t(\o)} \big( \T^Q_{\varrho_t(\o)}(\o^Q) \big) \big) \big),
  \end{align*}
 where we used the shift property of $(\T^W_t, t \geq 0)$ on themselves and on the additive functional $(L_t, t \geq 0)$, 
 as well as $\varrho_s \circ \T^X_t = P^{-1}(L_{s+t}) - \varrho_t$ by~\eqref{eq:G_IM:shifts:shift on P-1(Ls)}.
 Observing that $(\T^Q_t, t \geq 0)$ and $(\g^P_x, x \in \R)$ commute
 (because the natural shift operators
  $(\hT^{Q,e}_t, t \geq 0)$ and translation operators $(\hg^{Q, e}_q, q \in \R)$ of the Cartesian parts commute, cf.~definitions~\eqref{eq:natural centering and translation}), we get
  \begin{align*}
   & \T^X_s \big( \T^X_t(\o) \big) \\
   & = \big( \T^W_{s+t} (\o^W) , \g^P_{-L_{s+t}(\o) + L_t(\o)} \circ \g^P_{-L_t(\o)} \circ \T^Q_{\varrho_{s+t}(\o) - \varrho_t(\o)} \circ \T^Q_{\varrho_t(\o)}(\o^Q) \big) \\
   & = \big( \T^W_{s+t} (\o^W) , \g^P_{-L_{s+t}(\o)} \big( \T^Q_{\varrho_{s+t}(\o)} (\o^Q) \big) \big) \\
   & = \T^X_{s+t} (\o). \qedhere
  \end{align*}
\end{proof}
  
\subsection{Strong Markov Properties of the Underlying Processes}\label{app:G_IM:strong Markov of (W,Q)}
We give some rigorous context for the results of subsection~\ref{subsec:G_IM:strong Markov of (W,Q)}.
As noted there, the following results are not the canonical Markov properties,
as we will only consider and shift the second part of the combined process $(W,Q)$ here, so everything is still ``independent'' of the first coordinate.
These ``partial'' time shifts are not commonly treated, because joint Markov processes $\big((X_t, Y_t), t \geq 0 \big)$ typically run with a shared 
time parameter $t$ and thus are translated collectively by the same time shift. Therefore we will need to lift the following ``Markov properties'' manually.

$(Q_s, s \geq 0)$ is ``Markovian'' with respect to $(\bsF^Q_s, s \geq 0)$ in the following sense:
\begin{lemma} \label{lem:G_IM:Q Markovian barFQ}
 For all $g \in \cG$, $q \in \R^n$, $f \in b\sB(\R)^{\otimes n}$, $s,t \geq 0$,
  \begin{align*}
    \EV_{g,q} \big( f(Q_{s+t}) \,\big|\, \bsF^Q_s \big) = \EV_{g, Q_s} \big( f(Q_t) \big).
  \end{align*}
\end{lemma}
\begin{proof}
 As $(Q_s, s \geq 0)$ is adapted to $(\bsF^Q_s, s \geq 0)$, it suffices to check that 
 for all $A \in \sF^W_\infty$, $B \in \sF^Q_s$, $g \in \sG$, $q \in \R^n$, $f \in b\sB(\R)^{\otimes n}$, $s,t \geq 0$,
 \begin{align*}
  \EV_{g,q} \big( f(Q_{s+t}) ~ \1_{A \times B} \big) 
  & = \EV_{g,q} \big( \EV_{g, Q_s} \big( f(Q_t) \big)  \, \1_{A \times B} \big),
 \end{align*}
 which follows by separating both components in the product space with the help of Fubini's theorem and applying the Markov property of $(\hQ_s, s \geq 0)$.\sqed
\end{proof}
% \begin{proof}
%  It is obvious that $(Q_s, s \geq 0)$ is adapted to $(\bsF^Q_s, s \geq 0)$.
%  
%  The system $\{ A \times B:  A \in \sF^W_\infty, B \in \sF^Q_s \}$ is an $\cap$-stable generator of $\bsF^Q_s$,
%  so the claim follows from the definition of the combined process on the product space with the help of Fubini's theorem, as we have
%  \begin{align*}
%   \EV_{g,q} \big( f(Q_{s+t}) ~ \1_{A \times B} \big) 
%   & = \EV^W_g \big( \1_A \big) ~ \EV^Q_q \big( f(\hQ_{s+t}) \, \1_B \big) \\
%   & = \EV^W_g \big( \1_A \big) ~ \EV^Q_q \big( \EV^Q_{\hQ_s} \big( f(\hQ_t) \big) \, \1_B \big) \\
%   & = \EV_{g,q} \big( \EV^Q_{Q_s} \big( f(\hQ_t) \big) \, \1_{A \times B} \big) \\
%   & = \EV_{g,q} \big( \EV_{g, Q_s} \big( f(Q_t) \big)  \, \1_{A \times B} \big)
%  \end{align*}
%  for all $A \in \sF^W_\infty$, $B \in \sF^Q_s$, $g \in \sG$, $q \in \R^n$, $f \in b\sB(\R)^{\otimes n}$, $s,t \geq 0$.
% \end{proof}

As $(\hQ_s, s \geq 0)$ is a Feller process, the ``Markov property'' of lemma~\ref{lem:G_IM:Q Markovian barFQ}
yields the ``strong Markov property'' of $(Q_s, s \geq 0)$ with respect to $(\bsF^Q_s, s \geq 0)$ in the following sense:
\begin{lemma} \label{lem:G_IM:Q strongly Markovian barFQ}
  For all $g \in \cG$, $q \in \R^n$, $f \in b\sB(\R)^{\otimes n}$, $s \geq 0$, and every stopping time~$\t$ over~$(\bsF^Q_s, s \geq 0)$,
  \begin{align*}
    \EV_{g,q} \big( f(Q_{s+\t}) \,\big|\, \bsF^Q_\t \big) = \EV_{g, Q_\t} \big( f(Q_s) \big).
  \end{align*}
\end{lemma}
\begin{proof}[Note on the proof]
 It seems difficult to directly transfer the strong Markov property of $\hQ$ to $Q$,
 as an $(\bsF_s, s \geq 0)$-stopping time also randomizes the first coordinate of $(W, Q)$ and, even if the processes are independent,
 it does not appear easy to separate both parts in the random time.
 We recommend to reiterate the standard argument which shows that every Feller process is strongly Markovian (see, e.g., \cite[Section III.8]{RogersWilliams1}),
 and adjust it to the product space setting for $Q$.
\end{proof}

We are now ready to infer the Markov property of the combined process $(W,Q)$ with respect to the shifts $\T^W_t$ and $\T^Q_{\varrho_t}$ 
and to the actual filtration $(\sF_t, t \geq 0)$:
The combined shift operators $\T_t := \T^W_t \otimes \T^Q_{\varrho_t}$, $t \geq 0$, on $\O$ are defined in the intuitive way,
that is, for all $\o = (\o^W, \o^Q) \in \O$, we consider
 \begin{align} \label{eq:G_IM:def shift Theta}
   \T_t(\o) = \T^W_t \otimes \T^Q_{\varrho_t} (\o) = \big( \hT^W_t(\o^W), \hT^Q_{\varrho_t \left((\o^W, \o^Q)\right) }(\o^Q) \big).
 \end{align}

The basic version of the ``Markov property'' for $(W,Q)$ with respect to $(\sF_t, t \geq 0)$
via the just defined combined shift operators $(\T_t, t \geq 0)$ is as follows:
 
\begin{lemma} \label{lem:G_IM:Markov of (W,Q)}
 For all $g \in \cG$, $q \in \R^n$, $f \in b\sB(\R)$, $h \in b\sB(\R)^{\otimes n}$, $r, s, t \geq 0$,
 \begin{align*} 
   \EV_{g,q} \big( f(W_r) \, h(Q_s) \circ \T^W_t \otimes \T^Q_{\varrho_t} \, \big| \, \sF_t \big) = \EV_{W_t, Q_{\varrho_t}} \big( f(W_r) \, h(Q_s) \big).
 \end{align*}
\end{lemma}
\begin{proof}
 By using $\sF_t \subseteq \sF^W_t \otimes \sF^Q_\infty$ together
 with the Markov property of $W$ with respect to $(\sF^W_t \otimes \sF^Q_\infty, t \geq 0)$ (which follows from the product space construction), we obtain
 \begin{align*} 
   & \EV_{g,q} \big( f(W_r) \circ \T^W_t \, h(Q_s) \circ \T^Q_{\varrho_t} \, \big| \, \sF_t \big) \\
   & = \EV_{g,q} \big( \EV_{g,q} \big( f(W_r) \circ \T^W_t \, \big| \,  \sF^W_t \otimes \sF^Q_\infty \big) ~ h(Q_s) \circ \T^Q_{\varrho_t} \, \big| \, \sF_t \big) \\
   & = \EV_{g,q} \big( \EV^W_{W_t} \big( f(\hW_r) \big) ~  h(Q_s) \circ \T^Q_{\varrho_t} \, \big| \, \sF_t \big).
 \end{align*} 
 Employing $\sF_t \subseteq \bsF^Q_{\varrho_t}$ (by Lemma~\ref{lem:G_IM:F subset barFQ_rho}),
 the adaptedness of $W$ to $(\bsF^Q_{\varrho_t}, t \geq 0)$ (by Lemmas~\ref{lem:G_IM:FW subset barFQ_tau} and~\ref{lem:G_IM:rho is barFQ stopping time}),
 as well as the ``strong Markov property'' of $Q$ with respect to $\varrho_t$ (as given in Lemma~\ref{lem:G_IM:Q strongly Markovian barFQ}), we get
 \begin{align*} 
   & \EV_{g,q} \big( f(W_r) \circ \T^W_t \, h(Q_s) \circ \T^Q_{\varrho_t} \, \big| \, \sF_t \big) \\
   & = \EV_{g,q} \big( \EV^W_{W_t} \big( f(\hW_r) \big) ~ \EV_{g,q} \big(h(Q_s) \circ \T^Q_{\varrho_t} \, \big| \, \bsF^Q_{\varrho_t} \big) \, \big| \, \sF_t \big) \\
   & = \EV_{g,q} \big( \EV^W_{W_t} \big( f(\hW_r) \big) ~ \EV^Q_{g,Q_{\varrho_t}} \big( h(Q_s) \big) \, \big| \, \sF_t \big) \\
   & = \EV^W_{W_t} \big( f(\hW_r) \big) ~ \EV^Q_{Q_{\varrho_t}} \big( h(\hQ_s) \big) \\
   & = \EV_{W_t, Q_{\varrho_t}} \big( f(W_r) \, h(Q_s) \big),
 \end{align*} 
 where we also used that $(W,Q)$ is adapted to $(\sF_t, t \geq 0)$ (see~\eqref{eq:G_IM:F_W subset of F} and~\eqref{eq:G_IM:Q rho- is adapted to F}).\sqed
\end{proof}

Complying with the usual generalization of Markovian shifts, we lift the above lemma with the help of the monotone class theorem
to $\sF^0_\infty$-measurable functions by slightly adjusting the routine proof (see, e.g., \cite[Proposition I.8.4]{BlumenthalGetoor69}),
and obtain Theorem~\ref{theo:G_IM:Markov of (W,Q)}.

\subsection{On the Markov Property of \texorpdfstring{$X$}{X}}\label{app:G_IM:shift of excursion times}

The main non-trivial parts of the process~$X$, as defined in section~\ref{sec:G_IM}, are the excursion times $(\eta^e_t, t \geq 0)$, $e \in \cE$. 
In preparation of the proof of the Markov property of $X$, we examine on 
how a shift of these components by the time $t$
relates to the basic shifts of the underlying processes $Q$ and $W$. 
Fix $t \geq 0$, and recall the definitions of $\pP_e$, $\pQ^e$, $\nQ^e$, $\nP$, as given at the beginning of subsection~\ref{subsec:G_IM:Markov property of X}.

\begin{lemma} \label{lem:G_IM:shift for eta I}
 For all $\o \in \O$, $s \geq 0$, $e \in \cE \cup \{0\}$,
 \begin{align*}
  \eta^e_{t+s}(\o)
  & = P_e \big( \big( \pP^{-1} (L_s \circ \T^W_t - \eta_t) \big) (\o) \big) \circ \T^Q_{\varrho_t} (\o) - ( L_s \circ \T^W_t + L_t ) (\o),
 \end{align*}
 and
 \begin{align*}
  \eta^e_{t+s}(\o)
  & = P_e \big( \, \cdot \, , \pP^{-1} \big( \o, (L_s \circ \T^W_t - \eta_t) (\o) \big) \big) \circ \T^Q_{\varrho_t} (\o) - ( L_s \circ \T^W_t + L_t ) (\o).
 \end{align*}
\end{lemma}
\begin{proof}
 As $P$ and $L$ are increasing, $P(u) > L_{t+s}$ and $u \geq 0$ imply $u \geq P^{-1}(L_t)$, so
 \begin{align*}
  P^{-1}(L_{t+s}) - P^{-1}(L_t)
  %& = \inf \big\{ u \geq 0: P(u) > L_{t+s} \big\} - P^{-1}(L_t) \\
  & = \inf \big\{ u \geq P^{-1}(L_t): P(u) > L_{t+s} \big\} - P^{-1}(L_t) \\
  %& = \inf \big\{ u \geq 0: P\big(u + P^{-1}(L_t)\big) > L_{t+s} \big\} \\
  & = \inf \big\{ u \geq 0: \pP(u) + P P^{-1}(L_t) > L_s \circ \T^W_t + L_t \big\} \\
  & = \pP^{-1} ( L_s \circ \T^W_t - \eta_t ).
 \end{align*}
 Therefore, we obtain
 \begin{align*}
  P_e P^{-1}(L_{t+s}) (\o)
  & = P_e \big( \pP^{-1} \big( L_s \circ \T^W_t - \eta_t \big) (\o) + P^{-1}(L_t)(\o) \big) (\o) \\
  & = P_e \big( \big( \pP^{-1} (L_s \circ \T^W_t - \eta_t) \big) (\o) \big) \circ \T^Q_{\varrho_t} (\o). \qedhere
 \end{align*}
\end{proof}

\begin{lemma} \label{lem:G_IM:shift for eta II}
 For all $\o \in \O$, $s \geq 0$, $e \in \cE \cup \{0\}$,
 \begin{align*}
  \eta^e_{t+s}(\o)
  & = \big( P_e \nP^{-1} \big( L_s - \eta_t(\o) \big) - L_s \big) \circ \T^W_t \otimes \T^Q_{\varrho_t} (\o) - L_t(\o).
 \end{align*}  
% \begin{align*}
%  \eta^i_{t+s}(\o)
%  & = (P_i (\, \cdot, \,, \nP^{-1} (\o, (L_s - \eta_t(\o)) - L_s) \circ \T^W_t \otimes \T^Q_{\varrho_t} (\o) - L_t(\o).
% \end{align*}  
\end{lemma}
\begin{proof}
 It is
  \begin{align*}
   & \big( P_e \nP^{-1} \big( L_s - \eta_t(\o) \big) - L_s \big) \circ \T^W_t \otimes \T^Q_{\varrho_t} (\o) - L_t(\o) & \\
   & = P_e \nP^{-1} \big( (L_s \circ \T^W_t - \eta_t) (\o) \big) \circ \T^Q_{\varrho_t} (\o) - ( L_s \circ \T^W_t + L_t ) (\o), &
 \end{align*}
 so with regard to Lemma~\ref{lem:G_IM:shift for eta I}, it suffices to show that for all $v \in \R$,
  \begin{align*}
    P_e \big( \pP^{-1}(\o, v) \big) \circ \T^Q_{\varrho_t}(\o)
    & = P_e \nP^{-1}(v) \circ \T^Q_{\varrho_t}(\o)
  \end{align*}
 holds true: We have $\pP = \nP \circ \T^Q_{\varrho_t}$ by definition, which results in
  \begin{align*}
   \pP^{-1} = \big( \nP \circ \T^Q_{\varrho_t} \big)^{-1} = \nP^{-1} \circ \T^Q_{\varrho_t},
  \end{align*}
 because for all $\o \in \O$, $v \in \R$,
% $s \geq 0$
% \begin{align*}
%  ((\nP(\circ \T^Q_{\varrho_t})^{-1}(\o)) (s)
%  & = \inf \{ u \geq 0: (\nP \circ \T^Q_{\varrho_t})_u (\o) > s \} \\
%  & = \inf \{ u \geq 0: \nP_u(\T^Q_{\varrho_t} (\o)) > s \} \\
%  & = \nP^{-1}(\T^Q_{\varrho_t} (\o))) (s) = ((\nP^{-1} \circ \T^Q_{\varrho_t})(\o)) (s).
% \end{align*} 
 \begin{align*}
  \big( \nP \circ \T^Q_{\varrho_t} \big)^{-1}(\o, v)
  %& = \inf \big\{ u \geq 0: \big( \nP \circ \T^Q_{\varrho_t} \big)(\o, u) > v \big\} \\
  & = \inf \big\{ u \geq 0: \nP \big(\T^Q_{\varrho_t} (\o), u \big) > v \big\} \\
  %& = \nP^{-1} \big( \T^Q_{\varrho_t} (\o), v) \\
  & = \nP^{-1} (\, \cdot \, , v) \circ \T^Q_{\varrho_t}(\o).
 \end{align*}
 This gives us
% \begin{align*}
%  P_i \nP^{-1}(u) \circ \T^Q_{\varrho_t}(\o)
%  & = P_i ( \nP^{-1}_u (\T^Q_{\varrho_t}(\o)) ) (\T^Q_{\varrho_t}(\o)) \\
%  & = P_i ( \pP^{-1}_u (\o) ) (\T^Q_{\varrho_t}(\o)),
% \end{align*}
 \begin{align*}
  P_e \big( \, \cdot \, , \pP^{-1}(\o, v) \big) \circ \T^Q_{\varrho_t}(\o)
  & = P_e \big( \, \cdot \, , \nP^{-1} (\, \cdot \, , v) \circ \T^Q_{\varrho_t}(\o) \big) \circ \T^Q_{\varrho_t}(\o) \\
  & = P_e \nP^{-1}(v) \circ \T^Q_{\varrho_t}(\o),
 \end{align*}
 completing the proof.\sqed
\end{proof}

\begin{proof}[Proof of Theorem~\ref{theo:G_IM:shift for eta}]
 The first identity follows directly from Lemma~\ref{lem:G_IM:shift for eta II}, as insertion of the definitions of $\nP_e$ and $\eta_t$ results in
  \begin{align*}
    \big( \nP_e(\,\cdot\,) + \eta_t(\o) \big) \circ \T^Q_{\varrho_t} (\o)
    & = \big( P_e ( \, \cdot \, + \varrho_t) - P(\varrho_t) + \eta_t \big) (\o) \\
    & = P_e \circ \T^Q_{\varrho_t} (\o) - L_t(\o).
  \end{align*}
  The relation $\nP_e = P_e \circ \G^Q$ implies the second identity of the claim,
  and this expression together with $P^{-1} \circ \g^P_{\eta_t(\o)}(v) = P^{-1} \big(v - \eta_t(\o) \big)$ for all $v \in \R$ yields the last identity, as
  \begin{align*}
  & \big( P_e P^{-1} \big( L_s - \eta_t(\o) \big) - \big( L_s - \eta_t(\o) \big) \big) \circ \big( \id^W \otimes \G^Q \big) \circ \big( \T^W_t \otimes \T^Q_{\varrho_t} \big) (\o)  \\
  & \quad = \big( P_e P^{-1}(L_s) - L_s \big) \circ \big( \id^W \otimes \g^P_{\eta_t(\o)} \big) \circ \big( \id^W \otimes \G^Q \big) \circ \big( \T^W_t \otimes \T^Q_{\varrho_t} \big) (\o).  \qedhere
 \end{align*}
\end{proof}

\begin{proof}[Proof of Corollary~\ref{cor:G_IM:shift on eta}:]
 As $\nP_e$ is strictly increasing and $\nP_e(0) = P_e(0) - P(0) = 0$ holds $\PV_g = \PV_{g,0}$-a.s., 
 we have $P^{-1}(v) = 0$ a.s.\ for every non-positive number $v \leq 0$.
 Thus, if $L_s \circ \T^W_t < \eta_t$, we get from the first identity of Theorem~\ref{theo:G_IM:shift for eta}:
 \begin{align*}
  \eta^e_{t+s}(\o)
  & = \big( \nP_e(0) - \big( L_s - \eta_t(\o) \big) \big) \circ \T^W_t \otimes \T^Q_{\varrho_t} (\o) \\  
  & = \big( P_e(\varrho_t) - P(\varrho_t) - \big( L_s \circ \T^W_t  - \eta_t \big) \big) (\o) \\
  & = \big( \eta^e_t - \eta_t - L_s \circ \T^W_t  + \eta_t \big) (\o),
 \end{align*}
 where we just inserted the definitions of $\nP$ and $\eta_t$ for the last two identities.\sqed
\end{proof}

\end{appendix}

\section*{Acknowledgements}
The main parts of this paper were developed during the author's Ph.D.\ thesis~\cite{Werner16} supervised by Prof.~J\"urgen~Potthoff, whose constant support the 
author gratefully acknowledges.

\bibliographystyle{amsplain}
\bibliography{diss}

\end{document}